\documentclass[english,10pt]{smfart}
\usepackage[english]{babel}
\usepackage[english]{babel}
\usepackage{pgfplots}
\pgfplotsset{compat=1.16}
\usetikzlibrary{external}
\usepackage{caption}
\usepackage{subcaption}
\usepackage{smfthm}

\theoremstyle{plain} %\newtheorem{scholie}{Scholie}

\usepackage{lmodern}
\usepackage{rsfso}  % very beautiful cal letters
\usepackage{amsmath,amssymb,amsthm,amsfonts}
\usepackage{graphicx}
\usepackage[left=2.8cm,right=2.8cm,top=3.0cm,bottom=3.0cm]{geometry}

\usepackage{mathtools,thmtools} %\usepackage{mathtools,thmtools,stmaryrd}
\usepackage{empheq}
\usepackage{enumitem}
\usepackage{lmodern}
\usepackage{esint}
\usepackage{upgreek}
\usepackage{braket}
\usepackage[mathscr]{eucal}
\usepackage{dsfont}
\usepackage{xfrac}

\usepackage[colorlinks=true]{hyperref}
\hypersetup{citecolor=blue,urlcolor=black}

\newcommand*{\haus}{\cal{H}}

\DeclareMathOperator{\Obs}{Obs}

\makeatletter
\newcommand*\bigcdot{\mathpalette\bigcdot@{.7}}
\newcommand*\bigcdot@[2]{\mathbin{\vcenter{\hbox{\scalebox{#2}{$\m@th#1\,\bullet\,$}}}}}
\makeatother

\newcommand*{\obs}{{\textup{obs}}}

\numberwithin{equation}{section}

\title[Observability of the Schrödinger equation with confining potential]{Observability of the Schrödinger equation\\with subquadratic confining potential\\in the Euclidean space\\}
\author{Antoine Prouff}
\address{Université Paris-Saclay, CNRS, Laboratoire de Mathématiques d'Orsay, 91405, Orsay, France.}
\email{antoine.prouff@universite-paris-saclay.fr}

\AtBeginDocument{\renewcommand*{\smfbyname}{}} % remove 'by'

\theoremstyle{definition}

\newcommand*{\N}{\mathbf{N}}
\newcommand*{\Z}{\mathbf{Z}}
\newcommand*{\Q}{\mathbf{Q}}
\newcommand*{\R}{\mathbf{R}}

\newcommand*{\eps}{\varepsilon}

\newcommand*{\cal}[1]{{\mathcal{#1}}}

\renewcommand*{\frak}[1]{{\mathfrak{#1}}}
\renewcommand*{\rm}[1]{{\mathrm{#1}}}
\newcommand*{\ovl}[1]{\overline{#1}}

\newcommand*{\one}{\mathbf{1}}

\newcommand*{\sph}{\mathbf{S}}
\newcommand*{\tor}{\mathbf{T}}

\newcommand*{\e}{e}	%{\mathrm{e}}
\newcommand*{\ii}{i} 	%{\mathrm{i}}
\newcommand*{\dd}{\mathop{}\mathopen{} d}

\newcommand*{\cont}{C}

\newcommand*{\sch}{{\mathcal S}}

\DeclarePairedDelimiterX{\brak}[2]{\langle}{\rangle}{#1, #2}
\DeclarePairedDelimiterX{\inp}[2]{(}{)}{#1, #2}
\DeclarePairedDelimiterX{\comm}[2]{[}{]}{#1, #2}
\DeclarePairedDelimiterX{\poiss}[2]{\{}{\}}{#1, #2}
\DeclarePairedDelimiterX{\liebrak}[2]{[}{]}{#1, #2}
\DeclarePairedDelimiter{\parens}{(}{)}
\DeclarePairedDelimiter\abs{\lvert}{\rvert}
\DeclarePairedDelimiter\norm{\lVert}{\rVert}
\DeclarePairedDelimiter\jap{\langle}{\rangle}

\newcommand*{\comp}{c}

\DeclareMathOperator{\dom}{Dom}

\DeclareMathOperator{\sign}{sign}

\DeclareMathOperator{\supp}{supp}
\DeclareMathOperator{\dist}{dist}

\DeclareMathOperator{\im}{Im}
\newcommand*{\lap}{\Delta} %\bigtriangleup for ugly version
\DeclareMathOperator{\hess}{Hess}
\newcommand*{\sympf}{\sigma}
\DeclareMathOperator{\ran}{Ran}

\DeclareMathOperator{\quantization}{Op}

\newcommand*{\Op}[2][]{{\quantization}_{#1}\hspace{-0.075em}\parens*{#2}}

\newcommand*{\id}{\mathrm{Id}}

\newcommand*{\strongto}[2][]{\xrightarrow[#2]{#1}}

\renewcommand*{\set}[2]{\left\{#1:#2\right\}}

\newcommand*{\moyal}{\mathbin{\#}}

\newlength\oversetwidth
\newlength\underwidth

\begin{document}
\frontmatter
\begin{abstract}
We consider the Schrödinger equation in $\mathbf{R}^d$, $d \ge 1$, with a confining potential growing at most quadratically. Our main theorem characterizes open sets from which observability holds, provided they are sufficiently regular in a certain sense. The observability condition involves the Hamiltonian flow associated with the Schrödinger operator under consideration. It is obtained using semiclassical analysis techniques. It allows to provide with an accurate estimation of the optimal observation time. We illustrate this result with several examples. In the case of two-dimensional harmonic potentials, focusing on conical or rotation-invariant observation sets, we express our observability condition in terms of arithmetical properties of the characteristic frequencies of the oscillator.
\end{abstract}

 \renewcommand{\subjclassname}{\textup{2020} Mathematics Subject Classification}

\thanks{This project originated and was initially developed while visiting Universidad Politécnica de Madrid during the academic year 2020-2021 and the spring 2022. I thank this institution for its hospitality. I also thank Fabricio Macià for introducing me to this problem and for sharing many ideas and material, including unpublished works with Shu Nakamura~\cite{MaciaNakamura:pcomm}. I am also grateful to Matthieu Léautaud for countless discussions and for his helpful comments on a preliminary version of this paper. This work has been partially supported by FMJH's ``Junior scientific visibility program" (France) and grant MTM2017-85934-C3-3-P (MINECO, Spain).}
\maketitle
{
\hypersetup{linkcolor=black}
\tableofcontents
}

\mainmatter

%%%%%%%%%%%%%%%%%%%%%

\section{Introduction and main results}

We are concerned with the observability of the Schrödinger equation with a confining potential in the Euclidean space:
\begin{equation} \label{eq:Schrodinger}
\ii \partial_t \psi = P \psi \,, \qquad P = V(x) - \tfrac{1}{2} \lap \,, \qquad t \in \R, x \in \R^d,
\end{equation}
where $V$ is a real-valued potential, bounded from below. Specific assumptions shall be stated below. The general problem reads as follows: we wonder which measurable sets $\omega \subset \R^d$ and times $T > 0$ satisfy
\begin{equation} \label{eq:obsineq}
\exists C > 0 : \forall u \in L^2(\R^d) \,, \qquad \norm*{u}_{L^2(\R^d)}^2 \le C \int_0^T \norm*{\e^{- \ii t P} u}_{L^2(\omega)}^2 \dd t \,.
\tag*{$\mathrm{Obs}(\omega, T)$}
\end{equation}
When this property $\Obs(\omega, T)$ is true, we say that the Schrödinger equation~\eqref{eq:Schrodinger} is observable from $\omega$ in time $T$, or that $\omega$ observes the Schrödinger equation. The question consists in finding conditions on the pair $(\omega, T)$ ensuring that one can recover a fraction of the mass of the initial data $u$, by observing the solution $\psi(t) = \e^{-\ii t P} u$ of~\eqref{eq:Schrodinger} in $\omega$ during a time $T$. We will often call $\omega$ the observation set and $T$ the observation time. As for the constant $C$ in the inequality, we will refer to it as the observation cost throughout the text. When an observation set $\omega$ is fixed, the infimum of times $T > 0$ such that $\Obs(\omega, T)$ holds is called the optimal observation time, and is denoted by $T_\star = T_\star(\omega)$. It is clear that this so-called observability inequality holds for $\omega = \R^d$ in any time $T > 0$. This is because the propagator solving the Schrödinger equation $\e^{- \ii t P}$ is an isometry  on $L^2(\R^d)$.\footnote{Another consequence of this is that the condition $\Obs(\omega, T)$ is ``open" with respect to $T$: if $\Obs(\omega, T)$ is true with cost $C > 0$, then $\Obs(\omega, T - \eps)$ is true as soon as $\eps < 1/C$. See Lemma~\ref{lem:obsopenintime} in the appendix for a precise statement.} But from the viewpoint of applications, one would like to find the smallest possible observation sets and the corresponding optimal times for which the observability inequality holds.

The observability question for Schrödinger-type equations has been extensively investigated over the past decades, mainly in compact domains of $\R^d$ or compact Riemannian manifolds. See the surveys of Laurent~\cite{Laurent14} or Macià~\cite{Macia:15} for an overview. In a compact Riemannian manifold, Lebeau showed that the so-called Geometric Control Condition (introduced for the wave equation in~\cite{RT:74,BLR:92}) is sufficient to get observability of the Schrödinger equation in any time $T > 0$~\cite{Leb:92}. This means that all billiard trajectories have to enter the observation set in finite time. See for instance Phung~\cite{Phung:01} for later developments in Euclidean domains. However, works by Haraux~\cite{Har:89plaque} and Jaffard~\cite{JaffardPlaques} on the torus show that this condition is not always necessary. Since then, considerable efforts have been made to find the good geometric condition characterizing the observability of the Schrödinger equation, depending on the geometrical context. This question is closely related to that of understanding the concentration or delocalization of Laplace eigenfunctions or quasimodes, which rule the propagation of states through the Schrödinger evolution; see~\cite{BZ:04}. The latter properties are linked to the behavior of the underlying classical dynamics, which is supposed to drive the quantum dynamics at high frequency. In the literature, mainly two different dynamical situations have been investigated. On the one hand, complete integrability, meaning existence of many conserved quantities, usually features symmetries that result in high multiplicity in the spectrum at the quantum level. This allows for possible concentration of eigenfunctions. On the other hand, chaotic systems, epitomized by the geodesic flow of negatively curved Riemannian manifolds, go along with strong instability properties. For instance, quantum ergodicity states that most\footnote{In fact, the situation is more complicated due to the possible existence of a sparse subsequence of eigenmodes concentrating around unstable closed classical trajectories---a phenomenon known as scarring.} Laplace eigenfunctions are delocalized on manifolds with ergodic geodesic flow. Here we collect a non-exhaustive list of references illustrating this diversity of situations. On the torus, observability was investigated by several authors. In addition to~\cite{Har:89plaque,JaffardPlaques}, let us mention Burq and Zworski~\cite{BZ:04,BZ:12,BZ:19}, Bourgain, Burq and Zworski~\cite{BBZ14}, Macià~\cite{MaciaTorus}, as well as Anantharaman and Macià~\cite{AMSurvey,AM:14}. General completely integrable systems were studied by Anantharaman, Fermanian Kammerer and Macià~\cite{AFKM:15}. As for the disk, the question of characterizing open sets from which observability holds was solved by Anantharaman, Léautaud and Macià~\cite{ALM:16cras,ALM:16}. Macià and Rivière thoroughly described what happens on the sphere and on Zoll manifolds~\cite{MR:16,MR:19}. In the negatively curved setting, we refer to Anantharaman~\cite{An:08}, Anantharaman and Rivière~\cite{AR:12}, Eswarathasan and Rivière~\cite{ER:17}, Dyatlov and Jin~\cite{DJ:17}, Jin~\cite{Jin:18} and recently Dyatlov, Jin and Nonnenmacher~\cite{DJN:22}. See also Privat, Trélat and Zuazua~\cite{PTZ:14} in connection with quantum ergodicity.

Recently, there has been a growing interest in the question of observability for the Schrödinger equation in the Euclidean space, for which new difficulties arise due to the presence of infinity in space. Täufer~\cite{Taufer:22} deals with the observability of the free Schrödinger equation in $\R^d$, showing that it is observable from any non-empty periodic open set in any positive time. It relies on the Floquet-Bloch transform and the theory of lacunary Fourier series. It was later generalized by Le Balc'h and Martin~\cite{LBM:23} to the case of periodic measurable observation sets with a periodic $L^\infty$ potential, in dimension $2$. In~\cite{HWW:22}, Huang, Wang and Wang characterize measurable sets for which the Schrödinger equation~\eqref{eq:Schrodinger} is observable, in dimension $d = 1$ when $V(x) = \abs*{x}^{2m}$, $m \in \N$. They prove that, in the case where $m = 1$ (resp. $m \ge 2$), one has observability from $\omega \subset \R$ in some time (resp. in any time) if and only if
\begin{equation} \label{eq:weakthickness}
\liminf_{x \to + \infty} \dfrac{\abs*{\omega \cap [-x, x]}}{\abs*{[-x, x]}} > 0 \,,
\end{equation}
where $\abs*{ \bigcdot }$ is the one-dimensional Lebesgue measure. Such a set is called ``weakly thick". Simultaneously, Martin and Pravda-Starov~\cite{MartinPravda-Starov:21} provided a generalization of this condition in dimension $d$ which turns out to be necessary if $d \ge 1$ and sufficient if $d = 1$ for observability to hold, in the case of the fractional harmonic Schrödinger equation, namely equation~\eqref{eq:Schrodinger} with $P = ( - \lap + \abs*{x}^2 )^s$, where $s \ge 1$. In the particular cases of potentials or operators discussed above, the techniques that are used, mainly relying on abstract harmonic analysis tools, provide very strong results. However, it seems that more general potentials remain out of reach, since the arguments involved require the knowledge of precise spectral estimates on eigenvalues and eigenfunctions, explicit asymptotics and symmetry properties. Moreover, regarding the case of the harmonic oscillator, the existing results focus on the properties of the sets for which observability holds, but given such a set, they do not give a hint of what would be the minimal time for which the observability inequality holds. In fact they provide an upper bound for this optimal time independent of the open set, corresponding to half a period of the classical harmonic oscillator. But it is reasonable to think that this upper bound can be improved taking into account the geometry of the observation set.

\subsection*{Motivations, assumptions and notation}

The present work aims to address the issues discussed above, namely:
\begin{enumerate} [label=\alph*)]
\item find a robust method to prove that the Schrödinger equation is observable from a given set with less restrictions on the dimension or the potential (e.g.\ variations of the harmonic potential like $x \cdot Ax$ where $A$ is a real symmetric positive-definite $d \times d$  matrix, or potentials of the form $\jap*{x}^{2m}$ with $m > 0$ a real number);
\item provide a more accurate upper bound for the optimal observation time depending on the shape of the observation set.
\end{enumerate}
Throughout this work, we make the following assumptions on the potential:
\begin{enonce}{Assumption} \label{assum:V}
The potential $V$ is $C^\infty$ smooth and satisfies
\begin{align}
\exists m > 0, \exists C, r > 0 : \forall \abs*{x} \ge r \qquad
	\dfrac{1}{C} \jap*{x}^{2m} \le V(x) \le C \jap*{x}^{2m} \,, \label{eq:assumgrowthV} \\
\forall \alpha \in \N^d, \exists C_\alpha > 0 : \forall x \in \R^d, \qquad
	\abs*{\partial^\alpha V(x)} \le C_\alpha \jap*{x}^{2m - \abs*{\alpha}} \,. \label{eq:assumgrowthderivativesV}
\end{align}
Unless stated otherwise, we assume that the potential is subquadratic, namely $0 < m \le 1$.
\end{enonce}
Throughout the article, we shall refer to the left-hand side inequality in~\eqref{eq:assumgrowthV} by saying that the potential is \emph{elliptic}. In addition, the notion of \emph{principal symbol} that we will use is made clear below.

\begin{defi}[Principal symbol] \label{def:principalsymbol}
Let $V_0$ and $V$ be two potentials satisfying Assumption~\ref{assum:V} above with a power $m > 0$. We say that $V_0$ and $V$ have the same principal symbol if
\begin{equation*}
\forall \alpha \in \N^d, \exists C_\alpha > 0 : \forall x \in \R^d, \qquad
	\abs*{\partial^\alpha (V - V_0)(x)} \le C_\alpha \jap*{x}^{2m - 1 - \abs*{\alpha}} \,.
\end{equation*}
This defines an equivalence relation. The equivalence class of such a potential $V$ is called the principal symbol of $V$.
\end{defi}

Classical spectral theory arguments ensure that the operator $V(x) - \tfrac{1}{2} \lap$ with domain $C_c^\infty(\R^d)$ is essentially self-adjoint (from now on, its closure will be denoted by $P$) and that the evolution problem~\eqref{eq:Schrodinger} on $L^2(\R^d)$ is well-posed. In fact, most of our results will depend only on the principal symbol of $V$, namely they will not depend on perturbations of the potential of order $\jap{x}^{2m - 1}$.

Our strategy emphasizes the role of the underlying classical dynamics ruling the evolution of high-energy solutions to the Schrödinger equation~\eqref{eq:Schrodinger}, by means of the so-called quantum-classical correspondence principle. This motivates the introduction of the symbol of the operator $P$, defined by
\begin{equation*}
p(x, \xi)
	:= V(x) + \frac{\abs*{\xi}^2}{2} \,, \qquad (x, \xi) \in \R^{2d} \,.
\end{equation*}
This is a smooth function on the phase space $\R^{2d} \simeq \R_x^d \times \R_\xi^d$, tending to $+ \infty$ as $(x, \xi) \to \infty$, since the potential is elliptic. Throughout this text, typical phase space points will be denoted by $\rho = (x, \xi)$, and we will sometimes use the notation $\pi : \R^{2d} \to \R^d$ for the projection $(x, \xi) \mapsto x$. We will often refer to $p$ as the classical Hamiltonian, and to its quantization $P$ as the quantum Hamiltonian.  The Hamiltonian flow $(\phi^t)_{t \in \R}$ on $\R^{2d}$, which preserves $p$, is defined as the flow generated by the Hamilton equation:
\begin{equation} \label{eq:defHamiltonianflow}
\dfrac{\dd}{\dd t} \phi^t(\rho) = J \nabla p\left(\phi^t(\rho)\right) \,, \qquad \phi^0(\rho) = \rho \,.
\end{equation}
It is well-defined for all times under our assumptions. Here $J = \begin{pmatrix} 0 & I_d \\ - I_d & 0 \end{pmatrix}$ is the symplectic matrix. Introducing $(x^t, \xi^t) = \phi^t(\rho)$ the position and momentum components of the flow, this can be rewritten as
\begin{equation} \label{eq:defHamiltonianflowcoord}
\left\{
\begin{aligned}
\dfrac{\dd}{\dd t} x^t &= \xi^t \\
\dfrac{\dd}{\dd t}\xi^t &= - \nabla V(x^t)
\end{aligned}
\right. \,, \qquad (x^0, \xi^0) = \rho \,.
\end{equation}
In the sequel, we will refer to the $x$-component of a trajectory of the Hamiltonian flow as a \emph{projected trajectory}.

\subsection{Main result}

Let us insist on the fact that the result below applies for confining potentials having a subquadratic growth, i.e.\ $0 < m \le 1$. We will explain later why we restrict ourselves to this case. Throughout the article, the open ball of radius $r$ centered at $x \in \R^d$ is denoted by $B_r(x)$. Our main result reads as follows.

\begin{theo} \label{thm:main}
Let $V_0$ and $V$ be potentials on $\R^d$ satisfying Assumption~\ref{assum:V} with some $m \in (0, 1]$, having the same principal symbol. Set $P = V(x) - \tfrac{1}{2} \Delta$ and denote by $\e^{- \ii t P}$ the propagator solving the Schrödinger equation
\begin{equation*}
\ii \partial_t \psi = P \psi \,.
\end{equation*}
Also denote by $(\phi_0^t)_{t \in \R}$ the Hamiltonian flow associated with the symbol $p_0(x, \xi) = V_0(x) + \frac{1}{2} \abs{\xi}^2$. For any Borel set $\omega \subset \R^d$, define for any $R > 0$ the thickened set
\begin{equation*}
\omega_R = \bigcup_{x \in \omega} B_R(x) \,,
\end{equation*}
and introduce for any $T > 0$ the classical quantity\footnote{The integral makes sense when $\omega$ is Borel. Indeed, the map $(t, \rho) \mapsto \one_{\omega \times \R^d}(\phi_0^t(\rho))$ is then Lebesgue-measurable, so that the same is true for $t \mapsto \one_{\omega \times \R^d}(\phi_0^t(\rho))$ when $\rho$ is fixed. Tonelli's theorem~\cite[Theorem 4.2.5]{Lerner:bookintegration} then shows that the map $\rho \mapsto \int_0^T \one_{\omega \times \R^d}(\phi_0^t(\rho)) \dd t$ is Lebesgue-measurable.}
\begin{equation*}
\frak{K}_{p_0}^\infty(\omega, T)
	= \liminf_{\rho \to \infty} \int_0^T \one_{\omega \times \R^d}\left(\phi_0^t(\rho)\right) \dd t
	= \liminf_{\rho \to \infty} \abs*{\set{t \in (0, T)}{(\pi \circ \phi_0^t)(\rho) \in \omega}} \,.
\end{equation*}
Fix a Borel set $\omega \subset \R^d$. Then the following two assertions hold:
\begin{enumerate}[label=(\roman*)]
\item (Sufficient condition) Assume there exists $T_0 > 0$ such that
\begin{equation} \label{eq:dynamicalcondition}
\frak{K}_{p_0}^\infty := \frak{K}_{p_0}^\infty(\omega, T_0) > 0 \,.
\end{equation}
Then there exists a constant $L = L(d, T_0, p_0, p) > 0$ such that for $R = \frac{L}{\frak{K}_{p_0}^\infty}$, for any compact set $K \subset \R^d$ and any $T > T_0$, $\Obs(\omega_R \setminus K, T)$ is true, namely:
\begin{equation*}
\exists C > 0 : \forall u \in L^2(\R^d) \,, \qquad
    \norm*{u}_{L^2(\R^d)}^2 \le C \int_0^T \norm*{\e^{- \ii t P} u}_{L^2(\omega_R \setminus K)}^2 \dd t \,.
\end{equation*}
\item (Necessary condition) Assume there exists a time $T > 0$ such that $\Obs(\omega, T)$ is true with cost $C_\obs > 0$, that is to say
\begin{equation} \label{eq:obscost}
    \forall u \in L^2(\R^d) \,, \qquad
    \norm*{u}_{L^2(\R^d)}^2 \le C_\obs \int_0^T \norm*{\e^{- \ii t P} u}_{L^2(\omega)}^2 \dd t \,.
\end{equation}
Then there exists a constant $c = c(d, T, p_0, p)$ such that for any $R \ge 1$ and any compact set $K \subset \R^d$, it holds:
\begin{equation*}
\frak{K}_{p_0}^\infty(\omega_R \setminus K, T)
	\ge \dfrac{1}{C_\obs} - c \dfrac{\jap{\log R}^{1/2}}{R} \,.
\end{equation*}
\end{enumerate}
\end{theo}

The rest of the introduction is organized as follows: in Subsection~\ref{subsec:ideaofproof}, we comment on Theorem~\ref{thm:main} and describe the main ideas of the proof. Then we discuss various examples of application. We begin with examples in dimension~$1$ in Subsection~\ref{subsec:ex1D}. In Subsection~\ref{subsec:ex2D}, we investigate the particular case of harmonic oscillators in two dimensions. We specifically focus on conical and rotation-invariant observation sets in Subsections~\ref{subsub:conicalsets} and~\ref{subsub:sphericalsets} respectively. These are cases where one can prove accurate estimates on the optimal observation time---see for instance Proposition~\ref{prop:twocones}. Arithmetical properties of the characteristic frequencies of the harmonic oscillator under consideration also play a key role, as evidenced by Proposition~\ref{prop:anisotropicsphericalsets}. Then in Subsection~\ref{subsec:exother}, we present other consequences of Theorem~\ref{thm:main} concerning observability of eigenfunctions of the Schrödinger operator $P$ and energy decay of the damped wave equation. Lastly, we discuss the links between our work and the Kato smoothing effect in Subsection~\ref{subsec:Kato}, and provide with further explanations regarding the natural semiclassical scaling of the problem and the criticality of quadratic potentials in Subsection~\ref{subsec:semiclassicalscaling}.

\subsection{Idea of proof and comments} \label{subsec:ideaofproof}

The core of our work consists in establishing a suitable version of Egorov's theorem to relate the evolution through the Schrödinger flow of high-energy initial data on the quantum side, to the action of the associated Hamiltonian flow on the classical side. This is done using semiclassical analysis. To apply this theory, we approximate the indicator function of $\omega$ by a smooth and sufficiently flat cut-off function. This is how the larger set $\omega_R$ arises. Although Theorem~\ref{thm:main} is not a complete characterization of sets for which observability holds, it provides an almost necessary and sufficient condition of observability, up to thickening the observation set, and it gives sharp results in many concrete situations. See the examples given in Subsections~\ref{subsec:ex1D}, \ref{subsec:ex2D} and~\ref{subsec:exother} below. We review remarkable features of this statement.

\begin{itemize}[label=\textbullet]
    \item The observability condition~\eqref{eq:dynamicalcondition} we find is reminiscent of the Geometric Control Condition that rules the observability or control of the wave equation in a number of geometrical contexts, especially compact Riemannian manifolds~\cite{RT:74,BLR:88,BLR:92}. It reflects the importance of the quantum-classical correspondence in this problem: high-energy solutions to the Schrödinger equation, lifted to phase space, propagate along the trajectories of the Hamiltonian flow. Our constant $\frak{K}_{p_0}^\infty(\omega, T)$ is to some extent different from the one quantifying the Geometric Control Condition for the wave equation (see the constant $C(t)$ of Lebeau~\cite{Leb:96} or the constant $\frak{K}(T)$ of Laurent and Léautaud~\cite{LL:16}). Indeed, the latter constant consists in averaging some function (typically the indicator function of $\omega$) along speed-one geodesics in a time interval $[0, T]$. In contrast, our constant $\frak{K}_{p_0}^\infty(\omega, T)$ does the same, except that the length of trajectories tends to infinity as their initial datum $\rho$ goes to infinity in phase space. This is consistent with the infinite speed of propagation of singularities for the Schrödinger equation.
	\item Let us insist on the fact that the Schrödinger equation~\eqref{eq:Schrodinger} does not contain any semiclassical parameter. Instead, we artificially introduce a semiclassical parameter $R \to + \infty$, which we use to enlarge the observation set. This is natural in view of the fact that remainders in the quantum-classical correspondence are expressed in terms of derivatives of the symbol under consideration: scaling these symbols by $\frac{1}{R}$ thus produces remainders of the same order.
    \item On the technical side, the non-compactness of the Euclidean space yields new difficulties. In our problem, the use of semiclassical defect measures seems to be limited to very particular geometries of the observation set: roughly speaking, only homogeneous symbols can be paired with such measures, which would theoretically restrict the scope of the result to conical observation sets. Instead, we use (and prove) a version of Egorov's theorem to study the operator $\e^{\ii t P} \one_\omega \e^{- \ii t P}$. The idea of using Egorov's theorem was introduced in control theory by Dehman, Lebeau~\cite{DL:09} and Laurent, Léautaud~\cite{LL:16}. Of course, we must pay a particular attention to the remainder terms, in connection with the non-compactness of the ambient space. The great advantage of this is that we can describe the evolution of a fairly large class of symbols on the phase space, which in turn allows to study observability for a variety of observation sets.
    \item Our result is very robust since it is valid for a fairly large class of potentials, with the noteworthy property that the statement only involves the principal symbol of the potential. Indeed, up to enlarging the parameter $R$, the fact that the dynamical condition~\eqref{eq:dynamicalcondition} is fulfilled or not in $\omega_R$ is independent of the representative of the equivalence class of $V_0$ (introduced in Definition~\ref{def:principalsymbol}) chosen to compute $\frak{K}_p^\infty(\omega_R, T)$. This is a consequence of Corollary~\ref{cor:classicalobssubprincipalperturbation}. This was already evidenced in the context of propagation of singularities for solutions to the perturbed harmonic Schrödinger equation; see Mao and Nakamura~\cite{MaoNakamura:09}. The stability under subprincipal perturbation of the potential fails to be true if one considers superquadratic potentials ($m > 1$), as we can see by the examination of the trajectories of the flow. Take $V_0$ satisfying Assumptions~\ref{assum:V} for some $m > 1$, and perturb this potential with some $W$ behaving like $\jap*{x}^{2m - 1}$. Consider the Hamiltonian flow associated with the potential $V = V_0 + W$. Then the second derivative of a trajectory of the classical flow is given by
    \begin{equation*}
    \dfrac{\dd^2}{\dd x^2} x^t = - \nabla V_0(x^t) - \nabla W(x^t) \,.
    \end{equation*}
    We remark that the perturbation is of order $\nabla W(x^t) \approx \jap*{x^t}^{2 (m - 1)}$, which may blow up when $x^t$ is large. When $m \le 1$, the perturbation of the trajectory remains bounded, and can therefore be absorbed by thickening the observation set. See Subsection~\ref{subsec:invariancesubprincipalperturbation} and the proof of Theorem~\ref{thm:main} at the end of Section~\ref{sec:proofmain} for further details.
    \item At the level of the Hamiltonian flow, the difference between $m \le 1$ and $m > 1$ can also be understood by looking at the equation solved by the differential of the flow: differentiating the Hamilton equation~\eqref{eq:defHamiltonianflow} yields
    \begin{equation*}
    \dfrac{\dd}{\dd t} \dd \phi^t(\rho) = J \hess p\left(\phi^t(\rho)\right) \dd \phi^t(\rho) \,.
    \end{equation*}
    We deduce that the differential of the flow behaves as
    \begin{equation*}
    \abs*{\dd \phi^t} \lesssim \e^{t \abs{\hess p}} \,,
    \end{equation*}
    which means that the norm of the Hessian of the Hamiltonian plays the role of a local Lyapunov exponent for the classical dynamics. Yet $\hess p$ is uniformly bounded on phase space if and only if $m \le 1$. Incidentally, it is likely that for $m < 1$, one can exploit the decay of $\hess p$ at infinity in the space variable in order to get small remainders in the proof of Egorov's theorem (see Proposition~\ref{prop:Egorovthm}) instead of taking $R$ large. This might allow to thicken $\omega$ by any positive $\eps$ rather than by a large parameter $R$. Since we are mostly interested in quadratic potentials in this work, we chose not to refine our result in this direction.
    \item It is possible that the necessary condition can be slightly improved by propagating coherent states rather than using Egorov's theorem on quantum observables. This is discussed in more details in Subsection~\ref{subsub:refinementconical}.
\end{itemize}

\subsection{Examples in dimension \texorpdfstring{$1$}{1}} \label{subsec:ex1D}

The one-dimensional case gives an insight of how the potential can influence the geometry of sets for which observability holds.

\subsubsection{Harmonic potential}

The one-dimensional harmonic oscillator corresponds to $V(x) = \frac{1}{2} x^2$. The Hamiltonian flow reads:
    \begin{equation*}
        \phi^t(x, \xi) = \left(x \cos t + \xi \sin t, - x \sin t + \xi \cos t\right) \,, \qquad (x, \xi) \in \R^2, t \in \R \,.
    \end{equation*}
    Our dynamical condition~\eqref{eq:dynamicalcondition} can then be written as
    \begin{equation*}
        \liminf_{(x, \xi) \to \infty} \int_0^T \one_\omega(x \cos t + \xi \sin t) \dd t > 0 \,.
    \end{equation*}
    In view of the periodicity of the flow, it is relevant to consider $T = 2 \pi$. Under this additional assumption, condition~\eqref{eq:dynamicalcondition} reduces to
    \begin{equation} \label{eq:frakK-HWW}
        {\frak K}^\infty
        	:= \liminf_{A \to \infty} \int_0^{2 \pi} \one_\omega(A \sin t) \dd t > 0 \,,
    \end{equation}
    where $A$ has to be thought as (the square-root of) the energy $p(x, \xi) = \frac{1}{2} (x^2 + \xi^2)$. We claim that this is equivalent to the weak thickness~\eqref{eq:weakthickness} condition of Huang, Wang and Wang~\cite{HWW:22}. Suppose that $\frak{K}^\infty > 0$. First, notice that
    \begin{equation*}
        \int_0^{2 \pi} \one_\omega(A \sin t) \dd t
            = 2 \int_{- \pi/2}^{\pi/2} \one_\omega(A \sin t) \dd t \,.
    \end{equation*}
    Second, fix $c \in (0, \frak{K}^\infty/2)$. Since the integrand is bounded by $1$, we can slightly reduce the time interval to $[- \frac{\pi}{2} + \frac{c}{3}, \frac{\pi}{2} - \frac{c}{3}]$ so that $y = A \sin t$ defines a proper change of variables:
    \begin{align*}
        \frac{c}{3}
            &\le \liminf_{A \to \infty} \int_{- \pi/2}^{\pi/2} \one_\omega(A \sin t) \dd t - \frac{2}{3} c
            \le \liminf_{A \to \infty} \int_{- \frac{\pi}{2} + \frac{c}{3}}^{\frac{\pi}{2} - \frac{c}{3}} \one_\omega(A \sin t) \dd t \\
            &\le \liminf_{A \to \infty} \int_{- \frac{\pi}{2} + \frac{c}{3}}^{\frac{\pi}{2} - \frac{c}{3}} \one_\omega(A \sin t) \dfrac{A \abs{\cos t}}{A \frac{2}{\pi} \times \frac{c}{3}} \dd t
            = \dfrac{3 \pi}{2 c} \liminf_{A \to \infty} \dfrac{1}{A} \int_{- A \sin(\frac{\pi}{2} - \frac{c}{3})}^{A \sin(\frac{\pi}{2} - \frac{c}{3})} \one_\omega(y) \dd y \,.
    \end{align*}
    We used the concavity inequality $\cos t \ge 1 - \frac{2}{\pi} \abs{t}$ on $[- \frac{\pi}{2}, \frac{\pi}{2}]$ to get the third inequality. This gives
    \begin{equation*}
        \liminf_{A \to \infty} \dfrac{\abs*{\omega \cap [- A, A]}}{\abs*{[- A, A]}} > 0 \,,
    \end{equation*}
    namely $\omega$ is weakly thick. Conversely, we can follow the same lines, using that the Jacobian $\abs{\cos t}$ is less than $1$, to show that any weakly thick set satisfies~\eqref{eq:frakK-HWW}. Although our main theorem allows to conclude that observability is true only on a slightly larger set, it is more precise than the previous result from~\cite{HWW:22} with respect to the optimal observation time: we can estimate this optimal time depending on the geometry of the observation set. In addition, our result is stable under subprincipal perturbation of the potential. In particular, weak thickness of $\omega$ implies observability from $\omega_R$ (for some $R$ given by Theorem~\ref{thm:main}) for any potential whose principal symbol is $\tfrac{1}{2} x^2$ (or any positive multiple of $x^2$). Anticipating on the next paragraph, observe that a weakly thick set can contain arbitrarily large gaps, hence is not necessarily thick (see~\cite[Example 4.12]{HWW:22}).

\subsubsection{Potentials having critical points}

An interesting phenomenon appears when the potential possesses a sequence of critical points going to infinity. To construct such a potential, we proceed as follows. We set
    \begin{equation} \label{eq:potentialwithcriticalpoints}
        V(x) = \bigl( 2 + \sin\left(a \log \jap*{x}\right) \bigr) x^2 \,, \qquad x \in \R \,,
    \end{equation}
    where $a$ is a positive parameter to be chosen properly. See Figure~\ref{subfig:potentialwithcriticalpoints} for an illustration.
    
\begin{figure}
		\centering
		\begin{subfigure}{0.45\textwidth}
		\includegraphics[scale=1.0]{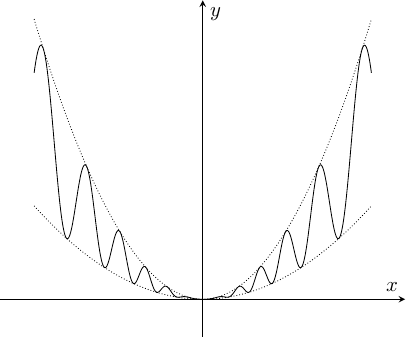}
		\caption{A potential $V$ of the form~\eqref{eq:potentialwithcriticalpoints}. The dotted lines correspond to the potentials $x^2$ and $3 x^2$.}
		\label{subfig:potentialwithcriticalpoints}
		\end{subfigure}
		\hfill
		\begin{subfigure}{0.45\textwidth}
		\includegraphics[scale=1.0]{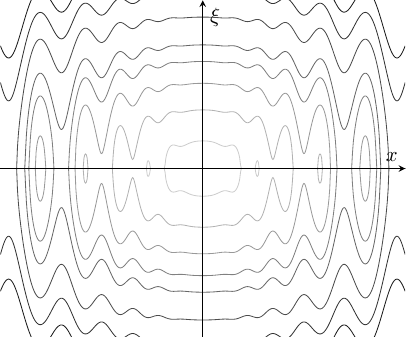}
        \caption{Some level sets of the Hamiltonian $p(x, \xi) = V(x) + \tfrac{1}{2} \abs{\xi}^2$. The corresponding picture for the harmonic potential is just a collection of concentric ellipses.}
        \label{subfig:hamiltonianwithcriticalpoints}
        \end{subfigure}
        \caption{Case of a potential with critical points.}
        \label{fig:potentialwithcriticalpoints}
\end{figure}
    
    One can check that Assumption~\ref{assum:V} is fulfilled: $V$ is subquadratic, elliptic (bounded from below by $x^2$) and each derivative yields a gain of $\jap{x}^{-1}$. Notice however that this is not a subprincipal perturbation of the harmonic potential. It holds for any $x \in \R$:
    \begin{align*}
        V'(x)
            &= \dfrac{x}{\jap{x}^2} \left( 2 \jap{x}^2 \left(2 + \sin\left(a \log \jap{x}\right)\right) + a x^2 \cos \left(a \log \jap{x}\right) \right) \\
            &= \dfrac{x}{\jap{x}^2} \left( \left(4 + 2 \sin\left(a \log \jap{x}\right)\right) + x^2 \left(4 + 2 \sin\left(a \log \jap{x}\right)\right) + a \cos \left(a \log \jap{x}\right)\right) \,.
    \end{align*}
    Factorizing the last two terms, we can write for a certain angle $\varphi_a$:
    \begin{align*}
        V'(x)
            &= \dfrac{x}{\jap{x}^2} \left( \left(4 + 2 \sin\left(a \log \jap{x}\right)\right) + x^2 \left(4 + \sqrt{4 + a^2} \sin\left(\varphi_a + a \log \jap{x}\right)\right)\right) \\
            &= \dfrac{x}{\jap{x}^2} \left( \left(4 + 2 \sin\left(a \log \jap{x}\right)\right) + 4 x^2 \left(1 + \sqrt{\dfrac{1}{4} + \left(\dfrac{a}{4}\right)^2} \sin\left(\varphi_a + a \log \jap{x}\right)\right)\right) \,.
    \end{align*}
    When $\frac{1}{4} + (\frac{a}{4})^2 > 1$, which is true if and only if $a > 2 \sqrt{3}$, we can find two sequences $(x_n^+)_{n \in \N}$ and $(x_n^-)_{n \in \N}$ tending to infinity such that
    \begin{equation*}
    \left\{
    \begin{aligned}
        \sqrt{\dfrac{1}{4} + \left(\dfrac{a}{4}\right)^2} \sin\left(\varphi_a + a \log \jap{x_n^+}\right) &\ge - 1 + \eta \\
        \sqrt{\dfrac{1}{4} + \left(\dfrac{a}{4}\right)^2} \sin\left(\varphi_a + a \log \jap{x_n^-}\right) &\le - 1 - \eta
     \end{aligned}
     \right. \,,
    \end{equation*}
    for some sufficiently small $\eta > 0$. The intermediate value theorem then implies that there exist infinitely many points $x_n^0$, with $\abs{x_n^0}$ tending to infinity, where $V'(x_n^0) = 0$. Now we observe from~\eqref{eq:defHamiltonianflowcoord} that the trajectories of the Hamiltonian flow with initial data $\rho_n = (x_n^0, 0)$ are stationary, that is
    \begin{equation*}
        \phi^t(\rho_n) = \rho_n \,, \qquad \forall t \in \R \,.
    \end{equation*}
    We deduce the following: assume that the Schrödinger equation~\eqref{eq:Schrodinger} is observable from $\omega \subset \R$ in some time for this potential. Then, the necessary condition of Theorem~\ref{thm:main} tells us that there exists $R > 0$ such that for any $n$ large enough, $x_n^0 \in \omega_R$. We can rephrase this as
    \begin{equation} \label{eq:observationsetwhentherearecriticalpoints}
    \exists n_0 \in \N : \forall n \ge n_0 \,, \qquad
    	\omega \cap B_R(x_n^0) \neq \varnothing \,.
    \end{equation}
    This is consistent with the phase portrait depicted in Figure~\ref{subfig:hamiltonianwithcriticalpoints}: some energy might be trapped around small closed trajectories encircling stable critical points. Hence, in order to have observability, $\omega$ cannot be too far away from those points. In fact, one observes that~\eqref{eq:observationsetwhentherearecriticalpoints} concerns all critical points, whatever the sign of $V''(x_n^0)$.

    In conclusion, the situation of a potential of the form~\eqref{eq:potentialwithcriticalpoints} is in contrast with the previous case of the harmonic potential $\frac{1}{2} x^2$ where the weak thickness condition allowed for large gaps around any sequence of points $x_n \to \infty$ satisfying $\abs{x_{n+1}} \gg \abs{x_n}$. Notice that $\omega$ can still have large gaps away from critical points though.

\subsubsection{Sublinear potentials}

Our last remark in the one-dimensional case concerns potentials having a sublinear growth, namely $m \in (0, 1/2]$. In this situation, the trajectories of the Hamiltonian flow whose initial datum has purely potential energy (namely $\xi = 0$) do not escape far away from their initial location. This is because:
    \begin{equation*}
        \dfrac{\dd}{\dd t} \xi^t  = - V'(x^t) = O\left(\jap*{x^t}^{2m - 1}\right) \,,
    \end{equation*}
    which remains bounded uniformly as soon as $m \le 1/2$. For the same reason, $m=1/2$ also appears to be critical in Proposition~\ref{prop:timeincylinder}. If observability from $\omega \subset \R$ holds in some time for such a potential, the necessary condition of Theorem~\ref{thm:main} leads to the conclusion that $\omega$ has to intersect any interval of length $2 R$, for some $R > 0$. Likewise, in higher dimension, any set from which the Schrödinger equation is observable must satisfy
	\begin{equation} \label{eq:notholes}
	\exists R > 0 : \forall x \in \R^d \,, \qquad
		\omega \cap B_R(x) \neq \varnothing \,.
	\end{equation} 
	Therefore, sets observing the Schrödinger equation~\eqref{eq:Schrodinger} for a sublinear potential cannot have arbitrarily large holes.\footnote{Notice that~\eqref{eq:notholes} is much weaker that the usual thickness condition of control theory:
    \begin{equation*}
    \exists R, c > 0 : \forall x \in \R^d \,, \qquad
    	\abs*{\omega \cap B_R(x)} \ge c \abs*{B_R(x)} \,.
    \end{equation*}}
    Although the case of bounded potentials (i.e.\ $m = 0$) is not in the scope of this article, let us mention that this observation is consistent with recent results on the free Schrödinger equation. See Huang, Wang, Wang~\cite{HWW:22} and Täufer~\cite{Taufer:22}, as well as Le Balc'h and Martin~\cite{LBM:23} for the case of bounded periodic potentials in two dimensions.

\subsection{Observability of two-dimensional harmonic oscillators} \label{subsec:ex2D}
%%%%%%%%%%%%%%%%%%%%%%%%%%%%%%%%%%

As an application of Theorem~\ref{thm:main}, we study the observability of harmonic oscillators in conical or rotation-invariant sets. Our results mainly concern the two-dimensional case. The examples presented in this subsection suggest that there is no general reformulation of our dynamical condition~\eqref{eq:dynamicalcondition} in purely geometrical terms. That is to say, it seems difficult to find an equivalent condition that would not involve the Hamiltonian flow (e.g.\ thickness, weak thickness...). In contrast, by restricting ourselves to a certain class of potentials (harmonic oscillators at the principal level here) and a certain class of observation sets (conical or rotation-invariant), one can indeed transform the dynamical condition into a geometrical one. Along the way, we will see that observability properties are very sensitive to slight modifications of the coefficients of the harmonic oscillator under consideration. This subsection culminates in Proposition~\ref{prop:anisotropicsphericalsets}, where we show that observability of rotation-invariant sets is governed by Diophantine properties of the oscillator's coefficients.

Let us first recall basics about general harmonic oscillators. Let $A$ be a real symmetric positive-definite $d \times d$ matrix and set $H_A = \tfrac{1}{2} (x \cdot A x - \Delta)$. Up to an orthonormal change of coordinates, one can assume that $A$ is diagonal, so that the potential can be written
\begin{equation*}
V_A(x)
	= \dfrac{1}{2} x \cdot A x
	= \dfrac{1}{2} \sum_{j = 1}^d \nu_j^2 x_j^2 \,.
\end{equation*}
The \emph{characteristic frequencies} of $H_A$ are those numbers $\nu_1, \nu_2, \ldots, \nu_d$, that we will always assume to be positive. The corresponding Hamiltonian flow is explicit: denoting by $x_1(t), x_2(t), \ldots, x_d(t)$ and $\xi_1(t), \xi_2(t), \ldots, \xi_d(t)$ the components of $\phi^t$, we can solve the Hamilton equations~\eqref{eq:defHamiltonianflowcoord}:
\begin{equation} \label{eq:explicitoschflow}
\left\{
\begin{aligned}
x_j(t) &= \cos(\nu_j t) x_j(0) + \dfrac{1}{\nu_j} \sin(\nu_j t) \xi_j(0) \\
\xi_j(t) &= - \nu_j \sin(\nu_j t) x_j(0) + \cos(\nu_j t) \xi_j(0)
\end{aligned}
\right.\,, \qquad \forall j \in \{1, 2, \ldots, d\} \,.
\end{equation}
From this expression, we see that each coordinate is periodic, so that the trajectories whose initial conditions are of the form $x_j(0) = x_0 \delta_{j = j_0}, \xi_j(0) = \xi_0 \delta_{j = j_0}$ with $x_0, \xi_0 \in \R$, are periodic, with period $2 \pi/\nu_{j_0}$ (unless both $x_0$ and $\xi_0$ vanish, in which case the trajectory is a point). Assuming $d = 2$, we can classify harmonic oscillators into three categories. See Figure~\ref{fig:trajectoriesharmonicoscillators} for an illustration.
\begin{itemize} [label=\textbullet]
\item We call \emph{isotropic} a harmonic oscillator\footnote{For general dimension, we still call isotropic any harmonic oscillator having all its characteristic frequencies equal.} with $\nu_1 = \nu_2 = \nu$. In this situation, energy surfaces, that is, level sets of the classical Hamiltonian, are concentric spheres in phase space (up to a symplectic change of coordinates). Trajectories of the Hamiltonian flow are great circles on these spheres, so that their projection on the $x$-variable ``physical space" are ellipses. The flow is periodic, with period $\frac{2\pi}{\nu}$.
\item The harmonic oscillator is said to be \emph{anisotropic rational} when $\frac{\nu_2}{\nu_1}$ is a rational number different from $1$. Trajectories, although all closed, exhibit a more complicated behavior. Writing $\frac{\nu_2}{\nu_1} = \frac{p}{q}$ with $p$ and $q$ coprime positive integers, the period of the flow is $p \frac{2\pi}{\nu_2} = q \frac{2\pi}{\nu_1}$. Projected trajectories are known in the physics literature as Lissajous curves~\cite{Lissajous}.
\item We say a harmonic oscillator is \emph{anisotropic irrational} when $\frac{\nu_2}{\nu_1} \in \R \setminus \Q$. In that case, the Hamiltonian flow is aperiodic. Trajectories are dense in invariant tori (see~\eqref{eq:invarianttori} below), yielding projected trajectories that fill rectangles parallel to the eigenspaces of the matrix $A$.
\end{itemize}

\begin{figure}
\centering
\begin{subfigure}{\textwidth}
         \centering
         \includegraphics[scale=1.18]{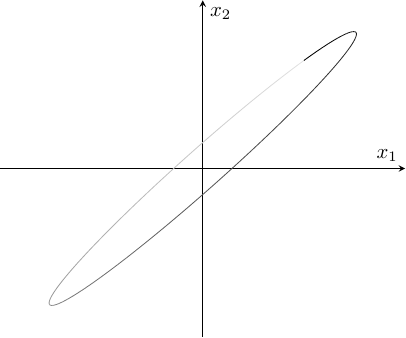}
         \caption{Isotropic harmonic oscillator: $\frac{\nu_2}{\nu_1} = 1$.}
         \label{subfig:isotropic}
     \end{subfigure}
     \vfill
     \begin{subfigure}{\textwidth}
         \centering
         \includegraphics[scale=1.18]{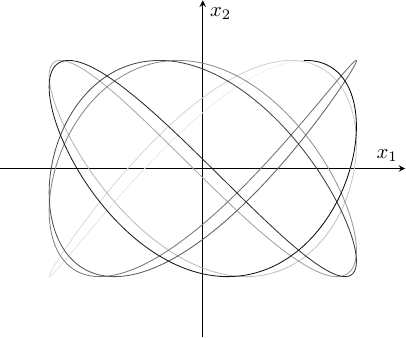}
         \caption{Rational anisotropic harmonic oscillator: $\frac{\nu_2}{\nu_1} = \frac{7}{5}$.}
         \label{subfig:anisotropicrational}
     \end{subfigure}
     \vfill
     \begin{subfigure}{\textwidth}
         \centering
         \includegraphics[scale=1.18]{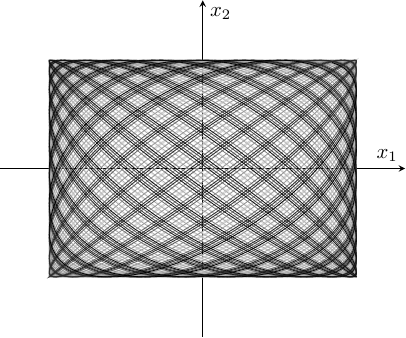}
         \caption{Irrational anisotropic harmonic oscillator: $\frac{\nu_2}{\nu_1} = \frac{\pi}{3}$.}
         \label{subfig:anisotropicirrational}
     \end{subfigure}
        \caption{Typical projected trajectories of two-dimensional harmonic oscillators. Shading indicates the course of the trajectory.}
        \label{fig:trajectoriesharmonicoscillators}
\end{figure}

In the multi-dimensional setting, the description of the flow can be achieved by examining the $\Q$-vector space generated by the characteristic frequencies. The dimension of the latter gives the number of periodic decoupled ``sub-oscillators" from which we can reconstruct the dynamics of the whole oscillator. This is thoroughly explained in the article of Arnaiz and Macià~\cite{AM:22}, who compute the set of quantum limits of general harmonic oscillators, and study their behavior when bounded perturbations of the potential are added~\cite{AM:22perturbed}.

In order to understand well the classical dynamics of the harmonic oscillator, it is convenient to take advantage of the complete integrability of this dynamical system. Here, the classical Hamiltonian is the sum of the one-dimensional Hamiltonians $\frac{1}{2}(\nu_j^2 x_j^2 + \xi_j^2)$, which are conserved by the flow, as one can see from the explicit expression~\eqref{eq:explicitoschflow}. This property implies that energy levels are foliated in (possibly degenerate) invariant $d$-dimensional tori of the form:
\begin{equation} \label{eq:invarianttori}
\tor_E
    = \set{(x, \xi) \in \R^{2d}}{\forall j, \, \dfrac{1}{2} \left(\nu_j^2 x_j^2 + \xi_j^2\right) = E_j} \,, \qquad
        E = (E_1, E_2, \ldots, E_d) \in \R_+^d \,.
\end{equation}
The projection of these tori on the $x$-variable space yields rectangles, as in Figure~\ref{subfig:anisotropicirrational}.

The goal of the following examples is to highlight the fact that observability is sensitive to the global properties of the Hamiltonian flow. We will show that isotropic and anisotropic harmonic oscillators behave differently with respect to observability, i.e.\ the sets that observe the Schrödinger equation are not the same. One can already anticipate that the isotropic oscillator $\nu_1 = \nu_2$ has less such sets since its classical trajectories are all ellipses, that is, they are very simple and only explore a small part of the classically allowed region. It contrasts with the anisotropic situation $\nu_1 \neq \nu_2$, where, in the rational case for instance, trajectories visit more exhaustively the classically allowed region. It makes it harder to find a set that is not reached by any of these trajectories. It is even more the case when $\nu_1$ and $\nu_2$ are rationally independent, since the trajectories are then dense in the invariant torus to which they belong, as we already discussed.

\subsubsection{Observability from conical sets} \label{subsub:conicalsets}

We first investigate the case where the observation set $\omega$ is conical, namely it is invariant by dilations with positive scaling factor:
\begin{equation} \label{eq:defconicalset}
\forall x \in \R^d, \forall \lambda > 0 \,, \qquad
	\left( x \in \omega \quad \Longleftrightarrow \quad \lambda x \in \omega \right) \,.
\end{equation}
We will see that exploiting the symmetries of harmonic oscillators is sometimes sufficient to obtain satisfactory results, without the need of our main theorem (see Subsection~\ref{subsub:refinementconical}). However, Theorem~\ref{thm:main} will prove useful to estimate precisely the optimal observation time in some situations.

As we already noticed, it follows from the expression of the flow~\eqref{eq:explicitoschflow} that, whatever the characteristic frequencies, the classical dynamics exhibits periodic trajectories contained in the coordinate axes. Those starting from the origin are of the form
\begin{equation*}
x_j(t) = \dfrac{1}{\nu_j} \sin(\nu_j t) \xi_j(0) \,, \qquad \xi_j(t) = \cos(t) \xi_j(0)
\end{equation*}
for one $j \in \{1, 2, \ldots, d\}$, and with all the other components being equal to zero. Thus it appears that a general necessary condition for a conical $\omega$ to observe the Schrödinger equation~\eqref{eq:Schrodinger}, working for any harmonic oscillator, is that it contains at least half of each line spanned by an eigenvector of $A$. Note that this works in any dimension.

\begin{prop} \label{prop:necessaryconditionconicset}
Consider $P = V(x) - \tfrac{1}{2} \Delta$, where $V$ is a potential fulfilling Assumptions~\ref{assum:V} and having principal symbol $V_A(x) = \frac{1}{2} x \cdot A x$, $A$ being a real symmetric positive-definite $d \times d$ matrix. Let $\omega \subset \R^d$ be a conical set and assume that it observes the Schrödinger equation in some time $T > 0$. Then for all eigenvector $v$ of $A$, it holds $v \in \ovl{\omega}$ or $-v \in \ovl{\omega}$.
\end{prop}

Now we place ourselves in dimension $d = 2$. We know from the above Proposition~\ref{prop:necessaryconditionconicset} that the closure of a conical set which observes the Schrödinger equation has to contain at least half of any line spanned by an eigenvector of the matrix $A$. Here, we exhibit a conical observation set, illustrated on Figure~\ref{fig:saturatingcurve}, that behaves differently according to whether the harmonic oscillator is isotropic or not.

\begin{prop}[Conical sets and anisotropy] \label{prop:twocones}
Let $d = 2$ and consider a potential $V$ fulfilling Assumption~\ref{assum:V}, and with principal symbol
\begin{equation*}
V_A(x) = \frac{1}{2} x \cdot A x \,, \qquad x \in \R^2 \,,
\end{equation*}
where $A$ is a real symmetric positive-definite matrix. Denote by $\nu_+ \ge \nu_- > 0$ its characteristic frequencies. Choose an orthonormal basis of eigenvectors $(e_+, e_-)$ of $A$, so that $A e_\pm = \nu_\pm^2 e_\pm$. For any $\eps \in (0, \pi/2)$, define the two cones with aperture $\eps$:
\begin{equation} \label{eq:deftwocones}
C_\eps^\pm = \set{x \in \R^2}{\abs{x \cdot e_\mp} < \tan\left(\dfrac{\eps}{2}\right) x \cdot e_\pm} \,.
\end{equation}
Then the set $\omega(\eps) = C_\eps^+ \cup C_\eps^-$ observes the Schrödinger equation if and only if the oscillator is anisotropic, that is $\nu_- < \nu_+$. In that case, there exist constants $C, c > 0$, possibly depending on $\nu_+, \nu_-$, such that for any $\eps \in (0, \pi/2)$,
\begin{equation} \label{eq:optimaltimeconicset}
T_0 - C \eps^2 \le T_\star\left(\omega(\eps)\right) \le T_0 - c \eps^2 \,,
    \qquad \rm{where} \quad T_0 = \dfrac{\pi}{\nu_+} \left(2 + \left\lfloor \dfrac{\nu_+}{\nu_-} \right\rfloor\right) \,.
\end{equation}
\end{prop}
This results leads to several noteworthy observations.

First, it does not distinguish between rational and irrational anisotropic oscillators: one cannot guess, from the knowledge that observability form $\omega(\eps)$ holds, whether the oscillator is rational or irrational.

Second, the time $T_0$, obtained formally as the limiting optimal observation time when $\eps \to 0$, does not vary continuously with respect to $\nu_+$ and $\nu_-$ because of the floor function. This is related to special symmetry properties of the Hamiltonian flow that appear when $\nu_+$ is a multiple of $\nu_-$, namely the projected trajectories can go from a quadrant to another one crossing the origin, and thus avoiding to cross the observation cones. See Figure~\ref{fig:saturatingcurve}.

\begin{figure}
\centering
\includegraphics[scale=1.5]{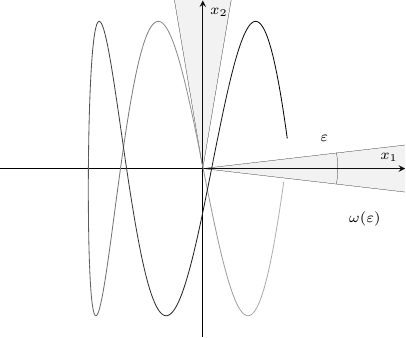}
\caption{The above projected trajectory is responsible for the lower bound on the optimal observation time in~\eqref{eq:optimaltimeconicset}. It is obtained with an oscillator such that $\frac{\nu_2}{\nu_1} = 3.9$. For $\frac{\nu_2}{\nu_1} = 4$, one can choose the initial datum so that the curve goes back to the upper-right quadrant, passing through the origin, without crossing the two cones. This yields a larger lower bound on the optimal time, corresponding to the jump from $\lfloor 3.9 \rfloor = 3$ to $\lfloor 4 \rfloor = 4$ in Formula~\eqref{eq:optimaltimeconicset}.}
\label{fig:saturatingcurve}
\end{figure}

Third, when $\eps$ is fixed, our result does not rule out that the optimal observation time is continuous with respect to the characteristic frequencies. This is because the constants $C$ and $c$ can depend on $\nu_1, \nu_2$.

It is interesting to see what happens when $\nu_+, \nu_- \to \nu$, that is to say when the operator $P$ becomes closer to an isotropic harmonic oscillator. As mentioned earlier, we know from Proposition~\ref{prop:necessaryconditionconicset} that observability is not true for a set of the form $C_\eps^+ \cup C_\eps^-$ for an isotropic oscillator ($\eps < \pi/2$ is important here). Thus it can seem surprising that the optimal observation time for such a set is bounded uniformly in $\nu_+, \nu_-$ as the frequencies tend to $\nu$. Actually, degeneracy in this limit should be seen on the observation cost, rather than on the optimal observation time. Indeed, computations suggest that the value of the dynamical constant $\frak{K}_p^\infty(\omega(\eps), T)$ tends to zero; see~\eqref{eq:dynamicalassumptioneps2} in the proof. This would imply a blow up of the observation cost as $\nu_1, \nu_2 \to \nu$, in virtue of the necessary condition part of Theorem~\ref{thm:main}.

\subsubsection{Refinement for the unperturbed isotropic harmonic oscillator} \label{subsub:refinementconical}

Theorem~\ref{thm:main} allows to conclude about whether an open set $\omega$ observes the Schrödinger equation provided this open set is in a sense ``regular": the thickening process yields open sets that are sufficiently close to a cut-off function. But the quest of characterizing general measurable sets seems to be more delicate. To understand the limitation of our main theorem, we investigate the very particular case of the isotropic harmonic oscillator and conical observation sets in dimension $d \ge 1$. In this setting, we can take advantage of symmetries and exact propagation of coherent states.

For the purpose of the statement, let us introduce some notation. A conical set in $\R^d$ is determined by the subset $\Sigma = \omega \cap \sph^{d - 1}$ in the unit sphere. When $\Sigma \subset \sph^{d - 1}$ we denote by $\omega(\Sigma)$ the conical set defined by
\begin{equation} \label{eq:defomegaIconical}
\omega(\Sigma) = \set{x \in \R^d \setminus \{0\}}{\dfrac{x}{\abs{x}} \in \Sigma} \,.
\end{equation}
Moreover, for any subset $\Sigma \subset \sph^{d - 1}$, we introduce the notation
\begin{equation*}
- \Sigma = \set{\theta \in \sph^{d - 1}}{- \theta \in \Sigma} \,.
\end{equation*}
The lower density of a measurable set $\Sigma \subset \sph^{d - 1}$, denoted by $\Theta_\Sigma^-$, is the function $\sph^{d - 1} \to [0, 1]$ defined by
\begin{equation} \label{eq:deflowerdensity}
\Theta_\Sigma^-(\theta)
	= \liminf_{r \to 0} \dfrac{\sigma\left(\Sigma \cap B_r(\theta)\right)}{\sigma\left(B_r(\theta)\right)} \,, \qquad \forall \theta \in \sph^{d - 1} \,,
\end{equation}
where $B_r(\theta)$ is the ball of radius $r$ centered at $\theta$ in $\R^d$, and $\sigma$ is the uniform probability measure on the unit sphere $\sph^{d - 1}$.

We insist on the fact that the statement below in proved for exact isotropic harmonic oscillators, and not for perturbations of it.

\begin{prop} \label{prop:isotropicconicalsets}
Let $P = \frac{1}{2} ( \nu^2 \abs{x}^2 - \Delta )$ be an isotropic oscillator with characteristic frequency $\nu > 0$. Let $\Sigma \subset \sph^{d - 1}$ be measurable, and $\omega(\Sigma)$ be the corresponding conical set. Set $\widehat \Sigma = \Sigma \cup - \Sigma$ the symmetrized version of $\Sigma$.
\begin{enumerate} [label=(\roman*)]
\item If the Schrödinger equation is observable from $\omega(\Sigma)$ in some time, then it holds
\begin{equation*}
\inf_{\sph^{d - 1}} \Theta_{\widehat \Sigma}^- > 0 \,.
\end{equation*}
\item If $\widehat \Sigma = \Sigma \cup - \Sigma$ has full measure, namely $\sigma\left(\sph^{d - 1} \setminus \widehat \Sigma\right) = 0$, or equivalently $\Theta_{\widehat \Sigma}^-(\theta) = 1$ for all $\theta \in \sph^{d - 1}$, then $\omega(\Sigma)$ observes the Schrödinger equation, with optimal observation time $T_\star < \frac{2 \pi}{\nu}$.
\end{enumerate}
\end{prop}

\begin{rema}
The gap between the sufficient and the necessary conditions above can be thought as the difference between $\Sigma$ being the complement of a Cantor set (thus having full measure) and $\Sigma$ being the complement of a fat Cantor set; see~\cite[Chapter 2, p. 80]{Stromberg:81}. Regarding the estimate on the optimal observation time, the strict inequality is due to Lemma~\ref{lem:obsopenintime}.
\end{rema}

In fact, considering the propagation of coherent state, as investigated for instance by Combescure and Robert (see~\cite{CombescureRobertWP}), one could conjecture that observability is characterized by the property
\begin{equation} \label{eq:dynassumptionL1}
\exists R > 0 : \qquad
    \liminf_{\rho \to \infty} \int_0^T \abs*{\omega \cap B_R\left(\phi^t(\rho)\right)} \dd t > 0 \,.
\end{equation}
This type of integral can be rewritten as
\begin{equation*}
\int_0^T \abs*{\omega \cap B_R\left(\phi^t(\rho)\right)} \dd t
    = \int_0^T \norm*{\one_\omega}_{L^1(B_R(\phi^t(\rho)))} \dd t \,.
\end{equation*}
The necessary condition of Theorem~\ref{thm:main}, namely $\frak{K}_p^\infty(\omega_R, T) > 0$ for some $R$ large enough, involves the quantity
\begin{equation} \label{eq:dynassumptionLinfty}
\int_0^T \one_{\omega_R}\left(\phi^t(\rho)\right) \dd t
    = \int_0^T \norm*{\one_\omega}_{L^\infty(B_R(\phi^t(\rho)))} \dd t \,.
\end{equation}
Since the $L^1$ norm in a ball of radius $R$ is controlled by the $L^\infty$ norm (times a constant of order $R^d$), we know that the dynamical condition~\eqref{eq:dynassumptionL1} is stronger than the condition $\frak{K}_p^\infty(\omega_R, T) > 0$, involving the $L^\infty$ norm as written in~\eqref{eq:dynassumptionLinfty}. In particular, if $\omega$ is dense but Lebesgue negligible, the condition $\frak{K}_p^\infty(\omega_R, T) > 0$ will be satisfied, since then $\omega_R = \R^d$ for any $R > 0$, whereas~\eqref{eq:dynassumptionL1} will not. In this situation, Theorem~\ref{thm:main} would then yield a trivial result, namely that observability holds from the whole space, although it clearly does not hold from $\omega$ itself. Thus~\eqref{eq:dynassumptionL1} seems to be a good guess to free ourselves from thickening the observation set. In addition, this condition would be consistent with the generalized geometric control condition introduced by Burq and Gérard in the context of stabilization of the wave equation~\cite{BurqGerard20}.

\subsubsection{Observability from spherical sets} \label{subsub:sphericalsets}

In this section, we investigate the observability properties of a set consisting in a union of spherical layers. In the sequel, we refer to rotation-invariant (measurable) sets as \emph{spherical} sets. Such a set $\omega$ is completely determined by the data of a measurable set $I \subset \R_+$, such that
\begin{equation} \label{eq:sphericalset}
\omega = \omega(I) = \set{x \in \R^d}{\abs{x} \in I} \,.
\end{equation}
Due to the thickening process that occurs when applying Theorem~\ref{thm:main}, we shall generally make further assumptions, that ensure that a set and its thickened version are somewhat equivalent.

The existence of many periodic circular orbits of the Hamiltonian flow for radial potentials implies that observability from $\omega(I)$ does not hold for such Hamiltonians if $I$ contains large gaps. In fact, the proposition below works for slightly more general potentials.

\begin{prop} \label{prop:sphericalsetwithaxiallysymmetricpotential}
Let $d \ge 2$. Suppose the Hamiltonian $P$ is of the form $P = V(x) - \tfrac{1}{2} \Delta$ with a potential $V$ satisfying Assumption~\ref{assum:V} together with:
\begin{enumerate} [label=(\roman*)]
\item\label{it:assum1} $V(-x) = V(x), \forall x \in \R^d$;
\item\label{it:assum2} there exists an orthogonal change of coordinates $M$ such that
\begin{equation*}
V(M S_\theta M^{-1} x)
	= V(x) \,,
		\qquad \forall x \in \R^d, \forall \theta \in \R \,,
\end{equation*}
where $S_\theta$ is the rotation of angle $\theta$ acting on the first two coordinates; in particular, for every $y \in \R^{d - 2}$, the map $V_y : (x_1, x_2) \mapsto V(M (x_1, x_2, y))$ is radial;
\item\label{it:assum3} the map $\tilde V_0$ such that $V_{y=0}(x_1, x_2) = \tilde V_0(\abs{(x_1, x_2)})$ is non-decreasing.
\end{enumerate}
Then for any spherical set $\omega(I)$, if observability holds from $\omega(I)$ in some time $T > 0$, it holds
\begin{equation} \label{eq:nobiggapinI}
\exists r > 0 : \forall s \in \R_+ : \qquad
    I \cap [s, s + r] \neq \varnothing \,.
\end{equation}
\end{prop}

\begin{rema}
The hypotheses are fulfilled for harmonic oscillators in $d$ dimensions having at least two identical characteristic frequencies.
\end{rema}

In dimension $2$, Proposition~\ref{prop:sphericalsetwithaxiallysymmetricpotential} allows to conclude that spherical sets observing the Schrödinger equation for isotropic harmonic oscillators have to occupy space somewhat uniformly---they cannot contain arbitrarily large gaps. Therefore, we shall rule out isotropic harmonic oscillators from our study of observability from spherical sets. Instead, we investigate how the anisotropy of a harmonic oscillator can help to get observability from an observation set made of concentric rings. The proposition below investigates, in dimension $2$, the observability from spherical sets of the form $\omega(I)$
where $I = \bigcup I_n$ is a countable union of open intervals in $\R_+$. We require additionally that $\abs{I_n} \to + \infty$ (we drop this assumption if there are only finitely many $I_n$'s). To any such set, we associate a number between $0$ and $1$ that quantifies the distribution of the annuli $\omega(I_n)$ at infinity:
\begin{equation} \label{eq:defkappastar}
\kappa_\star(I)
	= \min \set{\kappa \in [0, 1]}{\liminf_{r \to + \infty} \dfrac{1}{r} \abs*{I \cap [\kappa r, r]} = 0}
		\in [0, 1] \,.
\end{equation}

While investigating the observability property from such a set $\omega(I)$, we will see that it is relevant to compare the geometrical quantity $\kappa_\star(I)$ with a dynamical quantity that encodes relevant features of the underlying Hamiltonian flow. This dynamical constant is expressed in terms of a function $\Lambda : \R_+^\star \to [0, 1]$ defined by
\begin{equation} \label{eq:deffunctionLambda}
\Lambda(\mu) =
\left\{
\begin{aligned}
& \tan\left(\dfrac{\pi/2}{p + q}\right) & \quad & \rm{if} \;\; \mu = \frac{p}{q} \,, \;\; \gcd(p, q) = 1 \,, \;\; p - q \equiv 0 \pmod{2} \\
& \sin\left(\dfrac{\pi/2}{p + q}\right) & \quad & \rm{if} \;\; \mu = \frac{p}{q} \,, \;\; \gcd(p, q) = 1 \,, \;\; p - q \equiv 1 \pmod{2} \\
& 0 & \quad & \rm{if} \;\; \mu \in \R \setminus \Q 
\end{aligned}
\right. \, .
\end{equation}

As far as the optimal observation time $T_\star$ is concerned, we shall use Diophantine properties of $\mu$ to approximate irrational oscillators by rational ones, for which we can control $T_\star$ by the period of the flow. This motivates the introduction of the irrationality exponent of an irrational number $\mu$, defined by
\begin{equation} \label{eq:defirrationalityexponent}
\tau(\mu) = \sup \set{s \in \R}{\abs*{\mu - \dfrac{p}{q}} < \dfrac{1}{q^s} \; \textrm{for infinitely many coprime couples $(p, q)$}} \,.
\end{equation}
Dirichlet's approximation theorem tells us that $\tau(\mu) \in [2, + \infty]$, for any irrational number. Also keep in mind that $\tau(\mu) = 2$ is achieved for Lebesgue-almost every irrational. See the lecture notes~\cite{Durand:mnt} or the books~\cite{EinsiedlerWard:Book,Schmidt:Book} for further details.

\begin{prop}[Spherical sets and anisotropy] \label{prop:anisotropicsphericalsets}
Let $d = 2$ and consider a potential $V$ fulfilling Assumption~\ref{assum:V}, and with principal symbol
\begin{equation*}
V_A(x) = \dfrac{1}{2} x \cdot A x \,, \qquad x \in \R^2 \,,
\end{equation*}
where $A$ is a real symmetric positive-definite matrix. Denote by $\nu_1$ and $\nu_2$ the characteristic frequencies of $A$, and assume that $\nu_1 \neq \nu_2$. We fix $I = \bigcup I_n$ a union of open intervals in $\R_+$, assuming that $\abs{I_n} \to + \infty$. Denote by $\omega(I)$ the corresponding open spherical set in $\R^2$, as defined in~\eqref{eq:sphericalset}. Then observability from $\omega(I)$ holds in some time $T$ if and only if
\begin{equation} \label{eq:conditionsphericalsets}
\kappa_\star(I)
	> \Lambda\left(\frac{\nu_2}{\nu_1}\right) \,.
\end{equation}
Moreover, the optimal observation time $T_\star$ can be estimated as follows:
\begin{itemize}[label=\textbullet]
\item if $\frac{\nu_2}{\nu_1} \in \Q$, writing $\frac{\nu_2}{\nu_1} = \frac{p}{q}$ with $p, q$ positive coprime integers, it holds
\begin{equation*}
T_\star
	< \frac{\pi}{\nu_2} p = \frac{\pi}{\nu_1} q \,;
\end{equation*}
\item if $\frac{\nu_2}{\nu_1} \in \R \setminus \Q$ is Diophantine, that is $\tau = \tau(\tfrac{\nu_2}{\nu_1}) < \infty$, then it holds
\begin{equation} \label{eq:optimaltimeinDiophantinecase}
\forall \eps > 0, \exists c_\eps, C_\eps > 0 : \qquad
	c_\eps \left(\dfrac{1}{\kappa_\star(I)}\right)^{\frac{1}{\tau - 1 + \eps}} \le T_\star \le C_\eps \left(\dfrac{1}{\kappa_\star(I)}\right)^{\tau - 1 + \eps} \,.
\end{equation}
The constants $c_\eps$ and $C_\eps$ may depend on $\nu_1, \nu_2$, but not on $I$.
\end{itemize}
\end{prop}

Let us review the meaning of the different quantities involved in this statement.

The number $\kappa_\star(I)$ introduced in~\eqref{eq:defkappastar} encodes some notion of density\footnote{Beware of the fact that $\kappa_\star(I)$ does not coincide in general with the lower density of $I$ defined by
\begin{equation*}
\Theta_\infty(I)
	= \liminf_{r \to + \infty} \frac{\abs{I \cap [0, r]}}{\abs{[0, r]}} \,. \end{equation*}
In fact, the two quantities satisfy
\begin{equation*}
\Theta_\infty(I) \le \kappa_\star(I)
	\qquad \rm{and} \qquad
\kappa_\star(I) = 0 \;\, \Longleftrightarrow \,\; \Theta_\infty(I) = 0 \,.
\end{equation*}
The second assertion follows from the definition of $\kappa_\star(I)$. To check the first assertion, we write
\begin{equation*}
\dfrac{\abs{I \cap [0, r]}}{\abs{[0, r]}}
	= \kappa_\star \dfrac{\abs{I \cap [0, \kappa_\star r]}}{\abs{[0, \kappa_\star r]}} + \dfrac{1}{r} \abs{I \cap [\kappa_\star r, r]}
	\le \kappa_\star + \dfrac{1}{r} \abs{I \cap [\kappa_\star r, r]} \,.
\end{equation*}
Then taking lower limits as $r \to + \infty$ and using the definition of $\kappa_\star$ yield the desired inequality. Notice that the equality $\Theta_\infty(I) = \kappa_\star(I)$ is not true in general, as one can see from the example $I = \bigcup_{n \in \N} (n, n+\frac{1}{2})$, for which we have $\Theta_\infty(I) = \frac{1}{2}$ but $\kappa_\star(I) = 1$.} of the set~$I$. For instance, $\kappa_\star(I) = 1$ means that $I$ has positive density in any window $[\kappa r, r]$ with $\kappa < 1$ as $r \to + \infty$. In contrast, $\kappa_\star(I)$ close to zero means that the annuli are extremely sparse at infinity. This quantity is well-defined, for the map
\begin{equation*}
\kappa \longmapsto \liminf_{r \to + \infty} \dfrac{1}{r} \int_{\kappa r}^r \one_I(s) \dd s
\end{equation*}
is non-increasing and lower semi-continuous (even Lipschitz-continuous in fact). That it is non-increasing comes from the monotonicity of the integral and of the lower limit, whereas the continuity follows from the fact that
\begin{equation*}
\abs*{\dfrac{1}{r} \int_{\kappa_2 r}^r \one_I(s) \dd s - \dfrac{1}{r} \int_{\kappa_1 r}^r \one_I(s) \dd s}
    \le \abs*{\kappa_2 - \kappa_1} \,.
\end{equation*}

Given $\mu \in \R_+^\star$, the constant $\Lambda(\mu)$ defined in~\eqref{eq:deffunctionLambda} is related to the flow of a harmonic oscillator with characteristic frequencies $\nu_1, \nu_2$ such that $\mu = \frac{\nu_2}{\nu_1}$. More precisely, it corresponds to the largest ratio between the minimum and the maximum of the distance to the origin of a projected trajectory: if we write $(x^t, \xi^t)(\rho_0) = \phi^t(\rho_0)$, then we shall prove that
\begin{equation} \label{eq:Lambda(mu)geom}
\Lambda\left(\dfrac{\nu_2}{\nu_1}\right)
	= \sup_{\rho_0 \in \R^4 \setminus \{0\}} \dfrac{\inf_{t \in \R} \abs{x^t(\rho_0)}}{\sup_{t \in \R} \abs{x^t(\rho_0)}} \,.
\end{equation}
Thus we can refer to this quantity as the optimal ``radial aspect ratio" of projected trajectories.
Observability from $\omega(I)$ will depend on whether the critical trajectories that attain this maximal ratio spend sufficient time in $\omega(I)$, hence the criterion $\kappa_\star(I) > \Lambda(\tfrac{\nu_2}{\nu_2})$. See Figure~\ref{fig:non_obs_spherical} for an illustration of the case where such trajectories are not seen by the observation set.
\begin{figure}
\centering
\includegraphics[scale=1.5]{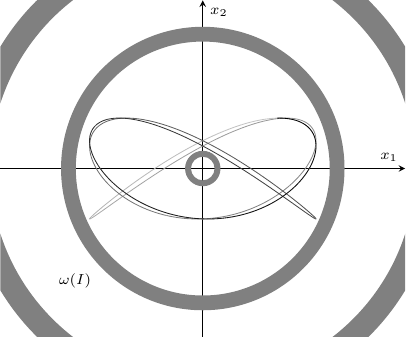}
\caption{The above curve is a projected trajectory of a harmonic oscillator with $\frac{\nu_2}{\nu_1} = \frac{4}{3}$, that does not intersect the observation set $\omega(I)$. The existence of a sequence of energy layers $\{p = E_n\}$, $E_n \to + \infty$, containing such curves would imply that observability from $\omega(I)$ fails.}
\label{fig:non_obs_spherical}
\end{figure}
Notice that maximizing the ratio in~\eqref{eq:Lambda(mu)geom} with respect to any non-zero initial data is the same as taking the upper limit as $\rho_0 \to \infty$ since the Hamiltonian flow is homogeneous. Thus $\Lambda(\mu)$ can be understood as a quantity that captures the behavior of the flow at infinity. In addition, we remark that $\Lambda(\mu) = \Lambda(1/\mu)$, which means that this value depends only on the spectrum of the matrix $A$, and not on the choice of a specific basis of $\R^2$. The maximum of $\Lambda$ is reached exactly at $1$, where it is equal to $\tan(\frac{\pi}{4}) = 1$. This is consistent with the fact that in two dimensions, isotropic harmonic oscillators are the only ones possessing circular orbits: the norm of the trajectory $\abs{x^t(\rho_0)}$ is constant for well-chosen initial data.

The distinction between rational and irrational values of $\mu$ is natural in light of the complete integrability of the flow of harmonic oscillators. When the ratio of characteristic frequencies $\mu = \frac{\nu_2}{\nu_1}$ is rational, writing $\mu = \frac{p}{q}$ with $p, q$ a couple of coprime integers, one can check that the Hamiltonian flow of the corresponding harmonic oscillator is periodic of period $\frac{2 \pi}{\nu_2} p = \frac{2 \pi}{\nu_1} q$. In that case, there are many orbits of the flow whose projection on the $x$-variable space stays away from the origin, thus producing a positive $\Lambda(\mu)$, as one can see on Figure~\ref{subfig:anisotropicrational}. When $\mu$ is irrational, it is known that (non-degenerate) trajectories are dense in the invariant torus to which they belong. In particular, any projected trajectory can get arbitrarily close to the origin, up to waiting a long enough time, so that $\Lambda(\mu) = 0$; see Figure~\ref{subfig:anisotropicirrational}.% Besides, Diophantine properties of $\mu$ allow to give a fairly precise description of the long time behavior of the flow; in other words, they enable us to quantify the rate at which trajectories explore the invariant tori in which they evolve. This idea is ubiquitous in the KAM literature. \todo{no reference!}

Lastly, let us point out that that the estimate~\eqref{eq:optimaltimeinDiophantinecase} of the optimal observation time for Diophantine irrational does not give any precise information for a given open set $I$, but is relevant for fixed $\nu_1, \nu_2$ in the asymptotics $\kappa_\star(I) \ll 1$.

\begin{rema}
It can look surprising that Proposition~\ref{prop:anisotropicsphericalsets} gives an exact characterization of spherical sets for which observability holds, whereas Theorem~\ref{thm:main} provides a necessary and sufficient condition up to thickening the observation set. This improvement is made possible by the extra assumption that $\abs{I_n} \to + \infty$. It ensures that thickening the observation set by a radius $R$ is negligible compared to the width of the annulus $\omega(I_n)$, for $n$ large.
\end{rema}

\begin{rema}[Non-Diophantine irrationals] \label{rmk:nonDiophantineirrationals}
When $\mu = \frac{\nu_2}{\nu_1} \in \R \setminus \Q$, one can estimate $T_\star$, even if $\tau = \tau(\mu) = + \infty$, using the so-called convergents of $\mu$. These are the rational numbers arising in the continued fraction expansion algorithm. Denote them in irreducible form by $\mu_j = \frac{p_j}{q_j}$. It is known that this sequence is the most efficient way to approximate an irrational number by rationals (a result known as Lagrange theorem; see~\cite[Theorem 1.3]{Durand:mnt} or~\cite{EinsiedlerWard:Book,Schmidt:Book}). These convergents satisfy
\begin{equation} \label{eq:diophantineapproximation}
\forall j \in \N \,, \qquad
	\abs*{\mu - \dfrac{p_j}{q_j}}
		< \dfrac{1}{q_j^2} \,.
\end{equation}
(This is why $\tau(\mu) \ge 2$ holds for any irrational.) We will show in the proof of Proposition~\ref{prop:anisotropicsphericalsets} the following: when $\mu \in \R \setminus \Q$, there exist constants $c_1, c_2 > 0$ and $\delta_1, \delta_2 > 0$, possibly depending on $\nu_1, \nu_2$, such that
\begin{equation} \label{eq:estimateoptimaltimewithconvergentsqj}
c_1 q_{j_1} \le T_\star \le c_2 q_{j_2}
\end{equation}
(see~\eqref{eq:boundsonTstar} in the proof), where $j_1$ is the largest index for which $q_j \le \frac{\delta_1}{\kappa_\star}$, and $j_2$ is the smallest index for which $q_j \ge \frac{\delta_2}{\kappa_\star}$.

The bounds~\eqref{eq:optimaltimeinDiophantinecase} are particularly interesting when $\tau$ has the smallest possible value, that is $\tau = 2$, which is the case of Lebesgue-almost every irrational. However, we see that the lower and upper bounds~\eqref{eq:optimaltimeinDiophantinecase} get far apart as $\tau$ goes to infinity. This reflects the fact that the gaps between the denominators of consecutive convergents get wider at each step of the continued fraction expansion. Irrationals having an infinite irrationality exponent are known as Liouville numbers. There are many of them: the set of Liouville numbers is an instance of a Lebesgue negligible set having the cardinality of the continuum. When $\frac{\nu_2}{\nu_1}$ is a Liouville number, the bounds~\eqref{eq:estimateoptimaltimewithconvergentsqj} on the optimal observation time are very poor, owing to the lacunary behavior of the $q_j$'s.
\end{rema}

\subsection{Other applications} \label{subsec:exother}

Let us briefly discuss two other applications of Theorem~\ref{thm:main}.

\subsubsection{Uniform observability of eigenfunctions}

Under Assumption~\ref{assum:V}, the operator $P$ is self-adjoint with compact resolvent. Thus, its spectrum consists in a collection of eigenvalues with finite multiplicity. A direct consequence of an observability inequality $\Obs(\omega, T)$ in a set $\omega$ is the fact that the eigenfunctions of $P$ are uniformly observable from $\omega$:
\begin{equation*}
\exists c > 0 : \forall u \in L^2(\R^d) \,, \qquad
	\left( P u = \lambda u \quad \Longrightarrow \quad \norm*{u}_{L^2(\omega)} \ge c \norm*{u}_{L^2(\R^d)} \right) \,.
\end{equation*}
Theorem~\ref{thm:main} thus furnishes a sufficient condition for this to hold. In particular, for anisotropic oscillators, Proposition~\ref{prop:twocones} implies that uniform observability of eigenfunctions from the two cones defined in~\eqref{eq:deftwocones} is true. This can certainly be deduced from the work of Arnaiz and Macià~\cite{AM:22} that characterizes quantum limits of harmonic oscillators. From Proposition~\ref{prop:anisotropicsphericalsets}, we obtain a similar uniform estimate in spherical sets satisfying the assumptions of the proposition together with the condition~\eqref{eq:conditionsphericalsets}. This time, it is not clear that one can deduce this result as easily from the knowledge of quantum limits~\cite{AM:22}. See also~\cite{DSV:23} for recent works about spectral inequalities for the Hermite operator, and~\cite{Martin:22} for anisotropic Shubin operators.

\subsubsection{Energy decay of the damped wave equation}

Lastly, our study leads to stabilization results concerning the damped wave equation
\begin{equation} \label{eq:dweq}
\left\{
\begin{aligned}
\partial_t^2 \psi + P \psi + \one_\omega \partial_t \psi &= 0 \\
(\psi, \partial_t \psi)_{\vert t = 0} &= U_0 \in \dom P^{1/2} \times L^2
\end{aligned}
\right. \,,
\end{equation}
with damping in $\omega \subset \R^d$, provided $P \ge 0$ (assume for instance that the potential $V$ is non-negative). This equation comes with a natural energy
\begin{equation*}
\cal{E}(U_0, t)
	= \dfrac{1}{2} \left(\norm*{P^{1/2}\psi(t)}_{L^2}^2 + \norm*{\partial_t \psi(t)}_{L^2}^2\right) \,,
\end{equation*}
which decays over time. Let us recall that Anantharaman and Léautaud proved in~\cite[Theorem 2.3]{AL:14} that an observability inequality $\Obs(\omega, T)$ implies a decay at rate $t^{-1/2}$ for the damped wave equation~\eqref{eq:dweq}, meaning that there exists a constant $C > 0$ such that
\begin{equation*}
\cal{E}(U_0, t) \le \dfrac{C}{t} \left( \norm*{P u_0}_{L^2}^2 + \norm*{{P^{1/2}} u_1}_{L^2}^2 \right) \,, \qquad \forall t > 0 \,,
\end{equation*}
for all initial data in the domain of the damped wave operator, here $U_0 = (u_0, u_1) \in \dom P \times \dom P^{1/2}$. Their result applies in our setting since $P$ has compact resolvent under Assumption~\ref{assum:V}. Our examples thus provide concrete situations where such a decay occurs.

\subsection{Link with the Kato smoothing effect} \label{subsec:Kato}

The dynamical condition~\eqref{eq:dynamicalcondition} concerns only what happens at infinity in phase space. We will see that trajectories of the Hamiltonian flow escape from any compact set (in the $x$ variable) most of the time provided the initial data has large enough energy, namely $p(\rho)$ is large enough. This is the reason why one can remove any compact set from the observation without losing observability: no energy can be trapped in a compact set. Quantitatively, we will check that, given $T > 0$, there exist a constant $C > 0$ and $E_0 > 0$ such that
    \begin{equation} \label{eq:capitalOuniforminr??}
\forall r \ge 0, \forall \rho \in \{p \ge E_0\} \,, \quad
    \abs*{\set{t \in [0, T]}{\phi^t(\rho) \in B_r(0) \times \R^d}}
    	= \int_0^T \one_{B_r(x^t)} \dd t
        \le C \dfrac{r}{\sqrt{p(\rho)}}
    \end{equation}
(see Corollary~\ref{cor:classicalnonobs}). We can rephrase this by saying that compact sets are not \emph{classically observable}. This property is related to the Kato smoothing effect as follows. Writing $(x^t, \xi^t) = \phi^t(\rho)$, for any $\eps > 0$, we compute using Fubini's theorem:
\begin{equation*}
\int_0^T \dfrac{\sqrt{p(\rho)}}{\jap{x^t}^{1 + \eps}} \dd t
	= \int_0^T \left( \int_{\jap{x^t}}^{+ \infty} (1 + \eps) \dfrac{\sqrt{p(\rho)}}{r^{2 + \eps}} \dd r \right) \dd t
	= \int_1^{+ \infty} (1 + \eps) \dfrac{\sqrt{p(\rho)}}{r} \left( \int_0^T \one_{B_r(0)}\left(\jap{x^t}\right) \dd t\right) \dfrac{\dd r}{r^{1 + \eps}} \,.
\end{equation*}
From~\eqref{eq:capitalOuniforminr??}, we deduce that
\begin{equation*}
\int_0^T \dfrac{\sqrt{p(\rho)}}{\jap{x^t}^{1 + \eps}} \dd t
	\le C \int_1^{+ \infty} (1 + \eps) \dfrac{\dd r}{r^{1 + \eps}} \,,
\end{equation*}
and the latter integral is indeed convergent when $\eps > 0$. This is the classical analogue to the so-called Kato smoothing effect. In our context, the latter says roughly that
	\begin{equation*}
	\int_0^T \norm*{\jap{x}^{- \frac{1+\eps}{2}} P^{1/4} \e^{- \ii t P} u}_{L^2(\R^d)}^2 \dd t \le C \norm*{u}_{L^2(\R^d)}^2 \,.
	\end{equation*}
    See for instance Doi~\cite{Doi05} for a thorough discussion on this topic. See also the survey of Robbiano~\cite{Robbiano:smoothingSurvey}, as well as Robbiano and Zuily~\cite{RZ:08,RZ:09} and Burq~\cite{BurqSmoothing} for related results. The main phenomenon responsible for this smoothing effect is the fact that $P$ contains a Laplace-Beltrami operator associated with a non-trapping metric (here a flat metric), that is to say all geodesics escape at infinity forward and backward in time. In our case, working with a flat Laplacian enables us to compare the trajectories of the Hamiltonian flow to straight lines, at least for some time near the origin. It would be interesting to see whether our study can be adapted to operators of the form $P = V(x) - \tfrac{1}{2} \Delta_g$ with a non-trapping metric $g$ on $\R^d$ (sufficiently flat at infinity). See~\cite[Lemma 3.1]{MaciaNakamura:pcomm} for an alternative proof that non-trapping implies failure of observability from \emph{bounded} observation sets. The argument relies on semiclassical defect measures.

\subsection{Natural semiclassical scaling for homogeneous potentials} \label{subsec:semiclassicalscaling}

A way to comprehend what goes wrong when the potential is superquadratic is to introduce the natural semiclassical scales associated to our problem, based on an observation of Macià and Nakamura~\cite{MaciaNakamura:pcomm}. Take for simplicity $p(x, \xi) = \abs{x}^{2m} + \abs{\xi}^2$. Following classical arguments, we recall in Appendix~\ref{app:reduction} that the observability inequality reduces to a high-energy observability inequality: roughly speaking, we can restrict ourselves to $L^2$ functions $u$ that are microlocalized around some level set $\{p = E\}$ with $E \gg 1$. Writing
    \begin{equation*}
        p(x, \xi) = E
            \qquad \Longleftrightarrow \qquad
        \abs*{\dfrac{x}{E^{1/2m}}}^{2m} + \abs*{\dfrac{\xi}{E^{1/2}}}^2 = 1 \,,
    \end{equation*}
    we may introduce a small Planck parameter $h$ such that $E = h^{- \gamma}$ for some power $\gamma > 0$. Thus we have
    \begin{equation*}
        \abs{h^{\gamma/2m} x}^{2m} + \abs{h^{\gamma/2} \xi}^2 = 1 \,.
    \end{equation*}
    This motivates the definition of an $h$-dependent Weyl quantization (see Appendix~\ref{app:pseudo})
    \begin{equation*}
        \Op[h]{a}
        	:= \Op[1]{a(h^{\gamma/2m} x, h^{\gamma/2} \xi)} \,,
    \end{equation*}
    for any classical observable $a$ on the phase space. This quantization is properly ``normalized" by choosing $\gamma = \frac{2m}{m + 1}$: with this choice, the corresponding pseudodifferential calculus is expressed in powers of $h$, since then $h^{\gamma/2m} h^{\gamma/2} = h$. Therefore the relevant semiclassical Schrödinger operator is
    \begin{equation*}
        P_h = \Op[h]{p} = h^{\gamma} P \,.
    \end{equation*}
    If one wants to express the observability inequality in terms of the associated propagator, one is then lead to study
    \begin{equation*}
        \e^{- \ii t P} u = \e^{- \ii \frac{t h^{1 - \gamma}}{h} P_h} u \,.
    \end{equation*}
    In other words, running the Schrödinger evolution on a time interval $[0, T]$ amounts to consider a semiclassical time scale of order $h^{1 - \gamma} = h^{\frac{1 - m}{1 + m}}$. It is then clear that this time blows up as $h \to 0$ when $m > 1$. Yet the analysis of the quantum-classical correspondence, for long times, is much more difficult. In particular, it restricts considerably the amount of classical observables whose evolution can be described through the usual Egorov theorem. For this reason, we will not pursue in this direction and stick to the case $m \le 1$. An interesting approach to study this would be to consider first particular potentials for which the classical flow is completely integrable, e.g.\ anharmonic oscillators (see~\cite{BLR:22}). Indeed, observability of the Schrödinger equation has been successfully investigated taking advantage of the completely integrable nature of the underlying classical dynamics in some particular geometrical contexts (e.g.\ in the disk~\cite{ALM:16cras,ALM:16} which corresponds morally to $m = \infty$; see also~\cite{AM:14} on the torus and~\cite{AFKM:15}).

\subsection{Plan of the article}

Section~\ref{sec:studyclassicalflow} below is devoted to the study of the underlying classical dynamics: we show that the Hamiltonian flow is roughly stable under subprincipal perturbations of the potential, and that high-energy projected trajectories can cross compact sets only on a very short period of time. Then we establish an instance of quantum-classical correspondence adapted to our context in Section~\ref{sec:proofmain}, and subsequently prove Theorem~\ref{thm:main}. This is the core of the article. Next, in Sections~\ref{sec:proofconical} and~\ref{sec:proofspherical}, we deal with the examples presented in Subsections~\ref{subsub:conicalsets}, \ref{subsub:refinementconical} and Subsection~\ref{subsub:sphericalsets} (observability from conical and spherical sets respectively). Finally, we recall in Appendix~\ref{app:reduction} a classical result, related to the notion of unique continuation, that shows that the sought observability inequality is equivalent to a similar high-energy inequality. Appendix~\ref{app:pseudo} collects reminders about pseudodifferential operators, as well as refined estimates on the pseudodifferential calculus and the G{\aa}rding inequality needed for Section~\ref{sec:proofmain}.

%%%%%%%%%%%%%%%%%%%%%%%%%%%%%%%%%%%%%%%

\section{Study of the classical dynamics} \label{sec:studyclassicalflow}

In this section, we investigate the properties of the Hamiltonian flow $(\phi^t)_{t \in \R}$ associated with $p$. This study consists essentially in analyzing the ODE system that defines~$\phi^t$, namely the Hamilton equation~\eqref{eq:defHamiltonianflow}. The dynamical condition of Theorem~\ref{thm:main}
\begin{equation*}
    {\frak K}_p^\infty(\omega, T)
    	= \liminf_{\rho \to \infty} \int_0^T \one_{\omega \times \R^d}\left(\phi^t(\rho)\right) \dd t > 0
\end{equation*}
motivates the study of what can be referred to as ``classical observability".

\begin{defi}[Classical observability] \label{def:classicalobs}
Let $q = q(t; \rho)$ be a Borel-measurable\footnote{Recall that Borel-measurability is slightly stronger than Lebesgue-measurability. This restriction ensures that $t \mapsto q(t; \phi^t(\rho))$ is Lebesgue-measurable. This is not a problem in our context since we will consider functions $q$ that are continuous, or at worse, indicator functions of Borel sets.} function on $\R \times \R^{2d}$. Then we say that $q$ is classically observable~if
\begin{equation} \label{eq:notationCcal}
{\frak K}_p^\infty(q)
	:= \liminf_{\rho \to \infty} \int_{\R} q\left(t; \phi^t(\rho)\right) \dd t > 0 \,.
\end{equation}
\end{defi}

Of course, we will be specifically interested in the case where $p$ contains a subquadratic potential and $q = \one_{(0, T) \times \omega \times  \R^{2d}}$, but it is interesting to work out this problem in a more general setting in order to understand to what extent quadratic potentials are critical for the Schrödinger equation.

\subsection{Invariance of classical observability under subprincipal perturbation} \label{subsec:invariancesubprincipalperturbation}

In this subsection, we consider a set of classical symbols on $\R^{2d}$ of order $n_1$ in $x$ and $n_2$ in $\xi$, defined by
\begin{equation*}
S^{n_1, n_2}
	= \set{a \in \cont^\infty(\R^{2d})}{\forall \alpha \in \N^{2d}, \sup_{(x, \xi) \in \R^{2d}} \dfrac{\abs{\partial^\alpha a(x, \xi)}}{\jap*{x}^{n_1 - \abs{\alpha}} + \jap*{\xi}^{n_2 - \abs{\alpha}}} < \infty} \,.
\end{equation*}
A basic example is the classical Hamiltonian $p(x, \xi) = V(x) + \frac{1}{2} \abs{\xi}^2$ that we consider: it belongs to $S^{2m, 2}$. We draw the reader's attention to the fact that this is not a standard symbol class in microlocal analysis. Our aim here is simply to study symbols whose derivatives have similar decay properties as  the classical Hamiltonian $p$. We will not make use of any notion of pseudodifferential calculus in this subsection.

It is clear that these symbol classes are nested in the following way: if $n_1 \le n_1'$ and $n_2 \le n_2'$, then $S^{n_1, n_2} \subset S^{n_1', n_2'}$ (and this inclusion is even continuous with respect to the associated Fréchet structure). Given $n_1, n_2 \in \R$, a real-valued symbol $a \in S^{n_1, n_2}$ is said to be elliptic in $S^{n_1, n_2}$ if $a(x, \xi) \ge c (\jap{x}^{n_1} + \jap{\xi}^{n_2})$ provided $\abs{(x, \xi)}$ is large enough. In addition, the binary relation
\begin{equation} \label{eq:principalsymbolnonstandard}
\forall f, g \in S^{n_1, n_2} \,, \qquad
    f = g \pmod{S^{n_1 - 1, n_2 - 1}} \quad \Longleftrightarrow \quad f - g \in S^{n_1 - 1, n_2 - 1}
\end{equation}
is an equivalence relation, and the projection on the quotient space $S^{n_1, n_2} / S^{n_1 - 1, n_2 - 1}$ is called the principal symbol. Two symbols are said to have the same principal symbol if they belong to the same equivalence class through this projection. In the example of our classical Hamiltonian $p$, these notions of ellipticity and principal symbol are consistent with the terminology used right after Assumption~\ref{assum:V} regarding the potential $V$.

The proposition below is essentially an application of Grönwall's Lemma.

\begin{prop}[Stability estimate] \label{prop:stabilityestimate}
Fix $n_1, n_2 > 0$ and let $p_1, p_2 \in S^{n_1, n_2}$ be elliptic symbols in $S^{n_1, n_2}$. Assume they have the same principal symbol in the sense of~\eqref{eq:principalsymbolnonstandard}. Consider the Hamiltonian flows $(\phi_1^t)_{t \in \R}$ and $(\phi_2^t)_{t \in \R}$ associated with $p_1$ and $p_2$ respectively. Then there exists a constant $C > 0$ such that
\begin{equation*}
\abs*{\phi_2^t(\rho) - \phi_1^t(\rho)}
    \le \e^{C t \jap{p_1(\rho)}^{\max(0, 1 - \frac{2}{n_+})}} \,, \qquad \forall \rho \in \R^{2d}, \forall t \ge 0 \,,
\end{equation*}
where $n_+ = \max(n_1, n_2)$. In particular, when $n_1, n_2 \le 2$, there exists $C > 0$ such that
\begin{equation*}
\abs*{\phi_2^t(\rho) - \phi_1^t(\rho)}
    \le \e^{C t} \,, \qquad \forall \rho \in \R^{2d}, \forall t \ge 0 \,.
\end{equation*}
\end{prop}

\begin{rema}
This result ensures that the distance between $\phi_1^t(\rho)$ and $\phi_2^t(\rho)$ is bounded provided $n_+ \le 2$, on a time interval $[0, T]$ independent of $\rho$. In our problem, this condition on $n_+$ means exactly that the potential is subquadratic.
\end{rema}

\begin{proof}
In this proof, we write $n_+ = \max(n_1, n_2)$ and $n_- = \min(n_1, n_2)$. Set $\tilde p = p_2 - p_1$, which belongs to $S^{n_1 - 1, n_2 - 1}$ by assumption. The Hamilton equation~\eqref{eq:defHamiltonianflow} gives
\begin{align} \label{eq:diffineq}
\abs*{\dfrac{\dd}{\dd t} \left( \phi_2^t(\rho) - \phi_1^t(\rho) \right)}
    &= \abs*{J \left( \nabla p_2\left(\phi_2^t(\rho)\right) - \nabla p_1\left(\phi_1^t(\rho)\right) \right)} \nonumber\\
    &\le \abs*{\nabla p_2\left(\phi_2^t(\rho)\right) - \nabla p_2\left(\phi_1^t(\rho)\right)} + \abs*{\nabla \tilde p\left(\phi_1^t(\rho)\right)} \,.
\end{align}
By assumption, $p_1$ and $p_2$ are elliptic at infinity in $S^{n_1, n_2}$ so that for any $\rho = (x, \xi)$ large enough, one has:
\begin{equation} \label{eq:ellipticityp1p2}
\dfrac{1}{C} \left(\jap*{x}^{n_1} + \jap*{\xi}^{n_2}\right)
    \le \abs{p_j(\rho)}
    \le C \left(\jap*{x}^{n_1} + \jap*{\xi}^{n_2}\right) \,, \qquad j \in \{1, 2\} \,.
\end{equation}
From the definition of $S^{n_1 - 1, n_2 - 1}$, which contains $\tilde p$, we have
\begin{equation*}
\abs*{\nabla \tilde p(\rho)}
	\le C \left(\jap*{x}^{n_1 - 2} + \jap*{\xi}^{n_2 - 2}\right) \,.
\end{equation*}
The ellipticity of $p_2$, that is, the left-hand side of~\eqref{eq:ellipticityp1p2}, then yields
\begin{equation*}
\abs*{\nabla \tilde p(\rho)}
	\le C \left( \abs*{p_1(\rho)}^{\max(0, 1 - \frac{2}{n_1})} + \abs*{p_1(\rho)}^{\max(0, 1 - \frac{2}{n_2})} \right)
	\le C' \abs*{p_1(\rho)}^{\max(0, 1 - \frac{2}{n_+})} \,,
\end{equation*}
provided $\abs{\rho}$ is large enough. We obtain on the whole phase space:
\begin{equation} \label{eq:nablatildep}
\abs*{\nabla \tilde p(\rho)}
	\le C + C \abs*{p_1(\rho)}^{\max(0, 1 - \frac{2}{n_+})} \,, \qquad \forall \rho \in \R^{2d} \,.
\end{equation}
Now we deal with the other term in~\eqref{eq:diffineq}: the mean-value inequality yields
\begin{equation} \label{eq:mviforgronwall}
\abs*{\nabla p_2\left(\phi_2^t(\rho)\right) - \nabla p_2\left(\phi_1^t(\rho)\right)}
    \le \abs*{\phi_2^t(\rho) - \phi_1^t(\rho)} \times \sup_{s \in [0, 1]} \abs*{\hess p_2\left((1 - s) \phi_1^t(\rho) + s \phi_2^t(\rho)\right)} \,.
\end{equation}
Write for short $\rho_s^t = (1 - s) \phi_1^t(\rho) + s \phi_2^t(\rho)$. Using that $p_2 \in S^{n_1, n_2}$, we obtain
\begin{equation*}
\abs*{\hess p_2\left(\rho_s^t\right)}
    \le C \left( \jap*{(1 - s) x_1^t + s x_2^t}^{n_1 - 2} + \jap*{(1 - s) \xi_1^t + s \xi_2^t}^{n_2 - 2} \right) \,,
\end{equation*}
where we wrote $\phi_j^t(\rho) = (x_j^t, \xi_j^t)$, $j \in \{1, 2\}$. Then we use the classical inequality $\jap{a + b} \le 2(\jap{a} + \jap{b})$ to get
\begin{align*}
\abs*{\hess p_2\left(\rho_s^t\right)}
    &\le C \left( \left(\jap*{x_1^t} + \jap*{x_2^t}\right)^{\max(0, n_1 - 2)} + \left(\jap*{\xi_1^t} + \jap*{\xi_2^t}\right)^{\max(0, n_2 - 2)} \right) \\
    &\le C' \left(\jap*{x_1^t}^{\max(0, n_1 - 2)} + \jap*{\xi_1^t}^{\max(0, n_2 - 2)}\right)
    	+ C' \left(\jap*{x_2^t}^{\max(0, n_1 - 2)} + \jap*{\xi_2^t}^{\max(0, n_2 - 2)}\right) \,.
\end{align*}
Next we use the ellipticity of $p_1$ and $p_2$ and the fact that they are conserved by the corresponding flows:
\begin{align*}
\abs*{\hess p_2\left(\rho_s^t\right)}
	&\le C \left( \abs*{p_1\left(\phi_1^t(\rho)\right)}^{\max(0, 1 - \frac{2}{n_+})} + \abs*{p_2\left(\phi_2^t(\rho)\right)}^{\max(0, 1 - \frac{2}{n_+})} \right) \\
	&= C \left( \abs*{p_1(\rho)}^{\max(0, 1 - \frac{2}{n_+})} + \abs*{p_2(\rho)}^{\max(0, 1 - \frac{2}{n_+})} \right) \,,
\end{align*}
which holds for $\abs{\rho}$ large enough. Up to adding a constant, this works for all $\rho \in \R^d$. Finally we use the fact that $p_1$ and $p_2$ are comparable (a consequence of ellipticity) to obtain
\begin{equation*}
\abs*{\hess p_2\left(\rho_s^t\right)}
	\le C + C \abs*{p_1(\rho)}^{\max(0, 1 - \frac{2}{n_+})} \,,
		\qquad \forall \rho \in \R^{2d} \,.
\end{equation*}
Plugging this into~\eqref{eq:mviforgronwall}, that results in
\begin{equation*}
\abs*{\nabla p_2\left(\phi_2^t(\rho)\right) - \nabla p_2\left(\phi_1^t(\rho)\right)}
    \le C \abs*{\phi_2^t(\rho) - \phi_1^t(\rho)} \times \left(1 + \abs*{p_1(\rho)}^{\max(0, 1 - \frac{2}{n_+})}\right) \,,
\end{equation*}
for all $\rho \in \R^{2d}$. Putting this together with~\eqref{eq:nablatildep}, we estimate the right-hand side of~\eqref{eq:diffineq} from above as:
\begin{equation*}
\abs*{\dfrac{\dd}{\dd t} \left( \phi_2^t(\rho) - \phi_1^t(\rho) \right)}
    \le C\left(1 + \abs*{\phi_2^t(\rho) - \phi_1^t(\rho)} \right) \times \left(1 + \abs*{p_1(\rho)}^{\max(0, 1 - \frac{2}{n_+})}\right) \,.
\end{equation*}
We deduce that
\begin{align*}
\abs*{\dfrac{\dd}{\dd t} \jap*{\phi_2^t(\rho) - \phi_1^t(\rho)}}
    &= \abs*{\dfrac{\dd}{\dd t} \left( \phi_2^t(\rho) - \phi_1^t(\rho) \right) \cdot \dfrac{\phi_2^t(\rho) - \phi_1^t(\rho)}{\jap{\phi_2^t(\rho) - \phi_1^t(\rho)}}} \\
    &\le C \jap*{\phi_2^t(\rho) - \phi_1^t(\rho)} \left(1 + \abs*{p_1(\rho)}^{\max(0, 1 - \frac{2}{n_+})}\right) \,,
\end{align*}
for any $\rho \in \R^{2d}$. We conclude by Grönwall's Lemma that
\begin{equation*}
\jap*{\phi_2^t(\rho) - \phi_1^t(\rho)}
    \le \e^{C t \jap{p_1(\rho)}^{\max(0, 1 - \frac{2}{n_+})}} \,, \qquad \forall \rho \in \R^{2d}, \forall t \ge 0 \,,
\end{equation*}
which gives the sought result.
\end{proof}

The result below roughly states that our dynamical condition is invariant under sub-principal perturbation of the potential $V$, under the assumption that $V$ is sub-quadratic.

\begin{coro} \label{cor:classicalobssubprincipalperturbation}
Fix $0 < n_1, n_2 \le 2$ and let $p_1, p_2 \in S^{n_1, n_2}$ be elliptic symbols in $S^{n_1, n_2}$, and assume they have the same principal symbol in the sense of~\eqref{eq:principalsymbolnonstandard}. Consider the Hamiltonian flows $(\phi_1^t)_{t \in \R}$ and $(\phi_2^t)_{t \in \R}$ associated with $p_1$ and $p_2$ respectively. For any $T > 0$, there exists a constant $C = C_T > 0$ such that the following holds: for any function $q =q(t; \rho)$, Lipschitz in $\rho$ and such that
\begin{equation*}
\supp q \subset [-T, T] \times \R^{2d} \,,
\end{equation*}
one has
\begin{equation*}
\abs*{\int_{\R} q\left(t; \phi_2^t(\rho)\right) \dd t - \int_{\R} q\left(t; \phi_1^t(\rho)\right) \dd t}
	\le C \norm*{\nabla_\rho q}_{L^\infty(\R \times \R^{2d})} \,, \qquad \forall \rho \in \R^{2d} \,.
\end{equation*}
In particular, it holds
\begin{equation*}
\abs*{\frak{K}_{p_2}^\infty(q) - \frak{K}_{p_1}^\infty(q)}
	\le C \norm*{\nabla_\rho q}_{L^\infty(\R \times \R^{2d})} \,.
\end{equation*}
\end{coro}

\begin{proof}
This is a direct application of the mean-value inequality and Proposition~\ref{prop:stabilityestimate}, observing that $n_+ = \max(n_1, n_2) \le 2$:
\begin{multline*}
\abs*{\int_{\R} q\left(t; \phi_2^t(\rho)\right) \dd t - \int_{\R} q\left(t; \phi_1^t(\rho)\right) \dd t} \\
	\le \int_{-T}^T \norm*{\nabla_\rho q}_{L^\infty(\R \times \R^{2d})} \abs*{\phi_2^t(\rho) - \phi_1^t(\rho)} \dd t
	\le 2 T \e^{C T} \norm*{\nabla_\rho q}_{L^\infty(\R \times \R^{2d})} \,.
\end{multline*}
Taking lower limits in $\rho$ yields the second claim.
\end{proof}

\subsection{Quantitative estimates of classical (non-)observability}

In this subsection, we show that $\one_{(0, T) \times B_r(0) \times \R^d}$ is not classically observable in the sense of Definition~\ref{def:classicalobs} when the Hamiltonian is of the form $p(x, \xi) = V(x) + \frac{1}{2} \abs{\xi}^2$. Actually for this class of Hamiltonians, we can prove a more precise result.

\begin{prop} \label{prop:timeincylinder} 
Let $p$ be a symbol of the form $p(x, \xi) = V(x) + \frac{1}{2} \abs{\xi}^2$, with $V$ fulfilling Assumption~\ref{assum:V} with an arbitrary $m > 0$.
\begin{itemize}[label=\textbullet]
\item If $m \ge 1/2$, there exists a constant $C > 0$ and $E_0 > 0$ such that for all $E \ge E_0$, one has
\begin{equation*}
\forall r \ge 0, \forall \rho \in \{p = E\} \,, \qquad
    \abs*{\set{t \in \left[0, E^{\frac{1}{2} (\frac{1}{m} - 1)}\right]}{\phi^t(\rho) \in B_r(0) \times \R^d}}
        \le C \dfrac{r}{\sqrt{E}} \,.
\end{equation*}
\item If $m < 1/2$, then for any $\eps > 0$ small enough, there exists a constant $C > 0$ and $E_0 > 0$ such that for all $E \ge E_0$, one has
\begin{equation*}
\forall r \ge 0, \forall \rho \in \{p = E\} \,, \qquad
    \abs*{\set{t \in \left[0, E^{\frac{1}{2} (\frac{1}{m} - 1) - \eps}\right]}{\phi^t(\rho) \in B_r(0) \times \R^d}}
        \le C \dfrac{r}{\sqrt{E}} \,.
\end{equation*}
\end{itemize}
\end{prop}

\begin{coro}[Classical non-observability] \label{cor:classicalnonobs}
Under the assumptions of the proposition above, one has the following:
\begin{itemize}[label=\textbullet]
\item If $m < 1$, then for any $T \ge 0$, there exists a constant $C > 0$ and $E_0 > 0$ such that for all $E \ge E_0$, one has
\begin{equation*}
\forall r \ge 0, \forall \rho \in \{p = E\} \,, \qquad
	\abs*{\set{t \in [0, T]}{\phi^t(\rho) \in B_r(0) \times \R^d}}
    	= C \dfrac{r}{\sqrt{E}} \,.
\end{equation*}
\item If $m \ge 1$, there exists a constant $C > 0$ and $E_0 > 0$ such that for all $E \ge E_0$ and for all $T \ge 0$, one has
\begin{equation*}
\forall r \ge 0, \forall \rho \in \{p = E\} \,, \qquad
	\abs*{\set{t \in [0, T]}{\phi^t(\rho) \in B_r(0) \times \R^d}}
    	\le C \dfrac{r (1 + T)}{E^{\frac{1}{2m}}} \,.
\end{equation*}
\end{itemize}
\end{coro}

\begin{rema}
The corollary implies in particular that when $r$ and $T$ are fixed, the function $\one_{(0, T) \times B_r(0) \times \R^d}$ is not classically observable in the sense of Definition~\ref{def:classicalobs}.
\end{rema}

Let us explain the meaning of the typical scales appearing in Proposition~\ref{prop:timeincylinder} and the subsequent corollary. When $V$ satisfies Assumption~\ref{assum:V} with an arbitrary $m > 0$, one can single out a typical time scale in the energy layer $\{p(\rho) = E\}$ of order $\tau \approx E^{\frac{1}{2} (\frac{1}{m} - 1)}$, which corresponds roughly speaking to the ``period" of the trajectories of the flow, or rather, to the time needed to go from one turning point of a projected trajectory to another. We observe that for the harmonic oscillator, one has $m = 1$, hence $\tau \approx 1$ is indeed independent of the energy layer. Following this observation, we understand the criticality of quadratic potentials in our problem: if $m > 1$, the typical time scale of evolution of the flow tends to zero as the energy goes to infinity, which means that the flow mixes the phase space more and more in the high-energy limit in a time interval of the form $[0, T]$ with $T > 0$ fixed. On the contrary, for $m < 1$, the flow gets nicer on such a time interval because $\tau \to + \infty$ as $E \to + \infty$. We also have a typical scale with respect to the space variable, which is $r \approx E^{\frac{1}{2m}}$. This is the approximate diameter of the classically allowed region $K_E = \set{x \in \R^d}{V(x) \le E}$. This scale also appears naturally when one looks for a trajectory $t \mapsto \phi^t(\rho) = (x^t(\rho), \xi^t(\rho))$ such that $\abs{x^t(\rho)} = \rm{constant}$ (think for instance of the case of radial potentials). Differentiating $\abs{x^t(\rho)}^2$ with respect to time, one gets $x^t(\rho) \cdot \xi^t(\rho) = 0$ for all $t$, and differentiating again leads to $\abs{\xi^t(\rho)}^2 - x^t(\rho) \cdot \nabla V(x^t(\rho)) = 0$. Yet $\abs{\nabla V(x^t(\rho))} \lesssim \abs{x^t(\rho)}^{2m - 1}$, and $p$ is preserved by the flow. From this we can deduce that $\abs{x^t(\rho)} \approx p(\rho)^{\frac{1}{2m}}$. So if $r$ is larger than $p(\rho)^{\frac{1}{2m}}$, such trajectories will always stay in $B_r(0) \times \R^d$. Finally, if $\rho_0 = (x_0,  \xi_0) \in \{p(\rho) = E\}$ is such that $\abs{x_0} \le r$, with $r \le \eps p(\rho)^{\frac{1}{2m}}$, $\eps$ being sufficiently small, the momentum of the trajectory verifies $\abs{\xi_0} \gtrsim \sqrt{p(\rho)}$. Therefore, we can expect that the measure of times $t \in [0, \tau]$ such that $\abs{x^t(\rho)} \lesssim r$ will be of order $r/\sqrt{p(\rho)}$.

The proof of Proposition~\ref{prop:timeincylinder} relies on the lemma below.

\begin{lemm} \label{lem:rootspolynomial}
Let $a, b, c > 0$. Let $I \subset \R$ be a measurable set such that
\begin{equation*}
\forall (t_1, t_2) \in I \times I \,, \qquad
	a \abs{t_2 - t_1}^2 - b \abs{t_2 - t_1} + c \ge 0 \,.
\end{equation*}
Then it holds
\begin{equation} \label{eq:betweentworoots}
\abs*{I \cap [0, \tau]}
    \le \dfrac{8 a c}{b^2} \tau \,, \qquad \forall \tau \ge \dfrac{b}{2a} \,.
\end{equation}
\end{lemm}

\begin{rema}
Observe that the left-hand side of~\eqref{eq:betweentworoots} is always bounded by $\tau$. Thus, the lemma is mainly relevant in the case where $a c \ll b^2$, in which case the discriminant of the polynomial $a X^2 - b X + c$ is positive.
\end{rema}

\begin{proof}[Proof of Lemma~\ref{lem:rootspolynomial}]
First assume that the discriminant of the polynomial $a X^2 - b X + c$ is positive. Denote by $z_- \le z_+$ the (real) roots of the polynomial. Then we have
\begin{equation*}
\dfrac{b}{2a} = \dfrac{z_+ + z_-}{2} \le z_+ \le z_+ + z_- = \dfrac{b}{a}
    \qquad \rm{and} \qquad
z_- = \dfrac{z_+ z_-}{z_+} = \dfrac{c/a}{z_+} \le \dfrac{2c}{b} \,.
\end{equation*}
Since $a > 0$, we deduce that any $t$ such that $a t^2 - b t + c \ge 0$ verifies
\begin{equation} \label{eq:tgreaterorlessthan}
t \le z_- \le \dfrac{2c}{b}
    \qquad \rm{or} \qquad
t \ge z_+ \ge \dfrac{b}{2a} \,.
\end{equation}
We deduce that
\begin{equation} \label{eq:estimateIzerotau}
\abs*{I \cap \left[0, \dfrac{b}{2a}\right]}
	\le \abs*{\set{t \in [0, z_+]}{a t^2 - b t + c \ge 0}}
	\le \abs*{[0, z_-]}
	\le \dfrac{2 c}{b} \,.
\end{equation}
Now if $\tau \ge b/2a$, we split the interval $[0, \tau]$ as follows:
\begin{equation*}
[0, \tau]
	= \bigcup_{k = 1}^n \left[\tfrac{k - 1}{n} \tau, \tfrac{k}{n} \tau\right] \,,
		\qquad \rm{with} \;\, n = \left\lceil \dfrac{\tau}{b/2a} \right\rceil \ge 1 \,.
\end{equation*}
On each piece, we have
\begin{equation*}
\abs*{I \cap \left[\tfrac{k - 1}{n} \tau, \tfrac{k}{n} \tau\right]}
	= \abs*{\left(I - \tfrac{k - 1}{n} \tau\right) \cap [0, \tfrac{1}{n} \tau]}
	\le  \abs*{\left(I - \tfrac{k - 1}{n} \tau\right) \cap \left[0, \dfrac{b}{2a}\right]} \,,
\end{equation*}
where the last inequality is due to the definition of $n$. We can apply~\eqref{eq:estimateIzerotau} with $I - \tfrac{k - 1}{n} \tau$ instead of $I$, since the former set satisfies the assumptions of the lemma. Then, summing over $k$ yields
\begin{equation*}
\abs*{I \cap [0, \tau]}
	\le n \dfrac{2 c}{b}
	\le \left(\dfrac{\tau}{b/2a} + 1\right) \dfrac{2 c}{b}
	\le \dfrac{8 a c}{b^2} \tau \,,
\end{equation*}
which is the desired estimate. Finally if the discriminant is nonpositive, i.e.\ $b^2 \le 4 a c$, then
\begin{equation*}
\abs*{I \cap [0, \tau]}
	\le \tau
	\le \dfrac{4 a c}{b^2} \tau \,,
\end{equation*}
which concludes the proof.
\end{proof}

\begin{proof}[Proof of Proposition~\ref{prop:timeincylinder}]
Let us write for short $E = p(\rho)$, and introduce the components of the flow $(x^t, \xi^t) = \phi^t(\rho)$. Assume $E > 0$. The core of the argument is to compare $x^t$ to the straight trajectory $t \mapsto x^0 + t \xi^0$, which is of course easier to handle. In order to have two distinct points of the initial trajectory to be in the ball $B_r(0)$, its distance to the straight trajectory has to be very small or very large, which is possible in a time interval which is either small or large respectively. Introduce
\begin{equation*}
I = I_{\rho, r}
	= \set{t \in \R}{x^t \in B_r(0)} \,.
\end{equation*}
This set is measurable. Moreover, for any $t_1 \le t_2$,
using the Hamilton equation and the Taylor Formula at order $1$ with integral remainder, one has
\begin{equation*}
x^{t_2} = x^{t_1} + (t_2 - t_1) \xi^{t_1} - (t_2 - t_1)^2 \int_0^1 (1 - s) \nabla V\left(x^{(1 - s) t_1 + s t_2}\right) \dd s \,.
\end{equation*}
Assume now that $t_1, t_2 \in I$. Then the inverse triangle inequality leads to
\begin{equation} \label{eq:ineqr}
2r
	\ge \abs{t_2 - t_1} \abs{\xi^{t_1}} - (t_2 - t_1)^2 \sup_{t \in [t_1, t_2]} \abs*{\nabla V\left(x^t\right)} \,.
\end{equation}
At this stage we have to estimate differently the term involving $\nabla V$, depending on whether $m$ is greater or less than $1/2$ (or roughly speaking on whether the potential is approximately convex of concave).

\medskip
\emph{Case $m \ge 1/2$.}
Using that $V$ satisfies Assumption~\ref{assum:V}, we have
\begin{equation*}
\abs*{\xi^{t_1}}
    = \sqrt{2 (E - V(x^{t_1}))}
    \ge \sqrt{\max(0, E - C \jap{r}^{2m})} \,,
\end{equation*}
for some constant $C \ge 1$. Moreover, one can roughly estimate the remainder using the triangle inequality:
\begin{equation*}
\sup_{t \in [t_1, t_2]} \abs*{\nabla V(x^t)}
    \le C \sup_{t \in [t_1, t_2]} \jap*{x^t}^{2m - 1} \,.
\end{equation*}
Now we take advantage of the fact that $V$ is elliptic: up to enlarging the constant $C$, one has
\begin{equation*}
- C + \dfrac{1}{C} \jap*{x}^{2m} \le V(x) \le V(x) + \dfrac{1}{2} \abs*{\xi}^2 \,, \qquad \forall (x, \xi) \in \R^{2d} \,.
\end{equation*}
Therefore if $E$ is large enough (say larger than $C$), we obtain $\jap{x^t}^{2m - 1} \le C E^{1 - \frac{1}{2m}}$, with a possibly larger constant $C$ (we use $m \ge 1/2$ here). Inequality~\eqref{eq:ineqr} then becomes
\begin{equation*}
2r \ge \abs{t_2 - t_1} \sqrt{\max(0, E - C \jap{r}^{2m})} - C E^{1 - \frac{1}{2m}} \abs*{t_2 - t_1}^2 \,.
\end{equation*}
Set
\begin{equation} \label{eq:defabctau}
a = C E^{1 - \frac{1}{2m}} \,, \quad
b = \sqrt{\max(0, E - C \jap{r}^{2m})} \,, \quad
c = 2r \quad \rm{and} \quad
\tau = E^{\frac{1}{2}(\frac{1}{m} - 1)} \,.
\end{equation}
We have $\tau \ge \frac{b}{2a}$ since we can assume that $C \ge 1$:
\begin{equation*}
\dfrac{b}{2a}
    = \dfrac{\sqrt{\max(0, E - C \jap{r}^{2m})}}{2 C} E^{\frac{1}{2m} - 1}
    \le \dfrac{1}{2C} E^{\frac{1}{2}(\frac{1}{m} - 1)}
    \le \tau \,.
\end{equation*}
With these notation, we have that any $t_1, t_2 \in I$ satisfy
\begin{equation*}
a \abs{t_2 - t_1}^2 - b \abs{t_2 - t_1} + c \ge 0 \,.
\end{equation*}
Therefore, assuming first that $C \jap{r}^{2m} \le E/2$, we have $b \ge \sqrt{E/2} > 0$, so that Lemma~\ref{lem:rootspolynomial} applies. We obtain
\begin{align*}
\abs*{I \cap \left[0, E^{\frac{1}{2}(\frac{1}{m} - 1)}\right]}
    &\le \dfrac{8 a c}{b^2} \tau
    \le \dfrac{8 C E^{1 - \frac{1}{2m}} \times 2r}{E/2} E^{\frac{1}{2}(\frac{1}{m} - 1)}
    = 32 C \dfrac{r}{\sqrt{E}} \,.
\end{align*}
If on the contrary we have $\jap{r} \ge (E/2C)^{\frac{1}{2m}}$, as soon as $E \ge 2^{2m + 1} C$ we have $r \ge \frac{1}{2} (E/2C)^{\frac{1}{2m}}$, and we check that
\begin{equation*}
\abs*{I \cap \left[0, E^{\frac{1}{2}(\frac{1}{m} - 1)}\right]}
    \le E^{\frac{1}{2}(\frac{1}{m} - 1)}
    = \dfrac{r}{\sqrt{E}} \times \dfrac{E^{\frac{1}{2m}}}{r}
    \le \dfrac{r}{\sqrt{E}} \times 2^{1 + \frac{1}{2m}} C^{\frac{1}{2m}} \,.
\end{equation*}
This is valid for any $r > 0$, but in fact $r = 0$ works as well since $B_0(0) = \varnothing$. In addition, this is independent of the point $\rho \in \{p = E\}$, whence the result.

\medskip
\emph{Case $m < 1/2$.}
In the situation where the potential is ``sublinear", the inequality $\jap{x^t}^{2m - 1} \lesssim E^{1 - \frac{1}{2m}}$ is false in general since the power $2m - 1$ is nonpositive (such an inequality would require $V(x^t)$ to be controlled from below by $E$, which is possible near turning points of the trajectory but not in the well). Thus, a priori we can only have $\abs{\nabla V(x^t)} \le C$, which leads to
\begin{align} \label{eq:ineqstep0}
2r
    &\ge \abs{t_2 - t_1} \abs{\xi^{t_1}} - C \abs{t_2 - t_1}^2
    \ge \abs{t_2 - t_1} \sqrt{\max(0, E - C \jap{r}^{2m})} - C \abs{t_2 - t_1}^2 \nonumber\\
    &\ge \abs{t_2 - t_1} \sqrt{\max(0, E - C \jap{r})} - C \abs{t_2 - t_1}^2 \,.
\end{align}
This coincides with the previous case for the critical value $m = 1/2$: for any $t_1, t_2 \in I$, it holds
\begin{equation*}
a \abs{t_2 - t_1}^2 - b \abs{t_2 - t_1} + c \ge 0 \,,
\end{equation*}
where $a, b, c$ are defined in~\eqref{eq:defabctau} (with $m = 1/2$). Then, the first step of the proof tells us that there exists $C > 0$ such that for all $E$ large enough, it holds
\begin{equation} \label{eq:cylinderstep0}
\abs*{\set{t \in \left[0, \sqrt{E}\right]}{\phi^t(\rho) \in B_r(0) \times \R^d}}
    \le C \dfrac{r}{\sqrt{E}} \,, \qquad \forall r \ge 0, \forall \rho \in \{p = E\} \,.
\end{equation}
We shall use this additional information to improve~\eqref{eq:ineqstep0}, and then bootstrap this procedure to reach the critical time $E^{\frac{1}{2} (\frac{1}{m} - 1)}$. We will work this out by induction, taking~\eqref{eq:cylinderstep0} as our basis step. Consider $n \ge 0$ and suppose there exist $\gamma_n \in [\frac{1}{2}, \frac{1}{2}(\frac{1}{m} - 1))$ and $C_n \ge 1$ such that when $E$ is large enough, one has
\begin{equation} \label{eq:cylinderstepn}
\abs*{\set{t \in \left[0, E^{\gamma_n}\right]}{\phi^t(\rho) \in B_r(0) \times \R^d}}
    \le C_n \dfrac{r}{\sqrt{E}} \,, \qquad \forall r \ge 0, \forall \rho \in \{p = E\} \,.
\end{equation}
We first deduce from the Taylor formula a bound slightly more precise than~\eqref{eq:ineqr}:
\begin{align} \label{eq:ineqrbis}
2r
	&\ge \abs{t_2 - t_1} \abs{\xi^{t_1}} - \abs{t_2 - t_1} \int_{t_1}^{t_2} \abs*{\nabla V\left(x^t\right)} \dd t \nonumber\\
	&\ge \abs{t_2 - t_1} \sqrt{\max(0, E - C \jap{r}^{2m})} - \abs{t_2 - t_1} \int_{t_1}^{t_2} \abs*{\nabla V\left(x^t\right)} \dd t \,.
\end{align}
Take $\delta \in [0, 1]$ to be chosen later. We have
\begin{align} \label{eq:secondtermtobeestimated}
\int_{t_1}^{t_2} \abs*{\nabla V(x^t)} \dd t
    &\le C \int_{t_1}^{t_2} \jap*{x^t}^{2m - 1} \dd t
    = C \int_0^{+\infty} \left( \int_{t_1}^{t_2} \one_{u \le \jap{x^t}^{2m - 1}} \dd t \right) \dd u \nonumber\\
    &\le C \int_0^{+\infty} \abs*{\set{t \in [t_1, t_2]}{u \le \abs{x^t}^{2m - 1}}} \dd u \nonumber\\
    &\le C \int_0^{E^{\delta(1 - \frac{1}{2m})}} \abs{t_2 - t_1} \dd u
    	+ C \int_{E^{\delta(1 - \frac{1}{2m})}}^{+\infty} \abs*{\set{t \in [t_1, t_2]}{\abs{x^t} \le u^{- \frac{1}{1 - 2m}}}} \dd u \,.
\end{align}
The first inequality follows from our assumptions on $V$, the equality is a consequence of Fubini's Theorem, then we use that $2m - 1 \le 0$ to deduce $\jap{x^s}^{2m - 1} \le \abs{x^s}^{2m - 1}$, and finally we split the integral over $u$ into two pieces. To estimate the second piece, we split the interval $[t_1, t_2]$ into $N = \lceil \frac{\abs{t_2 - t_1}}{E^{\gamma_n}} \rceil$ intervals of length less than $E^{\gamma_n}$. On the $k$-th piece, we use the induction hypothesis~\eqref{eq:cylinderstepn}, with $\rho_k = \phi^{t_1 + \frac{k-1}{N} \abs{t_2 - t_1}}(\rho)$ instead of $\rho$, namely setting $\tilde t_k = t_1 + \frac{k-1}{N} \abs{t_2 - t_1}$, it holds
\begin{equation*}
\abs*{\set{t \in [\tilde t_k, \tilde t_{k+1}]}{\abs{x^{t}} \le u^{- \frac{1}{1 - 2m}}}}
	\le \abs*{\set{s \in [0, E^{\gamma_n}]}{\abs{x^{s + \tilde t_k}} \le u^{- \frac{1}{1 - 2m}}}}
	\le \dfrac{C_n}{\sqrt{E}} u^{- \frac{1}{1 - 2m}} \,.
\end{equation*}
Summing over $k \in \{1, 2, \ldots, N\}$ yields
\begin{equation*}
\abs*{\set{t \in [t_1, t_2]}{\abs{x^t} \le u^{- \frac{1}{1 - 2m}}}}
    \le \dfrac{C_n}{\sqrt{E}} u^{- \frac{1}{1 - 2m}} \left\lceil \dfrac{\abs{t_2 - t_1}}{E^{\gamma_n}} \right\rceil \,,
\end{equation*}
provided $E$ is large enough.
%If we assume that $\abs{t_2 - t_1} \ge E^{\gamma_n}$, we get
%\begin{equation*}
%    \abs*{\set{t \in [t_1, t_2]}{\abs{x^t} \le u^{- \frac{1}{1 - 2m}}}}
%    \le 2 \dfrac{C_n}{\sqrt{E}} u^{- \frac{1}{1 - 2m}} \dfrac{\abs{t_2 - t_1}}{E^{\gamma_n}} \,.
%\end{equation*}
Integrating over $u$, we obtain a bound for the second term in~\eqref{eq:secondtermtobeestimated}:
\begin{align*}
\int_{E^{\delta(1 - \frac{1}{2m})}}^{+\infty} \abs*{\set{t \in [t_1, t_2]}{\abs{x^t} \le u^{- \frac{1}{1 - 2m}}}} \dd u
    &\le \dfrac{C_n}{E^{1/2}} \left(\dfrac{\abs{t_2 - t_1}}{E^{\gamma_n}} + 1\right) \int_{E^{\delta(1 - \frac{1}{2m})}}^{+\infty} u^{- \frac{1}{1 - 2m}} \dd u \\
    &= \dfrac{C_n}{E^{1/2}} \left(\dfrac{\abs{t_2 - t_1}}{E^{\gamma_n}} + 1\right) \times \dfrac{-1}{1 - \frac{1}{1 - 2m}} E^{\delta(1 - \frac{1}{2m})(1 - \frac{1}{1 - 2m})} \\
    &= \left(\dfrac{1/2}{m} - 1\right) \times \dfrac{C_n}{E^{1/2}} \left(\dfrac{\abs{t_2 - t_1}}{E^{\gamma_n}} + 1\right) E^\delta \,.
\end{align*}
In the end we obtain
\begin{equation*}
\int_{t_1}^{t_2} \abs*{\nabla V(x^t)} \dd t
    \le \dfrac{C}{2} \abs{t_2 - t_1} \left( E^{\delta(1 - \frac{1}{2m})} + E^{\delta - \frac{1}{2} - \gamma_n} \right) + C E^{\delta - \frac{1}{2}} \,,
\end{equation*}
for some constant $C > 0$. By choosing $\delta = m (2 \gamma_n + 1)$ (we have indeed $\delta \in [2m, 1) \subset [0, 1)$ when $\gamma_n \in [\frac{1}{2}, \frac{1}{2}(\frac{1}{m} - 1))$), we obtain
\begin{equation*}
\int_{t_1}^{t_2} \abs*{\nabla V(x^t)} \dd t
    \le C \abs{t_2 - t_1} E^{(2m - 1) \gamma_n + m - \frac{1}{2}} + C E^{\delta - \frac{1}{2}} \,.
\end{equation*}
Going back to~\eqref{eq:ineqrbis}, if $t_1, t_2 \in I$, i.e.\ $x^{t_1}$ and $x^{t_2}$ lie in $B_r(0)$, we deduce
\begin{equation*}
2r \ge \abs{t_2 - t_1} \left(\sqrt{\max(0, E - C \jap{r}^{2m})} - C E^{\delta - \frac{1}{2}}\right) - C E^{\frac{1}{2} - \gamma_{n+1}} \abs{t_2 - t_1}^2 \,,
\end{equation*}
where we set $\gamma_{n+1} = (1 - 2m) \gamma_n + 1 - m$.
%We deduce that
%\begin{multline*}
%\abs*{\set{t \in [0, E^{\gamma_{n+1}}]}{x^t \in B_r(0)}}
%	\le \abs*{\set{t \in [0, E^{\gamma_n}]}{x^t \in B_r(0)}} \\ + \abs*{\set{t \in [0, E^{\gamma_{n+1}}]}{a t^2 - bt + c \ge 0}} \,,
%\end{multline*}
Now set
\begin{equation*}
a = C E^{\frac{1}{2} - \gamma_{n+1}} \,, \qquad
b = \sqrt{\max(0, E - C \jap{r}^{2m})} - C E^{\delta - \frac{1}{2}} \qquad \rm{and} \qquad
c = 2r \,.
\end{equation*}
Assuming first that $C \jap{r}^{2m} \le E/2$ and recalling that $\delta < 1$, we know that for $E$ large enough, we have $b \ge \sqrt{E/3}$. Any $t_1, t_2 \in I$ satisfy
\begin{equation*}
a \abs{t_2 - t_1}^2 - b \abs{t_2 - t_1} + c \ge 0 \,,
\end{equation*}
so we apply Lemma~\ref{lem:rootspolynomial} with $\tau = E^{\gamma_{n+1}} \ge \frac{b}{2a}$ to get
\begin{equation*}
\abs*{I \cap [0, E^{\gamma_{n+1}}]}
	\le \dfrac{8 a c}{b^2} E^{\gamma_{n+1}}
	\le \dfrac{16 C E^{1/2}}{E/3} r
	= \dfrac{48 C}{\sqrt{E}} r \,.
\end{equation*}
When $C \jap{r}^{2m} \ge E/2$, assuming that $E$ is large enough we have $r \ge \frac{1}{2} (E/2C)^{\frac{1}{2m}}$ and we conclude as in the previous step that
\begin{equation*}
\abs*{\set{t \in [0, E^{\gamma_{n+1}}]}{x^t \in B_r(0)}}
    \le \dfrac{r}{\sqrt{E}} \dfrac{E^{\gamma_{n+1} + \frac{1}{2}}}{r}
    \le \dfrac{r}{\sqrt{E}} 2^{1 + \frac{1}{2m}} C^{\frac{1}{2m}} E^{\gamma_{n+1} - \frac{1}{2}(\frac{1}{m} - 1)} \,.
\end{equation*}
Since by the induction hypothesis we have $\gamma_n \in [\frac{1}{2}, \frac{1}{2}(\frac{1}{m} - 1))$, then $\gamma_{n+1}$ belongs to the same interval because by definition, $\gamma_{n+1} \ge 1 - m \ge \frac{1}{2}$, and we have
\begin{equation} \label{eq:sequencegamman}
\dfrac{\gamma - \gamma_{n+1}}{1/2}
    = (1 - 2m) \dfrac{\gamma - \gamma_n}{1/2}
    \qquad \qquad \rm{where} \; \, \gamma = \dfrac{1}{2}\left(\dfrac{1}{m} - 1\right) \,.
\end{equation}
Therefore we see that $\gamma_{n+1} - \gamma < 0$, so as soon as $E$ is large enough, it holds
\begin{equation*}
\abs*{\set{t \in [0, E^{\gamma_{n+1}}]}{x^t \in B_r(0)}}
    \le C_{n+1} \dfrac{r}{\sqrt{E}}
\end{equation*}
for any $r \ge 0$, and for some constant $C_{n+1}$. Thus we have constructed by induction a non-decreasing sequence $(\gamma_n)_{n \in \N}$ for which~\eqref{eq:cylinderstepn} holds. We deduce from~\eqref{eq:sequencegamman} that it converges to $\gamma = \frac{1}{2}(\frac{1}{m} - 1)$, which yields the final result.
\end{proof}

\begin{proof}[Proof of Corollary~\ref{cor:classicalnonobs}]
Firstly we treat the case where $m < 1$. For $\eps$ small enough, $E^{\frac{1}{2}(\frac{1}{m} - 1) - \eps}$ tends to $+ \infty$ as $E \to + \infty$, so we can write using Proposition~\ref{prop:timeincylinder}:
\begin{equation*}
\abs*{\set{t \in [0, T]}{\phi^t(\rho) \in B_r(0) \times \R^d}}
    \le \abs*{\set{t \in \left[0, E^{\frac{1}{2}(\frac{1}{m} - 1) - \eps}\right]}{\phi^t(\rho) \in B_r(0) \times \R^d}}
    \le C \dfrac{r}{\sqrt{E}} \,,
\end{equation*}
provided $E$ is large enough, for all $\rho \in \{p = E\}$ and all $r \ge 0$.
Now in the case where $m \ge 1$, we know that $E^{\frac{1}{2}(\frac{1}{m} - 1)}$ remains bounded as $E \to + \infty$. By Proposition~\ref{prop:timeincylinder} again, there is a $E_0 > 0$ such that for any $E \ge E_0$, it holds
\begin{equation} \label{eq:estimatetimepiece}
\abs*{\set{t \in \left[0, E^{\frac{1}{2}(\frac{1}{m} - 1)}\right]}{\phi^t(\rho) \in B_r(0) \times \R^d}}
	\le C \dfrac{r}{\sqrt{E}} \,,
\end{equation}
whenever $r \ge 0$ and $\rho \in \{p(\rho) = E\}$. Let $n = \lceil \frac{T}{E^{\frac{1}{2}(\frac{1}{m} - 1)}} \rceil$. Writing $t_k = \frac{k}{n} T$ and $\rho_k = \phi^{t_k}(\rho)$ for any $k \in \{0, 1, \ldots, n\}$, we have
\begin{align*}
\abs*{\set{t \in [0, T]}{\phi^t(\rho) \in B_r(0) \times \R^d}}
	&\le \sum_{k = 1}^n \abs*{\set{t \in [t_{k-1}, t_k]}{\phi^t(\rho) \in B_r(0) \times \R^d}} \\
	&= \sum_{k = 1}^n \abs*{\set{t \in [0, \tfrac{1}{n} T]}{\phi^{t + t_{k-1}}(\rho) \in B_r(0) \times \R^d}} \\
	&\le \sum_{k = 1}^n \abs*{\set{t \in \left[0, E^{\frac{1}{2}(\frac{1}{m} - 1)}\right]}{\phi^t(\rho_{k-1}) \in B_r(0) \times \R^d}} \,.
\end{align*}
The last inequality comes from the definition of $n$. Estimate~\eqref{eq:estimatetimepiece} applies to each piece of this sum. We conclude that
\begin{equation*}
\abs*{\set{t \in [0, T]}{\phi^t(\rho) \in B_r(0) \times \R^d}}
	\le n C \dfrac{r}{\sqrt{E}}
	\le \dfrac{1 + T}{E^{\frac{1}{2}(\frac{1}{m} - 1)}} \times C \dfrac{r}{\sqrt{E}}
    = C \dfrac{r (1 + T)}{E^{\frac{1}{2m}}}
\end{equation*}
(we can ensure that $n \le \frac{1+T}{E^{\frac{1}{2}(\frac{1}{m} - 1)}}$ in the second equality up to enlarging $E_0$ so that it is larger than $1$, independently of $T$). This completes the proof.
\end{proof}

\subsection{Continuity of the composition by the flow in symbol classes}

From now on we go back to a \emph{subquadratic} potential, that is to say we suppose our classical Hamiltonian is of the form $p(x, \xi) = V(x) + \frac{1}{2} \abs{\xi}^2$, with $V$ satisfying Assumption~\ref{assum:V} with $m \in (0, 1]$. In the course of our study, we will need to check that the composition of a symbol with the Hamiltonian flow is still well-behaved in a suitable symbol class, in the sense that its derivatives remain controlled properly. The following lemma is common in the context of the quantum-classical correspondence: see for instance Bouzouina and Robert~\cite[Lemma 2.2]{BouzouinaRobert}. We reproduce a proof to obtain an estimate adapted to our context and to keep track of the dependence of constants on the parameters of the problem. We recall that a function $a \in \cont^\infty(\R^{2d})$ is said to be a symbol in the class $S(1)$ if all its derivatives are bounded. The quantities
\begin{equation*}
\abs*{a}_{S(1)}^\ell
	= \max_{\substack{\alpha \in \N^{2d} \\ 0 \le \abs{\alpha} \le \ell}} \sup_{\rho \in \R^{2d}} \abs*{\partial^\alpha a(\rho)} \,,
		\qquad \ell \in \N \,,
\end{equation*}
endow the vector space $S(1)$ with a Fréchet structure (see Appendix~\ref{app:pseudo} for further details).

\begin{lemm} \label{lem:symbolcomposition}
Let $a$ be a symbol in $S(1)$. Then the function $a \circ \phi^t$ still belongs to $S(1)$, and stays in a bounded subset of $S(1)$ locally uniformly with respect to $t$. More precisely, for any fixed $T > 0$, for any nonzero multi-index $\alpha \in \N^{2d}$, we have the derivative estimate
\begin{equation*}
\norm*{\partial^\alpha (a \circ \phi^t)}_\infty
    \le C_{\alpha}(T, p) \max_{1 \le \abs{\beta} \le \abs{\alpha}} \norm*{\partial^\beta a}_\infty \,,
\end{equation*}
uniformly in $t \in [-T, T]$. The constants $C_{\alpha}(T, p)$ depend only on $T$ and on the sup-norm of derivatives of order $\{2, 3, \ldots, \abs{\alpha} + 1\}$ of $p$.
\end{lemm}

\begin{proof}
In all the proof, $t$ ranges in a compact set, say $[-T, T]$ for some fixed $T > 0$.

\medskip
\emph{Step 1 \--- Control of differentials of the Hamiltonian flow.}
Differentiating the Hamilton equation~\eqref{eq:defHamiltonianflow} defining the flow $\phi^t$, we get
\begin{equation*}
\dfrac{\dd}{\dd t} \dd \phi^t(\rho) = J \hess p\left(\phi^t(\rho)\right) \dd \phi^t(\rho) \,.
\end{equation*}
By assumption on the potential $V$ (see~\eqref{eq:assumgrowthderivativesV}), we observe that the Hessian of $p$ is bounded. Since $\dd \phi^0(\rho) = \id$ for any $\rho \in \R^{2d}$, we classically deduce using Grönwall's Lemma that
\begin{equation*}
\norm*{\dd \phi^t(\rho)}
    \le \e^{T \abs*{J \hess p}_\infty}
    \le \e^{T \abs*{\hess p}_\infty} \,,
        \qquad \forall \rho \in \R^{2d}, \forall t \in [-T, T] \,.
\end{equation*}
For higher order differentials, we proceed by induction. Suppose that for some $k \ge 1$, all the differentials of order $\le k$ of $\phi^t$ are bounded uniformly in $t$ on $\R^{2d}$, with a bound involving derivatives of order $k+1$ of $p$. Differentiating the Hamilton equation $k + 1$ times, the Faà di Bruno formula shows that $\frac{\dd}{\dd t} \dd^{k+1} \phi^t(\rho)$ is a sum of terms of the form:
\begin{equation*}
J \dd^\ell (\nabla p)\left(\phi^t(\rho)\right). \left(\dd^{k_1} \phi^t(\rho), \dd^{k_2} \phi^t(\rho), \ldots, \dd^{k_\ell} \phi^t(\rho)\right) \,,
\end{equation*}
where $1 \le \ell \le k + 1$ and $k_1 + k_2 + \cdots + k_\ell = k + 1$. Such terms are bounded uniformly in $t \in [-T, T]$ by the induction hypothesis as soon as $\ell \ge 2$ (note that all the differentials of order $\ge 2$ of $p$ are bounded). So in fact the ODE on $\dd^{k+1} \phi^t(\rho)$ can be written
\begin{equation*}
    \dfrac{\dd}{\dd t} \dd^{k+1} \phi^t(\rho) = J \hess p\left(\phi^t(\rho)\right) \dd^{k+1} \phi^t(\rho) + R(t, \rho) \,,
\end{equation*}
where $R(t, \rho)$ satisfies
\begin{equation*}
\abs*{R(t, \rho)}_\infty \le C(T, p) \,, \qquad \forall \rho \in \R^{2d}, \forall t \in [-T, T] \,,
\end{equation*}
where the constant $C(T, p)$ depends only on the sup-norm of derivatives of order $\{2, 3, \ldots, k + 2\}$ of $p$. We conclude by Grönwall's Lemma again, together with Duhamel's Formula that $\dd^{k+1} \phi^t(\rho)$ is bounded similarly: given that $k + 1 \ge 2$, we have $\dd^{k+1} \phi^0(\rho) = 0$ for every $\rho \in \R^{2d}$, so that
\begin{equation*}
\norm*{\dd^{k+1} \phi^t(\rho)}
	\le \int_0^t C(T, p) \e^{\abs{\hess p}_\infty (t - s)} \dd s
	\le T C'(T, p) \,.
\end{equation*}
This finishes the induction.

\medskip
\emph{Step 2 \--- Estimates of derivatives of $a \circ \phi^t$.}
We estimate the derivatives in $x$ or $\xi$. Let $\alpha \in \N^{2d} \setminus \{0\}$, and denote by $(x_1^t, x_2^t, \ldots, x_d^t, \xi_1^t, \xi_2^t \ldots, \xi_d^t)$ the components of the flow. The chain rule together with the Faà di Bruno formula yield that $\partial^\alpha (a \circ \phi^t)$ can be expressed as a sum of terms of the form
\begin{equation*}
(\partial_x^{\tilde \alpha} \partial_\xi^{\tilde \beta} a) \circ \phi^t \times \prod_{j_1 \in \tilde \alpha} \partial^{\alpha_{j_1}} x_{j_1} \times \prod_{j_2 \in \tilde \beta} \partial^{\beta_{j_2}} \xi_{j_2} \,,
\end{equation*}
where $\tilde \alpha, \tilde \beta \in \N^d$ are such that $1 \le \abs{\tilde \alpha} + \abs{\tilde \beta} \le \abs{\alpha}$ and $\alpha_{j_1}, \beta_{j_2} \in \N^{2d} \setminus \{0\}$ satisfy $\sum_{j_1} \alpha_{j_1} + \sum_{j_2} \beta_{j_2} = \alpha$. (By $j_1 \in \tilde \alpha$, $j_2 \in \tilde \beta$, we mean that $j_1, j_2 \in \{1, 2, \ldots, d\}$ are indices for which $\tilde \alpha$ and $\tilde \beta$ are nonzero.) The claim follows immediately from the bounds on the derivatives of $x_j^t$ and $\xi_j^t$ proved in Step 1.
\end{proof}

\section{Proof of the main theorem} \label{sec:proofmain}

We start with a lemma that will enable us to replace $\one_{\omega_R \setminus B_r(0)}$ in the observability inequality with a well-behaved symbol.

\begin{lemm}[Mollifying the observation set] \label{lem:mollifyingobsset}
Let $\omega \subset \R^d$ and denote by $\omega_R$ the open set
\begin{equation*}
\omega_R = \bigcup_{x \in \omega} B_R(x) \,, \qquad R > 0 \,.
\end{equation*}
There exists a symbol $a = a_R \in S(1)$ depending only on the $x$ variable such that
\begin{equation*}
\one_{\omega_{R/2}}(x) \le a_R(x) \le \one_{\omega_R}(x) \,, \qquad \forall x \in \R^d \,.
\end{equation*}
In addition, it satisfies the seminorm estimates:
\begin{equation*}
\forall \ell \in \N, \exists C_\ell > 0 : \forall R \ge 1 \,, \qquad
	\abs*{a_R}_{S(1)}^\ell 
		\le C_\ell
			\qquad \rm{and} \qquad
	\abs*{\nabla a_R}_{S(1)}^\ell
		\le \dfrac{C_\ell}{R} \,.
\end{equation*}
The constants involved do not depend on $\omega$.
\end{lemm}

\begin{proof}
Fix $\kappa \in \cont_\comp^\infty(\R^d)$ a mollifier with the following properties:
\begin{equation*}
\kappa(x) \ge 0, \,, \forall x \in \R^d \,, \qquad
\supp \kappa \subset B_1(0) \qquad \rm{and} \qquad
\int_{\R^d} \kappa(x) \dd x = 1 \,.
\end{equation*}
For any $r > 0$, set $\kappa_r = r^{-d} \kappa(\bigcdot/r)$, so that $\norm{\kappa_r}_{L^1(\R^d)} = 1$. Set for any $R > 0$:
\begin{equation*}
a_R(x) = \left(\kappa_{\frac{1}{4} R} \ast \one_{\omega_{\frac{3}{4} R}}\right)(x) \,, \qquad \forall x \in \R^d \,.
\end{equation*}
We check that $a_R$ defined in this way satisfies the required properties. We first observe that, by definition, $a_R$ is non-negative, and that $a_R \le 1$ by Young's inequality. Now by standard properties of convolution, the support of $a_R$ is contained in $\omega_{\frac{3}{4} R} + B_{\frac{1}{4} R}(0) \subset \omega_R$ (recall that the support of $\kappa$ is a compact subset of $B_1(0)$), which proves that $a_R \le \one_{\omega_R}$. On the other hand, if $x \in \omega_{R/2}$, then $\kappa_{\frac{1}{4} R}(x - \bigcdot)$ is supported in $\omega_{\frac{3}{4} R}$, so that $a_R(x) = 1$, which proves that $a_R \ge \one_{\omega_{R/2}}$. Differentiating under the integral sign, we see that $\norm*{\partial^{\alpha} a_R}_\infty$ is of order $1/R^{\abs{\alpha}}$ for any multi-index $\alpha \in \N^d$, which yields the desired seminorm estimates ($R \ge 1$ is important here). The constants depend only on the supremum norms of derivatives of $\kappa$, and not on $\omega$.
\end{proof}

\begin{rema}
The symbol $a_R$ can be considered as a semiclassical symbol, with Planck parameter $1/R^2$, since by construction each derivative yields a gain of $1/R$. However in view of Lemma~\ref{lem:symbolcomposition}, this property is not preserved by composition by the Hamiltonian flow, since all the derivatives of $a_R^ \circ \phi^t$ of order $\ge 1$ behave as $1/R$. This comes from the fact that, when differentiating $a_R \circ \phi^t$ twice or more, the second, third, and higher order derivatives can hit $\phi^t$ instead of $a_R$.
\end{rema}

We prove a version of Egorov's theorem taking into account the above remark. Our approach is very classical; see~\cite{BouzouinaRobert} or~\cite[Chapter 11]{Zworski:book} for refinements. We refer again to Appendix~\ref{app:pseudo} for an account on the Weyl quantization $\quantization$.

\begin{prop}[Egorov] \label{prop:Egorovthm}
Let $a \in S(1)$. Then the symbol $a \circ \phi^t$ lies in $S(1)$ with seminorm estimates
\begin{equation*}
\forall T > 0, \forall \ell \in \N, \exists C_\ell(T, p) > 0 : \qquad
	\abs*{a \circ \phi^t}_{S(1)}^\ell
    	\le C_\ell(T, p) \abs*{a}_{S(1)}^\ell \,, \quad \forall t \in [- T, T] \,,
\end{equation*}
and, it holds
\begin{equation} \label{eq:Egoroveq}
\e^{\ii t P} \Op{a} \e^{- \ii t P}  = \Op{a \circ \phi^t} + \cal{R}_a(t) \,,
\end{equation}
where the remainder term $\cal{R}_a(t)$ is a bounded operator satisfying
\begin{equation*}
\forall T > 0, \exists C(T, p) > 0 : \qquad
	\norm*{\cal{R}_a(t)}_{L^2 \to L^2} \le C(T, p) \abs*{\nabla a}_{S(1)}^{k_d} \,, \quad \forall t \in [-T, T] \,,
\end{equation*}
for some integer $k_d$ depending only on the dimension.
\end{prop}

\begin{proof}
The claim that $a \circ \phi^t \in S(1)$ and the subsequent seminorm estimates are provided by Lemma~\ref{lem:symbolcomposition}. To prove~\eqref{eq:Egoroveq}, we follow the classical method that consists in differentiating the time dependent operator
\begin{equation*}
Q(s) = \e^{- \ii s P} \Op{a \circ \phi^s} \e^{\ii s P} \,,
\end{equation*}
and estimating this derivative. For the sake of simplicity, let us introduce $a_s = a \circ \phi^s$. All the operators in this composition map $\sch(\R^d)$ to itself continuously, so that $Q(s) u$ can be differentiated using the chain rule, for any $u \in \sch(\R^d)$. From now on, we will omit to write $u$. Recalling that, by definition of $\phi^s$, we have $\frac{\dd}{\dd s} a_s = \poiss*{p}{a_s}$, it holds
\begin{equation*}
\frac{\dd}{\dd s} \Op{a_s} = \Op{\poiss*{p}{a_s}}
\end{equation*}
(rigorously, one may apply the Dominated Convergence Theorem to the pairing $\brak*{v}{\Op{a_s} u}_{\sch', \sch(\R^d)}$ for two Schwartz functions $u$ and $v$). Therefore we get
\begin{equation} \label{eq:derivativeQ(s)}
\dfrac{\dd}{\dd s} Q(s)
	= - \ii \e^{- \ii s P} \left( \comm*{P}{\Op{a_s}} + \ii \Op{\poiss*{p}{a_s}} \right) \e^{\ii s P}
	= - \ii \e^{- \ii s P} \Op{\cal{R}_3(s)} \e^{\ii s P} \,.
\end{equation}
The symbol $\cal{R}_3(s)$ above is nothing but the remainder of order $3$ in the pseudodifferential calculus between $p$ and $a_s$. Proposition~\ref{prop:refinedestimateremainderpseudocalc} provides a bound on this remainder in terms of seminorms of $a_s$. Recall that, in the subcritical case $m \le 1$, $\partial^\alpha p \in S(1)$ whenever $\abs{\alpha} \ge 2$. Therefore according to Proposition~\ref{prop:refinedestimateremainderpseudocalc}, for any seminorm index $\ell \in \N$, there exist a constant $C_\ell > 0$ as well as an integer $k \ge 0$ such that
\begin{equation*}
\abs*{\cal{R}_3(s)}_{S(1)}^\ell \le C_\ell \abs*{d^3 a_s}_{S(1)}^k \abs*{d^3 p}_{S(1)}^k \,.
\end{equation*}
Then we use Lemma~\ref{lem:symbolcomposition} to obtain
\begin{equation*}
\abs*{\cal{R}_3(s)}_{S(1)}^\ell
    \le C_\ell\left(T, p\right) \abs*{\nabla a}_{S(1)}^k \,,
\end{equation*}
for any $s \in [- T, T]$. Therefore, the Calder\'{o}n-Vaillancourt Theorem (Theorem~\ref{thm:CV}) tells us that the norm of $\Op{\cal{R}_3(s)}$ is bounded, uniformly in $s \in [-T, T]$, by a seminorm of $\nabla a$, and a constant depending only on $T$ and $p$. Plugging this into~\eqref{eq:derivativeQ(s)}, given that the propagator $\e^{\ii s P}$ is an isometry, we obtain the same bound on $\frac{\dd}{\dd s} Q(s)$. Integrating this in $s$, we obtain from the mean-value inequality
\begin{equation*}
\forall t \in [-T, T] \,, \qquad
	\norm*{Q(t) - Q(0)}_{L^2 \to L^2}
		\le 2 T \sup_{s \in [-T, T]} \norm*{\dfrac{\dd}{\dd s} Q(s)}_{L^2 \to L^2}
	\le C(T, p) \abs*{\nabla a}_{S(1)}^{k_d} \,,
\end{equation*}
where the integer $k_d$ depends only on the dimension. Conjugating by the propagator, which is an isometry on $L^2$, yields the desired result.
\end{proof}

We are now in a position to prove our main result.

\begin{proof}[Proof of Theorem~\ref{thm:main}]
We fix $\omega \subset \R^d$, a compact set $K \subset \R^d$, and we introduce $\tilde \omega(R) = (\omega \setminus K_R)_R$, for $R > 0$. One can verify that $\tilde \omega(R) \subset \omega_R \setminus K$. By Lemma~\ref{lem:mollifyingobsset}, there exists a symbol $a_R \in S(1)$ depending on the parameter $R > 0$ such that
\begin{equation} \label{eq:symbolaR}
\one_{(\omega \setminus K_R) \times \R^d}
	\le a_R \le \one_{\tilde \omega(R) \times \R^d} \,,
		\qquad \forall R > 0 \,,
\end{equation}
and $\abs{\nabla a_R}_{S(1)}^\ell \le c_{d,\ell}/R$ for any $\ell \in \N$, with a constant $c_{d,\ell}$ depending only on the dimension and $\ell$, uniformly in $R \ge 1$. Notice that the symbol depends on $\omega$ and $K$ but not its seminorms. On the quantum side, one can regard the functions in~\eqref{eq:symbolaR} as multiplication operators, and understand the inequalities in the sense of self-adjoint operators. Conjugating by the Schrödinger propagator does not change the inequalities, so that:
\begin{equation*}
\e^{\ii t P} \one_{\omega \setminus K_R} \e^{- \ii t P}
	\le \e^{\ii t P} \Op{a_R} \e^{- \ii t P}
    \le \e^{\ii t P} \one_{\tilde \omega(R)} \e^{- \ii t P}
        \,, \qquad \forall t \in \R \,.
\end{equation*}
Then we use Egorov's theorem (Proposition~\ref{prop:Egorovthm}) and we integrate with respect to $t$ to get
\begin{equation} \label{eq:withremainder}
\int_0^{T_0} \e^{\ii t P} \one_{(\omega \setminus K_R) \times \R^d} \e^{- \ii t P} \dd t
	\le \int_0^{T_0} \Op{a_R \circ \phi^t} \dd t + \cal{R}_R
    \le \int_0^{T_0} \e^{\ii t P} \one_{\tilde \omega(R)} \e^{- \ii t P} \dd t \,,
\end{equation}
where the remainder term $\cal{R}_{R}$ is a bounded operator with
\begin{equation} \label{eq:remainderR'}
\norm*{\cal{R}_{R}}_{L^2 \to L^2}
    \le C \abs*{\nabla a_{R}}_{S(1)}^{k_d}
    \le \dfrac{C'}{R} \,,
    	\qquad \forall R \ge 1 \,.
\end{equation}
The constant $C'$ above depends only on $p$ and $T_0$ (and of course on the dimension $d$), but not on $\omega$ or $K$. One can check that the quantization and the integral over $t$ in the middle term of~\eqref{eq:withremainder} commute.\footnote{One can see this by pairing the operator under consideration with two Schwartz functions and use the Dominated Convergence Theorem.}

On the classical side, using the same notation as in~\eqref{eq:notationCcal}, we introduce the quantity
\begin{equation*}
\frak{K}_{p_0}^\infty\left(a_R \one_{(0, T)}\right)
	= \liminf_{\rho \to \infty} \int_0^T a_R\left(\phi_0^t(\rho)\right) \dd t \,,
\end{equation*}
and similarly for $p$ instead of $p_0$, replacing the flow $\phi_0^t$ by $\phi^t$. We claim that for any $T > 0$, there exists a constant $C'' > 0$ depending only on the dimension, on $T$ and on the Hamiltonians $p_0$ and $p$, such that for any compact $\tilde K$, and for any $R > 0$:
\begin{equation} \label{eq:inequalitiesclassicalside}
\left\{
\begin{aligned}
\frak{K}_{p_0}^\infty(\omega, T)
	&\le \frak{K}_p^\infty\left(a_R \one_{(0, T)}\right) + \dfrac{C''}{R} \\
\frak{K}_p^\infty\left(a_R \one_{(0, T)}\right)
	&\le \frak{K}_{p_0}^\infty\left(\omega_R \setminus \tilde K, T\right) + \dfrac{C''}{R}
\end{aligned}
\right. \,.
\end{equation}
The constant $C''$ does not depend on $\omega$ or $K$ from which we built $a_R$ neither. The proof of the first inequality in~\eqref{eq:inequalitiesclassicalside} reads as follows: Corollary~\ref{cor:classicalnonobs} shows that the quantity in the left-hand side does not change if we remove a compact set:
\begin{equation} \label{eq:reductionCcal1}
\frak{K}_{p_0}^\infty(\omega, T)
	= \frak{K}_{p_0}^\infty\left(\omega \setminus K_R, T\right) \,,
		\qquad \forall R > 0 \,.
\end{equation}
Now we use that $\one_{(\omega \setminus K_R) \times \R^d} \le a_R$ to get
\begin{equation} \label{eq:reductionCcal2}
\frak{K}_{p_0}^\infty\left(\omega \setminus K_R, T\right)
	\le \frak{K}_{p_0}^\infty\left(a_R \one_{(0, T)}\right) \,.
\end{equation}
Then we switch from $p_0$ to $p$, having the same principal symbol, using Corollary~\ref{cor:classicalobssubprincipalperturbation}: the function $a_R \one_{(0, T)}$ is compactly supported in time and $\frac{c_{d, 1}}{R}$-Lipschitz in the variable $\rho$, so that
\begin{equation*}
\frak{K}_{p_0}^\infty\left(a_R \one_{(0, T)}\right)
	\le \frak{K}_p^\infty\left(a_R \one_{(0, T)}\right) + \dfrac{C''}{R} \,,
		\qquad \forall R > 0 \,.
\end{equation*}
Putting this together with~\eqref{eq:reductionCcal1} and~\eqref{eq:reductionCcal2} yields the first inequality in~\eqref{eq:inequalitiesclassicalside}. The second inequality in~\eqref{eq:inequalitiesclassicalside} is proved using similar arguments: Corollary~\ref{cor:classicalobssubprincipalperturbation} leads to
\begin{equation*}
\frak{K}_p^\infty\left(a_R \one_{(0, T)}\right)
	\le \frak{K}_{p_0}^\infty\left(a_R \one_{(0, T)}\right) + \dfrac{C''}{R} \,, \qquad \forall R > 0 \,.
\end{equation*}
Then we use from the construction of $a_R$ in~\eqref{eq:symbolaR} that $a_R$ is supported in $\omega_R \times \R^d$, and we apply Corollary~\ref{cor:classicalnonobs} to remove a compact set $\tilde K$. This leads to the sought inequality.

\medskip
\emph{Sufficient condition.}
We wish to bound the left-hand side of~\eqref{eq:withremainder} from below. The high-energy classical observability constant $\frak{K}_{p_0}^\infty := \frak{K}_{p_0}^\infty(\omega, T_0)$ is assumed to be positive. From the first inequality in~\eqref{eq:inequalitiesclassicalside}, with $T_0$ in place of $T$, we can write
\begin{equation} \label{eq:averageboundedfrombelowatinfinity}
\exists A > 0 : \forall \abs*{\rho} \ge A \,, \qquad
	\int_0^{T_0} (a_R \circ \phi^t)(\rho) \dd t
		\ge \dfrac{1}{2} \frak{K}_{p_0}^\infty - \dfrac{C''}{R}
		= c_R \,.
\end{equation}
Take a cut-off function $\chi \in \cont_\comp^\infty(\R^{2d})$ such that $\chi \equiv 1$ on the unit ball, and set $\chi_R = \chi(\bigcdot/(A + R))$. Then $\chi_R$ has compact support, equals one on the ball $B_A(0)$, and it satisfies $\norm{\partial^\alpha \chi_R}_\infty = O(1/R^{\abs{\alpha}})$, with constants independent of $\omega$ again.\footnote{The parameter $A$ depends on $R$, but this will not be a problem in the sequel. The phase space region localized at distance $\le A$ from the origin will be handled by Proposition~\ref{prop:reduction}.} We split the symbol in the left-hand side of~\eqref{eq:averageboundedfrombelowatinfinity} using this cut-off function: we write
\begin{equation*}
\int_0^{T_0} a_R \circ \phi^t \dd t
    = b_0 + b_\infty
\end{equation*}
where we set
\begin{equation*}
b_0 = \chi_R \times \left( \int_0^{T_0} a_R \circ \phi^t \dd t - c_R \right)
    \qquad \rm{and} \qquad
b_\infty = (1 - \chi_R) \int_0^{T_0} a_R \circ \phi^t \dd t + c_R \chi_R \,.
\end{equation*}
Using the Leibniz Formula and Lemma~\ref{lem:symbolcomposition}, we can prove that $b_0 \in S(1)$. Moreover, $b_0$ is compactly supported in $\R^{2d}$, so that $\Op{b_0}$ is a compact operator by~\cite[Theorem 4.28]{Zworski:book}. As for $b_\infty$, the Leibniz Formula and Lemma~\ref{lem:symbolcomposition} lead to the following estimates on derivatives: for all $\alpha \in \N^{2d}$, one has
\begin{equation*}
\norm*{\partial^\alpha b_\infty}_\infty
    \le C_{\alpha} \max_{\alpha_1 + \alpha_2 = \alpha} \norm*{\partial^{\alpha_1} (1 - \chi_R)}_\infty \times \int_0^{T_0} \norm*{\partial^{\alpha_2} (a_R \circ \phi^t)}_\infty \dd t + c_R \dfrac{C_\alpha}{R^{\abs{\alpha}}}
    \le C_{\alpha, T_0, p} \left( \dfrac{1}{R^{\abs{\alpha}}} + \dfrac{1}{R} \right) \,.
\end{equation*}
The last inequality comes from distinguishing the cases $\alpha_2 = 0$ and $\alpha_2 \neq 0$. In the first case, we have $\partial^{\alpha_1}(1 - \chi_R) = O(R^{-\abs{\alpha}})$ and $\abs{a_R \circ \phi^t} \le 1$. Otherwise, Lemma~\ref{lem:symbolcomposition} tells us that $\partial^{\alpha_2} (a_R \circ \phi^t)$ behaves like $\abs{\nabla a_R}_{S(1)}^{\abs{\alpha}} = O(1/R)$, $R \ge 1$. In particular, $b_\infty \in S(1)$ and it holds $\abs{\hess b_\infty}_{S(1)}^\ell = O(1/R)$ for any $\ell \in \N$, with a constant independent of $\omega$ and $K$. In addition, we have $b_\infty \ge c_R$ in view of~\eqref{eq:averageboundedfrombelowatinfinity}. Therefore, the G{\aa}rding inequality (Proposition~\ref{prop:Gaarding}) yields
\begin{equation*}
\Op{b_\infty} \ge \left(c_R - \dfrac{C_1}{R}\right) \id \,.
\end{equation*}
The constant $C_1$ is independent of $\omega$ and $K$ in view of the seminorm estimates of $\hess b_\infty$ discussed above. Going back to~\eqref{eq:withremainder}, we have proved
\begin{equation*}
\int_0^T \e^{\ii t P} \one_{\tilde \omega(R)} \e^{- \ii t P} \dd t
    \ge c_R \id + \Op{b_0} + \cal{R} \,,
    	\qquad \forall R \ge 1 \,.
\end{equation*}
As we have seen in the course of the proof, $\Op{b_0}$ is a compact self-adjoint operator and $\norm*{\cal{R}}_{L^2 \to L^2} \le C_2/R$, with a constant $C_2$ depending only on the dimension, on $T_0$ and on the Hamiltonians $p_0, p$. In view of the definition of $c_R$ in~\eqref{eq:averageboundedfrombelowatinfinity}, taking $R = 4(C'' + C_2 + T_0)/\frak{K}_{p_0}^\infty$, we obtain the desired observability inequality, up to a compact operator:
\begin{equation*}
\int_0^{T_0} \e^{\ii t P} \one_{\tilde \omega(R)} \e^{- \ii t P} \dd t - \Op{b_0}
    \ge \dfrac{1}{4} \frak{K}_{p_0}^\infty \id \,.
\end{equation*}
Notice that indeed $R \ge 1$, since $\frak{K}_{p_0}^\infty \le T_0$. Proposition~\ref{prop:reduction} then applies (see Remark~\ref{rmk:uniquecontinuation}). It yields the sought observability inequality on $\tilde \omega(R) \subset \omega_R \setminus K$, in any time $T > T_0$.

\medskip
\emph{Necessary condition.}
Consider the symbol $a_R$ from~\eqref{eq:symbolaR} with $K = \varnothing$. We fix $R \ge 1$ (not necessarily large), $\tilde K$ compact, and we estimate the observation cost $C_\obs$ in~\eqref{eq:obscost} using the quantity $\frak{K}_{p_0}^\infty(\omega_R \setminus \tilde K, T)$. We will track carefully the dependence of remainders on the parameter $R$. Write for short
\begin{equation*}
\jap*{a_R}_T(\rho)
	= \int_0^T (a_R \circ \phi^t)(\rho) \dd t \,, \qquad \rho \in \R^{2d} \,,
\end{equation*}
and pick a point $\rho_0 \in \R^{2d}$ such that
\begin{equation*}
\jap*{a_R}_T(\rho_0)
	\le \inf_{\rho \in \R^{2d}} \jap*{a_R}_T(\rho) + \dfrac{1}{R} \,.
\end{equation*}
Notice that in virtue of the second inequality of~\eqref{eq:inequalitiesclassicalside}, we have
\begin{equation}  \label{eq:boundfromaboveaveragecalC}
\jap*{a_R}_T(\rho_0)
	\le \frak{K}_p^\infty\left(a_R \one_{(0, T)}\right) + \dfrac{1}{R}
	\le \frak{K}_{p_0}^\infty\left(\omega_R \setminus \tilde K, T\right) + \dfrac{C'' + 1}{R} \,.
\end{equation}
Differentiating under the integral sign and using Lemma~\ref{lem:symbolcomposition}, we check that $\jap{a_R}_T$ is Lipschitz as a function of $\rho$:
\begin{equation} \label{eq:gradientjapa_RT}
\forall \rho \in \R^{2d} \,, \qquad
	\abs*{\nabla \jap{a_R}_T(\rho)}
		\le T \sup_{t \in [0, T]} \abs*{\nabla (a_R \circ \phi^t)(\rho)}
		\le C(T, p) \norm*{\nabla a_R}_\infty
		\le \dfrac{c}{R} \,.
\end{equation}
Consider a Gaussian wave packet centered at $\rho_0$, namely, writing $\rho_0 = (x_0, \xi_0)$, we define
\begin{equation*}
w(x) = \pi^{-d/4} \exp\left(-  \frac{\abs{x - x_0}^2}{2}\right) \e^{\ii \xi_0 \cdot x} \,, \qquad x \in \R^d \,.
\end{equation*}
It is properly normalized: $\norm{w}_{L^2} = 1$. A classical computation (cf.~\cite[Proposition (1.48)]{FollandPhaseSpace}) shows that the Wigner transform of $w$ is the Gaussian in the phase space centered at $\rho_0$, defined by $\rho \mapsto \pi^{-d} \exp(- \abs{\rho - \rho_0}^2)$, that is to say
\begin{equation*}
\inp*{w}{\Op{\jap*{a_{2R}}_T} w}_{L^2}
    = \pi^{-d} \int_{\R^{2d}} \jap*{a_{2R}}_T(\rho) \exp(- \abs{\rho - \rho_0}^2) \dd \rho
    = \pi^{-d} \int_{\R^{2d}} \jap*{a_{2R}}_T(\rho_0 + \rho) \exp(- \abs{\rho}^2) \dd \rho \,.
\end{equation*}
Note that it is a non-negative quantity. Taking an arbitrary $A > 0$ and splitting the integral over $\R^{2d}$ into two pieces, we obtain
\begin{align*}
\inp*{w}{\Op{\jap{a_{2R}}_T} w}_{L^2}
    &\le \int_{B_A(0)} \left(\jap{a_{2R}}_T(\rho_0) + A \norm*{\nabla \jap{a_{2R}}_T}_\infty\right) \pi^{-d} \e^{- \abs{\rho}^2} \dd \rho \\
    &\qquad \qquad \qquad + \int_{\R^{2d} \setminus B_A(0)} \norm*{\jap{a_{2R}}_T}_\infty \pi^{-d} \e^{- \abs{\rho}^2} \dd \rho \\
    &\le \jap*{a_{2R}}_T(\rho_0) + A \dfrac{c}{R} + T \int_{\R^{2d} \setminus B_A(0)} \pi^{-d} \e^{- \abs{\rho}^2} \dd \rho \\
    &\le \frak{K}_{p_0}^\infty\left(\omega_R \setminus \tilde K, T\right) + \dfrac{C'' + 1 + A c}{R} + T \e^{- A^2/2} 2^d \,.
\end{align*}
We used~\eqref{eq:gradientjapa_RT} and~\eqref{eq:boundfromaboveaveragecalC} to obtain the last two inequalities. We take $A = \abs{2 \log R}^{1/2}$ to obtain
\begin{equation*}
\inp*{w}{\Op{\jap{a_{2R}}_T} w}_{L^2}
    \le \frak{K}_{p_0}(\omega_R \setminus \tilde K, T) + \tilde C \dfrac{1 + \abs{\log R}^{1/2}}{R}
\end{equation*}
for some constant $\tilde C > 0$ independent of $R$. Going back to the left-hand side of~\eqref{eq:withremainder} (recall that we chose $K = \varnothing$ here) with $T$ in place of $T_0$, as well as~\eqref{eq:remainderR'}, taking the inner product with $w$ on both sides, we deduce that
\begin{equation*}
\int_0^T \norm*{\e^{- \ii t P} w}_{L^2(\omega)}^2 \dd t
    \le \frak{K}_{p_0}(\omega_R \setminus \tilde K, T) + \tilde C \dfrac{1 + \abs{\log R}^{1/2}}{R} + \dfrac{C'}{R} \,.
\end{equation*}
By assumption, $\Obs(\omega, T)$ is true with a cost $C_\obs > 0$. Recalling that $\norm{w}_{L^2} = 1$, we can bound the left-hand side from below by $C_\obs^{-1}$. We arrive at
\begin{equation*}
\frak{K}_{p_0}(\omega_R \setminus \tilde K, T)
	\ge \dfrac{1}{C_\obs} - \tilde C \dfrac{1 + \abs{\log R}^{1/2}}{R} - \dfrac{C'}{R} \,,
\end{equation*}
which yields the sought result.
\end{proof}

\section{Proofs of observability results from conical sets} \label{sec:proofconical}

In this section, we give proofs of the results presented in Subsections~\ref{subsub:conicalsets} and~\ref{subsub:refinementconical}, which concern observation sets that are conical in the sense of~\eqref{eq:defconicalset}. Propositions~\ref{prop:necessaryconditionconicset}, \ref{prop:twocones} and~\ref{prop:isotropicconicalsets} are proved in Subsections~\ref{subsec:proofnecessaryconditionconicset}, \ref{subsec:prooftwocones} and~\ref{subsec:proofisotropicconicalsets} respectively.

\subsection{Proof of Proposition~\ref{prop:necessaryconditionconicset}} \label{subsec:proofnecessaryconditionconicset}

Let us prove the converse statement: assume there exists a normalized eigenvector $e$ of $A$ such that $e \not\in \ovl{\omega}$ and $-e \not\in \ovl{\omega}$. Let $\nu > 0$ be such that $A e = \nu^2 e$. We claim the following.

\begin{lemm} \label{lem:sejnotinomegaR}
There exists a constant $c > 0$ such that for any $R > 0$, it holds
\begin{equation*}
\forall s \in \R \,, \qquad
	\left( s e \in \omega_R \quad \Longrightarrow \quad \abs{s} \le c R \right) \,.
\end{equation*}
\end{lemm}

\begin{proof}
If $s \in \R$ is such that $s e \in \omega_R$, then there exists $y \in \omega \setminus \{0\}$ such that $\abs{s e - y} \le R$. Moreover, since $e$ belongs to the complement of the closed set $\ovl{\omega} \cup - \ovl{\omega}$, there exists $\eps > 0$ such that
\begin{equation*}
\forall x \in \left(\omega \cup - \omega\right) \setminus \{0\} \,, \qquad \abs*{e - \dfrac{x}{\abs{x}}} \ge \eps \,.
\end{equation*}
We apply this to $x = \sign(s) y$ to obtain
\begin{equation*}
\abs{s}
    \le \dfrac{1}{\eps} \abs*{s e - \abs{s} \dfrac{y}{\abs{y}}}
    \le \dfrac{1}{\eps} \abs{s e - y} + \dfrac{1}{\eps} \abs{y} \abs*{1 - \dfrac{\abs{s}}{\abs{y}}}
    \le \dfrac{1}{\eps} \abs{s e - y} + \dfrac{1}{\eps} \abs{s e - y}
    \le \dfrac{2 R}{\eps} \,.
\end{equation*}
We used the inverse triangle inequality to obtain the second to last inequality.
\end{proof}

Using this lemma, for any $T > 0$ and any $\eta > 0$, we can estimate the quantity
\begin{equation*}
\int_0^T \one_{\omega_R \times \R^d} \left(\phi^t(0, \eta e)\right) \dd t
    = \int_0^T \one_{\omega_R} \left(\dfrac{\eta}{\nu} \sin(\nu t) e\right) \dd t
    \le \int_0^{2 N \pi /\nu} \one_{\omega_R} \left(\dfrac{\eta}{\nu} \sin(\nu t) e\right) \dd t \,,
\end{equation*}
where $N = \lceil \nu T/2 \pi \rceil$. Using the periodicity of the sine and a change of variable, we deduce:
\begin{equation*}
\int_0^T \one_{\omega_R \times \R^d} \left(\phi^t(0, \eta e)\right) \dd t
    \le \dfrac{N}{\nu} \int_0^{2 \pi} \one_{\omega_R}\left(\dfrac{\eta}{\nu} \sin(t) e\right) \dd t
    = \dfrac{2 N}{\nu} \int_{-\pi/2}^{\pi/2} \one_{\omega_R} \left(\dfrac{\eta}{\nu} \sin(t) e\right) \dd t \,.
\end{equation*}
Provided $\eta \neq 0$, we make the change of variables $s = \eta \sin t$, for which we have $\dd t = (\eta^2 - s^2)^{-1/2} \dd s$; this leads to
\begin{equation*}
\int_0^T \one_{\omega_R \times \R^d} \left(\phi^t(0, \eta e)\right) \dd t
    \le \dfrac{2 N}{\nu} \int_{-\abs{\eta}}^{\abs{\eta}} \one_{\omega_R} \left(\dfrac{s}{\nu} e\right) \dfrac{\dd s}{\sqrt{\eta^2 - s^2}} \,.
\end{equation*}
From Lemma~\ref{lem:sejnotinomegaR} above, we conclude that for any $\eta$ large enough:
\begin{equation*}
\int_0^T \one_{\omega_R \times \R^d} \left(\phi^t(0, \eta e)\right) \dd t
    \le \dfrac{2 N}{\nu} \int_{-c R \nu}^{c R \nu} \dfrac{\dd s}{\sqrt{\eta^2 - s^2}} \,.
\end{equation*}
An extra change of variables yields
\begin{equation*}
\int_0^T \one_{\omega_R \times \R^d} \left(\phi^t(0, \eta e)\right) \dd t
    \le \dfrac{2 N}{\nu} \int_{-c R \nu/\eta}^{c R \nu/\eta} \dfrac{\dd s}{\sqrt{1 - s^2}}
    = O\left(\dfrac{R}{\eta}\right)
\end{equation*}
as $\eta$ tends to infinity and $R$ is fixed. We deduce that for any $R > 0$, it holds
\begin{equation*}
\liminf_{\rho \to \infty} \int_0^T \one_{\omega_R \times \R^d}\left(\phi^t(\rho)\right) \dd t = 0 \,.
\end{equation*}
The necessary condition of Theorem~\ref{thm:main} then proves that observability cannot hold from $\omega$ in time $T$. \hfill \qedsymbol

\subsection{Proof of Proposition~\ref{prop:twocones}} \label{subsec:prooftwocones}

We first reduce to the case where the matrix $A$ is diagonal in the canonical basis of $\R^2$. Then we investigate the isotropic and anisotropic cases separately.

\medskip
\paragraph*{Step 1 \--- Reduction to positive cones containing half coordinate axes.}
%%%%%%%%%%%%%%%%%%%%%%%%%%%%%%%%%%%%%%%%%%%%%%%%%%%%%%%%%%%%%%%%%%%%%
Let $S : \R^2 \to \R^2$ be a linear symplectic mapping. It holds $\nabla (p \circ S) = S^\ast (\nabla p) \circ S$, and we observe that
\begin{equation*}
\dfrac{\dd}{\dd t} S^{-1} \phi^t(S \rho)
    = S^{-1} J (S^{-1})^\ast S^\ast \nabla p\left(\phi^t(S \rho)\right)
    = J \nabla (p \circ S)\left(S^{-1} \phi^t(S \rho)\right) \,.
\end{equation*}
This means that the conjugation of the Hamiltonian flow of $p$ by $S$ is the Hamiltonian flow of $p \circ S$. Thus, for any measurable set $C \subset \R^2 \times \R^2$:
\begin{equation*}
\int_0^T \one_{C}\left(\phi^t(\rho)\right) \dd t
    = \int_0^T \one_{S^{-1} C}\left((S^{-1} \phi^t S)(S^{-1} \rho)\right) \dd t \,,
\end{equation*}
and finally, since $S^{-1} \rho \to \infty$ if and only if $\rho \to \infty$, we deduce that
\begin{equation} \label{eq:reductionpositivecones}
\liminf_{\rho \to \infty} \int_0^T \one_{C}\left(\phi^t(\rho)\right) \dd t
    = \liminf_{\rho \to \infty} \int_0^T \one_{S^{-1} C}\left((S^{-1} \phi^t S)(\rho)\right) \dd t \,.
\end{equation}
Denoting by $Q$ the orthogonal matrix that diagonalizes $A$ as follows:
\begin{equation*}
Q^{-1} A Q = \begin{pmatrix} \nu_-^2 & 0 \\ 0 & \nu_+^2 \end{pmatrix} \,,
    \qquad \rm{with} \; Q \begin{pmatrix} 1 \\ 0 \end{pmatrix} = e_- \; \rm{and} \; Q \begin{pmatrix} 0 \\ 1 \end{pmatrix} = e_+ \,,
\end{equation*}
we apply the above observation~\eqref{eq:reductionpositivecones} to the map
\begin{equation*}
S =
\begin{pmatrix} Q & 0 \\ 0 & Q \end{pmatrix} \,.
\end{equation*}
It is indeed symplectic since $Q = (Q^{-1})^\ast$ is an orthogonal matrix. When the subset of the phase space $C$ is of the form $\omega(\eps)$ given in the statement, the resulting set $S^{-1} C$ is $\tilde \omega(\eps) = C_\eps^1 \cup C_\eps^2$ where
\begin{equation*}
C_\eps^1
	= \set{(x_1, x_2) \in \R^2}{\abs{x_2} < \tan\left(\dfrac{\eps}{2}\right) x_1}
		\qquad \rm{and} \qquad
C_\eps^2
	= \set{(x_1, x_2) \in \R^2}{\abs{x_1} < \tan\left(\dfrac{\eps}{2}\right) x_2} \,.
\end{equation*}
The corresponding Hamiltonian is
\begin{equation*}
(p \circ S)(x, \xi)
    = \dfrac{1}{2} \left( Q x \cdot A Q x + \abs*{Q \xi}^2 \right)
    = \dfrac{1}{2} \left( \nu_-^2 {x_1}^2 + \nu_+^2 {x_2}^2 + \abs*{\xi}^2 \right) \,.
\end{equation*}
That is to say, we have reduced the problem to the study of observability from $\tilde \omega(\eps)$ for the above Hamiltonian: the Schrödinger equation is observable from $\omega(\eps)$ in time $T$ for the Hamiltonian $p$ is and only if it is observable from $\tilde \omega(\eps)$ in time $T$ for the Hamiltonian $p \circ S$. From now on, we write $\omega(\eps)$ instead of $\tilde \omega(\eps)$, $p$ instead of $p \circ S$ respectively, and $(\nu_1, \nu_2) = (\nu_-, \nu_+)$ .

\medskip
\paragraph*{Step 2 \--- Isotropic case.}
%%%%%%%%%%%%%%%%%%%%%%%%%%%%%%%%%%
The case where $\nu_+ = \nu_- = \nu$ follows from Proposition~\ref{prop:necessaryconditionconicset}. Indeed, since $\eps < \pi/2$, one has $\ovl{\omega(\eps)} \cap L_\pm = \{0\}$ where $L_\pm = \{x_2 = \pm x_1\}$ are eigenspaces of $A = \nu^2 \id$. Therefore, isotropic oscillators are not observable from $\omega(\eps)$.

\medskip

\paragraph*{Anisotropic case.} We assume that the harmonic oscillator is anisotropic, i.e.\ $\nu_1 < \nu_2$, and we want to show that $\omega(\eps)$ observes the Schrödinger equation. Anticipating the use of Theorem~\ref{thm:main} where the observation set has to be enlarged, we will rather prove that the dynamical condition in~\eqref{eq:dynamicalcondition} is verified by the smaller set $\omega(\frac{\eps}{2}) = C_{\eps/2}^1 \cup C_{\eps/2}^2$. We fix an initial point $\rho^0 = (x_1^0, x_2^0; \xi_1^0, \xi_2^0) \in \R^2 \times \R^2$. We write the space components of the flow as follows:
\begin{multline*}
x_j^t
    = \cos(\nu_j t) x_j^0 + \dfrac{1}{\nu_j} \sin(\nu_j t) \xi_j^0
    = A_j \sin(\nu_j t + \theta_j) \,, \qquad j \in \{1, 2\}, t \in \R \,, \\
		\rm{with} \qquad
A_j = \sqrt{(x_j^0)^2 + \left(\dfrac{\xi_j^0}{\nu_j}\right)^2}
    \qquad \rm{and} \qquad
\cos \theta_j =  \dfrac{\xi_j^0/\nu_j}{A_j} \,, \;\; \sin \theta_j = \dfrac{x_j^0}{A_j} \,.
\end{multline*}
Our first goal will be to prove that the dynamical condition~\eqref{eq:dynamicalcondition} is satisfied in the time interval $[0, T_0]$, where $T_0$ is given in~\eqref{eq:optimaltimeconicset}. We can consider $\rho^0$ to be nonzero since we are interested in what happens at infinity. Therefore $A_1 > 0$ or $A_2 > 0$. Also keep in mind that $\rho^0 \to \infty$ if and only if $\abs{(A_1, A_2)} \to + \infty$.

\medskip
\paragraph*{Step 3 \--- Time spent in $C_{\eps/2}^2$.}
%%%%%%%%%%%%%%%%%%%%%%%%%%%%%%%%%%%%%%%%%%%%%%%%%%%
First we look at the possibility to be in the cone $C_{\eps/2}^2$. This will certainly be the case provided $A_1$ is very small compared to $A_2$, that is to say the projected trajectory $(x_1^t, x_2^t)$ is almost contained in the ordinate axis. We prove the following:
\begin{equation} \label{eq:firstcone!}
\int_0^{T_0} \one_{C_{\eps/2}^2}(x_1^t, x_2^t) \dd t
    \ge \dfrac{\pi}{\nu_2} \left(1 - \dfrac{A_1/A_2}{\tan(\frac{\eps}{4})}\right) \,.
\end{equation}
Suppose $t \in [0, T_0]$ is such that $\sin(\nu_2 t + \theta_2) \ge \delta$, namely $x_2^t \ge A_2 \delta$. Assuming that $A_2 > 0$, one has
\begin{equation} \label{eq:firstcone?}
\abs{x_1^t} \le A_1 \le \dfrac{A_1}{A_2 \delta} x_2^t \,.
\end{equation}
We want to quantify the amount of $t$ such that this holds. In the following estimate, we use the fact that $T_0 \ge \frac{2\pi}{\nu_2}$ and the classical concavity inequality $\sin x \ge \frac{2}{\pi} x, \forall x \in [0, \frac{\pi}{2}]$:
\begin{equation*}
\int_0^{T_0} \one_{\sin(\nu_2 t + \theta_2) \ge \delta} \dd t
    \ge \int_0^{2 \pi/\nu_2} \one_{\sin(\nu_2 t) \ge \delta} \dd t
    \ge \dfrac{1}{\nu_2} \int_0^{2 \pi} \one_{\sin t \ge \delta} \dd t
    \ge \dfrac{2}{\nu_2} \int_0^{\pi/2} \one_{\frac{2}{\pi} t \ge \delta} \dd t
    = \dfrac{\pi}{\nu_2} (1 - \delta) \,.
\end{equation*}
Now in~\eqref{eq:firstcone?}, we wish $\frac{A_1}{A_2 \delta}$ to be strictly less than $\tan(\eps/4)$, that is to say $\delta > A_1/(A_2 \tan(\eps/4))$. Therefore, for any $\delta$ satisfying this condition, the time spent by the trajectory in $C_{\eps/2}^2$ can be bounded from below by:
\begin{equation*}
\int_0^{T_0} \one_{x_2^t \ge A_2 \delta} \dd t
    \ge \int_0^{T_0} \one_{\sin(\nu_2 t + \theta_2) \ge \delta} \dd t
    \ge \dfrac{\pi}{\nu_2} (1 - \delta) \,,
\end{equation*}
so that, maximizing the right-hand side with respect to $\delta$, one obtains~\eqref{eq:firstcone!}. Notice that this inequality is useful only if $A_1/A_2$ is small enough. In the opposite case where $A_1/A_2$ is large, we use another argument ($\nu_1$ and $\nu_2$ do not play a symmetric role here).

\medskip
\paragraph*{Step 4 \--- Time spent in $C_{\eps/2}^1$.}
%%%%%%%%%%%%%%%%%%%%%%%%%%%%%%%%%%%%%%%%%%%%%%%%%%%
Let us now consider the times when the trajectory is in the other cone $C_{\eps/2}^1$. Set $\eta = \lfloor \frac{\nu_2}{\nu_1} \rfloor + 1 - \frac{\nu_2}{\nu_1} \in (0, 1]$. The main claim in this step of the proof is the following:
\begin{equation} \label{eq:claimstep4}
\exists t_2 \in [0, T_0] : \qquad
	x_2^{t_2} = 0 \quad \rm{and} \quad x_1^{t_2} \ge A_1 \delta_1 \,,
		\qquad \rm{where} \; \, \delta_1 = \min\left(\dfrac{\nu_1}{\nu_2} \eta, 1 - \dfrac{\nu_1}{\nu_2}\right) \,.
\end{equation}	
Denote by $t_1$ the first zero of $\sin(\nu_1 t + \theta_1)$ in $[0, T_0]$. It exists since by definition, $T_0 \ge \frac{\pi}{\nu_2}(1 + \frac{\nu_2}{\nu_1}) \ge \frac{\pi}{\nu_1}$. It turns out that $t_1$ is given by
\begin{equation*}
t_1
    = \dfrac{\pi}{\nu_1} \left(\left\lceil \dfrac{\theta_1}{\pi} \right\rceil - \dfrac{\theta_1}{\pi}\right) \,.
\end{equation*}
Then $t_1 \in [0, \frac{\pi}{\nu_1})$, and we know that $\sin(\nu_1 t + \theta_1)$ has constant sign on $I_1 := [0, t_1]$, on $I_2 := [t_1, t_1 + \frac{\pi}{\nu_1}] \cap [0, T_0]$ and on $I_3 := [t_1 + \frac{\pi}{\nu_1}, t_1 + \frac{2 \pi}{\nu_1}] \cap [0, T_0]$. Observe that $I_1$ is possibly reduced to a singleton, $I_2$ is always non trivial, and $I_3$ is possibly empty. One can check this from the fact that $T_0$ can be rewritten
\begin{equation} \label{eq:T_ast}
T_0 = \frac{\pi}{\nu_1} + (1 + \eta) \dfrac{\pi}{\nu_2} \,.
\end{equation}
Because $T_0 \ge \frac{\pi}{\nu_1}$, we know that $\sin(\nu_1 t + \theta_1)$ vanishes at least once in $[0, T_0]$. We first distinguish cases according to whether there are a single one or more than two of these zeroes in this interval.

\begin{itemize} %%list
\item Assume first $t_1$ is the only zero in $[0, T_0]$. Then in view of~\eqref{eq:T_ast}, it lies at distance $> (1 + \eta) \frac{\pi}{\nu_2}$ from the boundary of $[0, T_0]$, otherwise $t_1 + \frac{\pi}{\nu_1}$ or $t_1 - \frac{\pi}{\nu_1}$ is another zero in $[0, T_0]$. In particular, the intervals $[0, t_1]$ and $[t_1, T_0]$ have length $\ge (1 + \eta) \frac{\pi}{\nu_2}$. We know that $\sin(\nu_1 t + \theta_1) \ge 0$ on one of these intervals, that we denote by $\tilde I$. Given that $\tilde I$ has length $\ge (1 + \eta) \frac{\pi}{\nu_2}$, it contains a zero of $\sin(\nu_2 t + \theta_2)$, lying at distance $\ge \frac{\pi}{\nu_2} \frac{\eta}{2}$ from the boundary of $\tilde I$. We denote such a zero by $t_2$. Given that the only zero of $\sin(\nu_1 t + \theta_1)$ in $\tilde I$ is $t_1$, we deduce that the distance between $t_2$ and the closest zero $t_1'$ of $\sin(\nu_1 t + \theta_1)$ is at least $\frac{\pi}{\nu_2} \frac{\eta}{2}$. Then the inequality $\sin x \ge \frac{2}{\pi} x$ on $x \in [0, \frac{\pi}{2}]$ yields
\begin{equation} \label{eq:frombelowt1}
\sin(\nu_1 t_2 + \theta_1)
    = \sin\left(\nu_1(t_2 - t_1') + \nu_1 t_1' + \theta_1\right)
    = \sin\left(\nu_1 \abs{t_2 - t_1'}\right)
    \ge \dfrac{2 \nu_1}{\pi} \abs{t_2 - t_1'}
    \ge \dfrac{\nu_1}{\nu_2} \eta \,.
\end{equation}
The absolute value resulting from the second inequality is due to the fact that we chose $t_2$ in an interval where $\sin(\nu_1 t + \theta_1) \ge 0$, or equivalently, $\nu_1 t_1 + \theta_1$ is an even or odd multiple of $\pi$ according to the sign of $t_2 - t_1$. We conclude that $x_2^{t_2} = 0$ by definition of $t_2$ and that we have $x_1^{t_2} \ge A_1 \frac{\nu_1}{\nu_2} \eta$ in virtue of~\eqref{eq:frombelowt1}, hence the claim~\eqref{eq:claimstep4}.

\item Now we treat the case where $t_1 + \frac{\pi}{\nu_1}$ also lies in $[0, T_0]$. In other words, the interval $J_1 := [t_1, t_1 + \frac{\pi}{\nu_1}]$ is contained in $[0, T_0]$. As we already mentioned, $\sin(\nu_1 t + \theta_1)$ has constant sign on $J_1$.
\begin{itemize} %%list
\item If $\sin(\nu_1 t + \theta_1) \ge 0$ on $J_1$, since $J_1$ has length $\frac{\pi}{\nu_1} > \frac{\pi}{\nu_2}$, then $t \mapsto \sin(\nu_2 t + \theta_2)$ vanishes in $J_1$, and we can choose a zero $t_2$ at distance $\ge \frac{\pi}{2}(\frac{1}{\nu_1} - \frac{1}{\nu_2})$ from the boundary of $J_1$. Reproducing the previous argument with the concavity inequality for the sine function, we deduce that
\begin{equation*}
\sin(\nu_1 t_2 + \theta_1)
    \ge \dfrac{2 \nu_1}{\pi} \times \dfrac{\pi}{2} \left(\dfrac{1}{\nu_1} - \dfrac{1}{\nu_2}\right)
    = 1 - \dfrac{\nu_1}{\nu_2} \,.
\end{equation*}
Therefore in this case, there is $t_2 \in [0, T_0]$ with $x_2^{t_2} = 0$ and $x_1^{t_2} \ge (1 - \frac{\nu_1}{\nu_2}) A_1$, hence the claim~\eqref{eq:claimstep4}.

\item In the remaining case where $\sin(\nu_1 t + \theta_1) \le 0$ on $J_1$, we introduce some additional notation. We denote by $t_-$ (resp. $t_+$) the largest (resp. smallest) zero of $\sin(\nu_2 t + \theta_2)$ which is $< t_1$ (resp. $> t_1 + \frac{\pi}{\nu_1}$), given respectively by
\begin{equation*}
t_- = \dfrac{\pi}{\nu_2} \left(\left\lceil \dfrac{\nu_2 t_1 + \theta_2}{\pi}\right\rceil - 1 - \dfrac{\theta_2}{\pi}\right)
    \;\, \rm{and} \;\,
t_+ = \dfrac{\pi}{\nu_2} \left(\left\lfloor \dfrac{\nu_2 (t_1 + \frac{\pi}{\nu_1}) + \theta_2}{\pi}\right\rfloor + 1 - \dfrac{\theta_2}{\pi}\right) \,.
\end{equation*}
They both have the property that $\sin(\nu_1 t_\pm + \theta_1) > 0$, but we wish to quantify this statement in order to have a uniform lower bound. We observe that we can write
\begin{equation*}
t_+ - t_-
    = \dfrac{\pi}{\nu_2} \left(k + 1 + \left\lfloor \dfrac{\nu_2}{\nu_1}\right\rfloor\right) \,,
\end{equation*}
with $k \in \{0, 1\}$. Indeed, from the definition of $t_+$ and $t_-$ and the properties of the floor and ceiling functions, we see that
\begin{equation*}
k = \left\lfloor \dfrac{\nu_2 (t_1 + \frac{\pi}{\nu_1}) + \theta_2}{\pi}\right\rfloor - \left\lceil \dfrac{\nu_2 t_1 + \theta_2}{\pi}\right\rceil + 1 - \left\lfloor \dfrac{\nu_2}{\nu_1} \right\rfloor
\end{equation*}
is an integer satisfying
\begin{equation*}
-1 \le \dfrac{\nu_2}{\nu_1} - 1 - \left\lfloor \dfrac{\nu_2}{\nu_1} \right\rfloor
    < k
    \le 1 + \dfrac{\nu_2}{\nu_1} - \left\lfloor \dfrac{\nu_2}{\nu_1} \right\rfloor < 2 \,,
\end{equation*}
whence $k = 0$ or $1$. In particular, we remark that the distance between $t_-$ and $t_+$ is always less than $T_0$. This implies that either $t_-$ or $t_+$ belongs to $[0, T_0]$.
\begin{itemize} %%list
\item Suppose $t_-$ and $t_+$ both belong to $[0, T_0]$. We have
\begin{equation*}
(t_+ - t_-) - \dfrac{\pi}{\nu_1}
    = \dfrac{\pi}{\nu_2} \left(k + 1 + \left\lfloor \dfrac{\nu_2}{\nu_1}\right\rfloor - \dfrac{\nu_2}{\nu_1}\right)
    = \dfrac{\pi}{\nu_2} (k + \eta)
    \ge \dfrac{\pi}{\nu_2} \eta \,.
\end{equation*}
Recalling that $t_- < t_1$ and $t_+ > t_1 + \frac{\pi}{\nu_1}$, we deduce that either $t_1 - t_- \ge \frac{\pi}{2 \nu_2} \eta$ or $t_+ - (t_1 + \frac{\pi}{\nu_1}) \ge \frac{\pi}{2 \nu_2} \eta$. We call $t_2$ the zero, among $t_-$ and $t_+$, that satisfies this property. Then, the concavity inequality for the sine function allows to conclude that $x_1^{t_2} \ge A_1 \frac{\nu_1}{\nu_2} \eta$ again.
\item If $t_- \not\in [0, T_0]$, so that $t_+ \in [0, T_0]$, we can estimate the distance of $t_+$ from $t_1 + \frac{\pi}{\nu_1}$ and $T_0$ as follows:
\begin{equation} \label{eq:needk=0}
t_+ - \left(t_1 + \dfrac{\pi}{\nu_1}\right)
    = t_+ - t_- - \dfrac{\pi}{\nu_1} - (t_1 - t_-)
    = \dfrac{\pi}{\nu_2} (k + \eta) - (t_1 - t_-)
    \ge \dfrac{\pi}{\nu_2} \eta - \dfrac{\pi}{\nu_2}(1 - k) \,,
\end{equation}
where we used the fact that $\abs{t_1 - t_-} \le \frac{\pi}{\nu_2}$ by construction in the last inequality; and
\begin{equation} \label{eq:checkthatzeroinside}
T_0 - t_+
    = T_0 - (t_+ - t_-) - t_-
    = \dfrac{\pi}{\nu_2}(1 - k) - t_-
    \ge \dfrac{\pi}{\nu_2}(1 - k) \,,
\end{equation}
where this time we have used that $t_- < 0$ by assumption.

Now observe that $t_+$ satisfies by definition
\begin{equation*}
t_1 + \dfrac{2 \pi}{\nu_1} - t_+
	= \dfrac{\pi}{\nu_1} - \left(t_+ - \left(t_1 + \dfrac{\pi}{\nu_1}\right)\right)
	\ge \dfrac{\pi}{\nu_1} - \dfrac{\pi}{\nu_2} \,.
\end{equation*}
Thus, if $\abs{t_+ - (t_1 + \frac{\pi}{\nu_1})} \ge \frac{\pi}{2} \min(\frac{\eta}{\nu_2}, \frac{1}{\nu_1} - \frac{1}{\nu_2})$, then $t_2 = t_+$ lies at distance $\ge \frac{\pi}{2} \min(\frac{\eta}{\nu_2}, \frac{1}{\nu_1} - \frac{1}{\nu_2})$ from the boundary of the interval $[t_1 + \frac{\pi}{\nu_1}, t_1 + \frac{2 \pi}{\nu_1}]$, to which it belongs. This allows to deduce that $x_1^{t_2} \ge A_1 \delta_1$ using the inequality $\sin x \ge \frac{2}{\pi} x$ on $[0, \frac{\pi}{2}]$ again, and $x_2^{t_2} = 0$ by definition of $t_2 = t_+$. If on the contrary $\abs{t_+ - (t_1 + \frac{\pi}{\nu_1})} \le \frac{\pi}{2} \min(\frac{\eta}{\nu_2}, \frac{1}{\nu_1} - \frac{1}{\nu_2})$, then from~\eqref{eq:needk=0}, it follows that $k = 0$, so that $t_2 = t_+ + \frac{\pi}{\nu_2} \le T_0$ from~\eqref{eq:checkthatzeroinside}. Then we have
\begin{equation*}
t_1 + \dfrac{2 \pi}{\nu_1} - t_2
	= \dfrac{\pi}{\nu_1} - \dfrac{\pi}{\nu_2} - \left(t_+ - \left(t_1 + \dfrac{\pi}{\nu_1}\right)\right)
	\ge \dfrac{\pi}{2} \left(\dfrac{1}{\nu_1} - \dfrac{1}{\nu_2}\right) \,.
\end{equation*}
In particular, $t_2$ lies again at large enough distance of the boundary of $[t_1 + \frac{\pi}{\nu_1}, t_1 + \frac{2 \pi}{\nu_1}]$. We deduce as before that $x_1^{t_2} \ge A_1 \delta_1$ and $x_2^{t_2} = 0$.
\item It remains to deal with the case where $t_+ \not\in [0, T_0]$, hence $t_- \in [0, T_0]$, which is symmetrical. We write
\begin{align}
t_1 - t_-
    &= - \left(t_+ - t_1 - \dfrac{\pi}{\nu_1}\right) + t_+ - t_- - \dfrac{\pi}{\nu_1}
    \ge - \dfrac{\pi}{\nu_2}  + \dfrac{\pi}{\nu_2} (k + \eta)
    = \dfrac{\pi}{\nu_2} \eta - \dfrac{\pi}{\nu_2}(1 - k) \,, \label{eq:needk=0bis}\\
t_-
    &= T_0 - (t_+ - t_-) + t_+ - T_0
    \ge \dfrac{\pi}{\nu_2} (1 - k) \label{eq:checkthatzeroinintervalbis}\,,
\end{align}
using respectively that $\abs{t_1 + \frac{\pi}{\nu_1} - t_+} \le \frac{\pi}{\nu_2}$ by construction of $t_+$, and $t_+ > T_0$ by assumption. By definition of $t_-$ we have
\begin{equation*}
t_- - \left(t_1 - \dfrac{\pi}{\nu_1}\right)
	= \dfrac{\pi}{\nu_1} - \left(t_1 - t_-\right)
	\ge \dfrac{\pi}{\nu_1} - \dfrac{\pi}{\nu_2} \,,
\end{equation*}
so that $t_2 = t_-$ satisfies $x_1^{t_2} \ge A_1 \delta_1$ and $x_2^{t_2} = 0$ provided $\abs{t_1 - t_-} \ge \frac{\pi}{2} \min(\frac{\eta}{\nu_2}, \frac{1}{\nu_1} - \frac{1}{\nu_2})$. Otherwise, $k = 0$ in virtue of~\eqref{eq:needk=0bis} so that~\eqref{eq:checkthatzeroinintervalbis} ensures that $t_2 = t_- - \frac{\pi}{\nu_2} \ge 0$. Then we check that
\begin{equation*}
t_2 - \left(t_1 - \dfrac{\pi}{\nu_1}\right)
	= \dfrac{\pi}{\nu_1} - \dfrac{\pi}{\nu_2} - \left(t_1 - t_-\right)
	\ge \dfrac{\pi}{2} \left(\dfrac{1}{\nu_1} - \dfrac{1}{\nu_2}\right) \,,
\end{equation*}
and we conclude similarly to the previous case.
\end{itemize} %%
\end{itemize} %%
\end{itemize} %%
The discussion above shows that~\eqref{eq:claimstep4} is true. In particular, $(x_1^{t_2}, x_2^{t_2})$ is in the cone $C_{\eps/2}^1$. Using that the sine function is $1$-Lipschitz, we know that for $t$ in a neighborhood of $0$, it holds
\begin{equation*}
\abs*{x_2^{t_2 + t}} \le A_2 \nu_2 \abs{t}
    \qquad \rm{and} \qquad
x_1^{t_2 + t} \ge A_1 \left(\delta_1 - \nu_1 \abs{t}\right) \,.
\end{equation*}
So for $t$ small enough, $(x_1^{t_2 + t}, x_2^{t_2 + t})$ will remain in the cone $C_{\eps/2}^1$. Quantitatively, as soon as $t$ fulfills the condition
\begin{equation} \label{eq:conditiononabst}
\abs{t}
    < \dfrac{\frac{\delta_1}{\nu_1}}{1 + \frac{\nu_2}{\nu_1} \frac{A_2/A_1}{\tan(\eps/4)}} \,,
\end{equation}
we compute that
\begin{equation*}
x_1^{t_2 + t}
    > A_1 \dfrac{\delta_1 \frac{\nu_2}{\nu_1} \frac{A_2/A_1}{\tan(\eps/4)}}{1 + \frac{\nu_2}{\nu_1} \frac{A_2/A_1}{\tan(\eps/4)}}
    > \dfrac{A_2 \nu_2}{\tan(\eps/4)} \abs{t}
    \ge \dfrac{\abs{x_2^{t_2 + t}}}{\tan(\eps/4)} \,.
\end{equation*}
This means that for $t$ satisfying~\eqref{eq:conditiononabst}, the point $(x_1^{t_2 + t}, x_2^{t_2 + t})$ belongs indeed to the cone $C_{\eps/2}^1$. In the case where $t_2 = 0$ or $t_2 = T_0$, we may restrict ourselves to times $t$ satisfying $t \ge 0$ or $t \le 0$ in addition to~\eqref{eq:conditiononabst}, so that in the end, we obtain
\begin{equation} \label{eq:secondcone!}
\int_0^{T_0} \one_{C_{\eps/2}^1}(x_1^t, x_2^t) \dd t
    \ge \min\left( T_0, \dfrac{\frac{\delta_1}{\nu_1}}{1 + \frac{\nu_2}{\nu_1} \frac{A_2/A_1}{\tan(\eps/4)}} \right) \,.
\end{equation}

\medskip
\paragraph*{Step 5 \--- Upper bound on the optimal observation time.}
%%%%%%%%%%%%%%%%%%%%%%%%%%%%%%%%%%%%%%%%%%%%%%%%%%%%%%%%%%%%%%%%%
Now that we have~\eqref{eq:firstcone!} and~\eqref{eq:secondcone!} at hand, we can obtain a lower bound independent of the values of $A_1$ and $A_2$. If on the one hand $\frac{A_1}{A_2} \le \frac{\tan(\eps/4)}{2}$, then~\eqref{eq:firstcone!} yields
\begin{equation*}
\int_0^{T_0} \one_{(x_1^t, x_2^t) \in \omega(\frac{\eps}{2})} \dd t
    \ge \dfrac{\pi}{2 \nu_2} \,,
\end{equation*}
while on the other hand, if $\frac{A_2}{A_1} \le \frac{2}{\tan(\eps/4)}$, then~\eqref{eq:secondcone!} leads to
\begin{equation} \label{eq:dynamicalassumptioneps2}
\int_0^{T_0} \one_{(x_1^t, x_2^t) \in \omega(\frac{\eps}{2})} \dd t
    \ge \min\left( T_0, \dfrac{\frac{\delta_1}{\nu_1}}{1 + \frac{\nu_2}{\nu_1} \frac{2}{\tan^2(\eps/4)}} \right)
    \ge \dfrac{\eps^2}{16} \min\left( T_0, \dfrac{\delta_1}{\nu_1 + 2 \nu_2} \right)
\end{equation}
(to get the second inequality, use that $\frac{\eps}{4} \le \tan(\frac{\eps}{4}) \le 1$ since $\eps \le \pi/2$ by assumption). On the whole, we have
\begin{equation*}
\int_0^{T_0} \one_{(x_1^t, x_2^t) \in \omega(\frac{\eps}{2})} \dd t
    \ge \dfrac{\eps^2}{32} \min\left( \dfrac{\pi}{\nu_2}, \dfrac{\delta_1}{\nu_1 + \nu_2} \right) = c \eps^2 \,,
\end{equation*}
and setting $T_\eps = T_0 - \frac{c}{2} \eps^2$, we deduce
\begin{equation*}
\int_0^{T_\eps} \one_{(x_1^t, x_2^t) \in \omega(\frac{\eps}{2})} \dd t
    \ge \int_0^{T_0} \one_{(x_1^t, x_2^t) \in \omega(\frac{\eps}{2})} \dd t - \dfrac{c}{2} \eps^2
    \ge \dfrac{c}{2} \eps^2 \,.
\end{equation*}
Therefore the dynamical condition~\eqref{eq:dynamicalcondition} holds in time $T_\eps$. Setting $T = T_0 - \frac{c}{4} \eps^2 > T_\eps$, we use Theorem~\ref{thm:main} to conclude that observability is true on $[0, T]$ from $\omega(\frac{\eps}{2})_R \setminus K$, for some $R > 0$ and for any compact set $K$. We can take $K$ to be a ball with radius large enough so that $\omega(\frac{\eps}{2})_R \setminus K \subset \omega(\eps)$ (this can be justified by an argument similar to Lemma~\ref{lem:sejnotinomegaR}). We conclude that observability holds from $\omega(\eps)$ in time $T$. This proves the upper bound in~\eqref{eq:optimaltimeconicset}.

\medskip
\paragraph*{Step 6 \--- Lower bound on the optimal observation time.}
%%%%%%%%%%%%%%%%%%%%%%%%%%%%%%%%%%%%%%%%%%%%%%%%%%%%%%%%%%%%%%%%%
Fix $\eps \in (0, \pi/4)$. We recall that $\nu_2 > \nu_1$. Our objective is to exhibit trajectories $(x_1^t, x_2^t)$ that do not meet the set $\omega(2 \eps)$. They typically look like the one shown in Figure~\ref{fig:saturatingcurve}. Take $\delta > 0$ a small parameter to be chosen later. These trajectories we look for are of the form
\begin{equation} \label{eq:deftrajectorieswelookfor}
x_1^t = A_1 \sin\left(\pi \dfrac{\nu_1}{\nu_2}(1 - \delta) - \nu_1 t\right)
    \qquad \rm{and} \qquad
x_2^t = A_2 s \sin\left(\pi \delta + \nu_2 t\right) \,,
\end{equation}
with $s \in \{+1, -1\}$, and $A_1, A_2 > 0$ to be tuned properly later on as well.

Let us introduce three remarkable times $t_0$, $t_1$ and $t_2$: provided $\delta < 1/2$, the first zeroes of $x_1^t$ and $x_2^t$ in the interval $[0, T_0]$ coincide and are given by
\begin{equation*}
t_0 = \frac{\pi}{\nu_2}(1 - \delta) \,.
\end{equation*}
The next zero of $x_1^t$ is
\begin{equation*}
t_1 = t_0 + \dfrac{\pi}{\nu_1} \,.
\end{equation*}
As for $x_2^t$, its first zero that is strictly larger than $t_1$ is given by
\begin{equation} \label{eq:deft2withT0}
t_2
    = \dfrac{\pi}{\nu_2} \left(1 + \left\lfloor \delta + \dfrac{\nu_2}{\pi} t_1 \right\rfloor - \delta\right)
    = \dfrac{\pi}{\nu_2} \left(1 + \left\lfloor 1 + \dfrac{\nu_2}{\nu_1} \right\rfloor - \delta\right)
    = T_0 - \dfrac{\pi}{\nu_2} \delta \,.
\end{equation}
Notice that $t_2 \le T_0$. By construction, the interval $[t_1, t_2]$ has length $t_2 - t_1 \in (0, \frac{\pi}{\nu_2}]$, and $x_2^t$ has constant sign on this interval. We choose the sign $s$ involved in the definition~\eqref{eq:deftrajectorieswelookfor} of $x_2^t$ in such a way that $x_2^t \le 0$ on $[t_1, t_2]$. In particular, the projected trajectory $(x_1^t, x_2^t)$ cannot cross $C_{2 \eps}^2$ in the time interval $[t_1, t_2]$. Likewise, since $x_1^0 > 0$, it follows that $x_1^t \le 0$ on $[t_0, t_1]$, by definition of $t_0, t_1$. In particular, the curve $(x_1^t, x_2^t)$ cannot be in $C_{2 \eps}^1$ for $t \in [t_0, t_1]$.

Set $T = t_2 - \frac{\pi}{\nu_2} \delta$. In each interval $[0, t_0]$, $[t_0, t_1]$ and $[t_1, T]$, we want to exclude the possibility for the trajectory to be in $C_{2 \eps}^1$ or $C_{2 \eps}^2$ by suitably choosing the parameters $\delta, A_1, A_2$.

To achieve this goal, we are interested in estimating from above and from below $x_1^t$ and $x_2^t$ in these intervals. We first deal with $x_1^t$. Recalling that the sine function is $1$-Lipschitz, we know that
\begin{equation} \label{eq:upperboundx1t}
\abs*{x_1^t} \le A_1 \nu_1 \min\left(\abs*{t - t_0}, \abs*{t - t_1}\right) \,, \qquad \forall t \in \R \,.
\end{equation}
We obtain lower estimates by roughly bounding from below $\sin x$ on $[0, \pi]$ by the ``triangle" function $\frac{2}{\pi} \min(x, \pi - x)$. For $t \in [0, t_0]$, that leads to
\begin{align} \label{eq:lowerboundx1t-}
\abs*{x_1^t}
    &\ge A_1 \dfrac{2}{\pi} \min\left(\nu_1 \abs*{t_0 - t}, \abs*{\pi - \nu_1 (t_0 - t)}\right)
    \ge A_1 \dfrac{2 \nu_1}{\pi} \min\left(\abs*{t_0 - t}, \abs*{\dfrac{\pi}{\nu_1} - \dfrac{\pi}{\nu_2} + \delta \dfrac{\pi}{\nu_2} + t}\right) \nonumber\\
    &\ge A_1 \dfrac{2 \nu_1}{\pi} \min\left(\abs*{t_0 - t}, \dfrac{\pi}{\nu_1} - \dfrac{\pi}{\nu_2}\right) \,,
\end{align}
for $t \in [t_0, t_1]$ we obtain
\begin{equation} \label{eq:lowerboundx1t0}
\abs*{x_1^t}
    \ge A_1 \dfrac{2 \nu_1}{\pi} \min\left(\abs*{t - t_0}, \abs*{t - t_1}\right) \,,
\end{equation}
while for $t \in [t_1, T]$, we obtain
\begin{align*}
\abs*{x_1^t}
    &\ge A_1 \dfrac{2}{\pi} \min\left(\nu_1 \abs*{t - t_1}, \abs*{\pi - \nu_1 (t - t_1)}\right)
    \ge A_1 \dfrac{2 \nu_1}{\pi} \min\left(\abs*{t - t_1}, \abs*{\dfrac{\pi}{\nu_1} + t_1 - t}\right) \\
    &\ge A_1 \dfrac{2 \nu_1}{\pi} \min\left(\abs*{t_1 - t}, \dfrac{\pi}{\nu_1} + t_1 - T\right) \,.
\end{align*}
The last inequality rests on the fact that $\frac{\pi}{\nu_1} + t_1 \ge T$. More quantitatively, we have
\begin{equation} \label{eq:checkonT}
\begin{split}
\dfrac{\pi}{\nu_2} + t_1
    &= T_0 + \dfrac{\pi}{\nu_2} \left( 1 + \dfrac{\nu_2}{\nu_1} + 1 - \delta - 2 - \left\lfloor \dfrac{\nu_2}{\nu_1} \right\rfloor \right) \\
    &= T + \dfrac{\pi}{\nu_2} \delta + \dfrac{\pi}{\nu_2} \left( \dfrac{\nu_2}{\nu_1} - \left\lfloor \dfrac{\nu_2}{\nu_1} \right\rfloor \right) \,.
\end{split}
\end{equation}
In particular, it holds
\begin{equation} \label{eq:checkonTadd}
\dfrac{\pi}{\nu_1} + t_1 - T
    = \pi \left(\dfrac{1}{\nu_1} - \dfrac{1}{\nu_2}\right) + \left(\dfrac{\pi}{\nu_2} + t_1 - T\right)
    \ge \pi \left(\dfrac{1}{\nu_1} - \dfrac{1}{\nu_2}\right) + \dfrac{\pi}{\nu_2} \delta \,,
\end{equation}
which leads to
\begin{equation*} \label{eq:lowerboundx1t+}
\abs*{x_1^t}
    \ge A_1 \dfrac{2 \nu_1}{\pi} \min\left(\abs*{t_1 - t}, \pi \left(\dfrac{1}{\nu_1} - \dfrac{1}{\nu_2}\right)\right) \,, \qquad \forall t \in [t_1, T] \,.
\end{equation*}
We obtain a similar estimate for $x_2^t$: it vanishes at $t_0$ and $t_2$, so using again that the sine function is $1$-Lipschitz we get
\begin{equation} \label{eq:upperboundx2t}
\abs*{x_2^t}
	\le A_2 \nu_2 \min\left(\abs*{t_0 - t}, \abs*{t_2 - t}\right) \,,
		\qquad \forall t \in \R \,.
\end{equation}
Near, $t_1$, we want an accurate upper bound using the fact that $x_2^{t_1} \le 0$ (recall that we chose the sign $s$ in~\eqref{eq:deftrajectorieswelookfor} so that this is true): for any $t \in \R$, we have
\begin{equation} \label{eq:upperboundx2tspecial}
x_2^t
    \le x_2^t - x_2^{t_1}
    \le A_2 \nu_2 \abs{t - t_1} \,.
\end{equation}
As for a lower bound, we obtain for $t \in [0, t_0]$:
\begin{align} \label{eq:lowerboundx2t-}
\abs*{x_2^t}
    &\ge A_2 \dfrac{2}{\pi} \min\left(\nu_2 \abs*{t_0 - t}, \abs*{\pi - \nu_2 (t_0 - t)}\right)
    \ge A_2 \dfrac{2 \nu_2}{\pi} \min\left(\abs*{t_0 - t}, \delta \dfrac{\pi}{\nu_2} + t\right) \nonumber\\
    &\ge A_2 \dfrac{2 \nu_2}{\pi} \min\left(\abs*{t_0 - t}, \delta \dfrac{\pi}{\nu_2}\right) \,,
\end{align}
and for $t \in [t_1, T]$:
\begin{align} \label{eq:lowerboundx2t+}
\abs*{x_2^t}
    &\ge A_2 \dfrac{2}{\pi} \min\left(\abs*{\pi - \nu_2 (t_2 - t)}, \nu_2 \abs*{t_2 - t}\right)
    \ge A_2 \dfrac{2 \nu_2}{\pi} \min\left(\abs*{\dfrac{\pi}{\nu_2} - (t_2 - t)}, \abs*{t_2 - t}\right) \nonumber\\
    &\ge A_2 \dfrac{2 \nu_2}{\pi} \min\left(\abs{t - t_1}, \dfrac{\pi}{\nu_2} \delta\right) \,.
\end{align}
This time, the last inequality holds true since on the one hand, $t_2 - t \ge t_2 - T = \dfrac{\pi}{\nu_2} \delta$, and on the other hand, thanks to~\eqref{eq:checkonT} and~\eqref{eq:deft2withT0}, we check that for any $t \in [t_1, T]$:
\begin{equation*}
\dfrac{\pi}{\nu_2} - (t_2 - t)
    = (t - t_1) + \left(\dfrac{\pi}{\nu_2} + t_1\right) - t_2
    = (t - t_1) + \dfrac{\pi}{\nu_2} \left(\dfrac{\nu_2}{\nu_1} - \left\lfloor \dfrac{\nu_2}{\nu_1} \right\rfloor\right)
    \ge t - t_1 \,.
\end{equation*}

Now we show that the two conditions
\begin{align}
2 \dfrac{\nu_1}{\nu_2} \eps
	&\le \dfrac{A_2}{A_1} \delta \,, \label{eq:C1-}\\
2 \dfrac{\nu_2}{\nu_1} \eps
	&\le \dfrac{A_1}{A_2} \min\left(1, \dfrac{\nu_2}{\nu_1} - 1 \right) \label{eq:C2-}\,,
\end{align}
imply that the curve $(x_1^t, x_2^t)$ does not cross the set $\omega(2 \eps)$ in the interval $[0, T]$. We study the three intervals $[0, t_0]$, $[t_0, t_1]$ and $[t_1, T]$ separately.
 \begin{itemize}
     \item Let $t \in [0, t_0]$. On the one hand, the condition~\eqref{eq:C1-} implies that
      \begin{equation*}
     A_1 \nu_1 \abs{t - t_0} \dfrac{4 \eps}{\pi}
        \le A_2 \dfrac{2 \nu_2}{\pi} \min\left(\abs{t - t_0}, \delta \dfrac{\pi}{\nu_2}\right)
     \end{equation*}
     (recall that $t_0 \le \frac{\pi}{\nu_2}$ and $\delta \le 1/2$). Using that $\tan \eps \le \frac{\eps}{\pi/4}$ for $\eps \in [0, \frac{\pi}{4}]$, we obtain
     \begin{equation*}
     A_1 \nu_1 \abs{t - t_0} \tan \eps
        \le A_2 \dfrac{2 \nu_2}{\pi} \min\left(\abs{t - t_0}, \delta \dfrac{\pi}{\nu_2}\right) \,,
     \end{equation*}
     which leads to $\tan(\eps) \abs{x_1^t} \le \abs{x_2^t}$ in virtue of~\eqref{eq:upperboundx1t} and~\eqref{eq:lowerboundx2t-}. Therefore $(x_1^t, x_2^t) \not\in C_{2 \eps}^1$. On the other hand, the condition~\eqref{eq:C2-} implies that
      \begin{equation*}
     A_2 \nu_2 \abs{t_0 - t} \dfrac{4 \eps}{\pi} \le A_1 \dfrac{2 \nu_1}{\pi} \min\left(\abs*{t_0 - t}, \dfrac{\pi}{\nu_1} - \dfrac{\pi}{\nu_2}\right)
     \end{equation*}
     (recall again that $t_0 \le \frac{\pi}{\nu_2}$). Using that $\tan \eps \le \frac{\eps}{\pi/4}$ for $\eps \in [0, \frac{\pi}{4}]$, we obtain
     \begin{equation*}
     A_2 \nu_2 \abs{t_0 - t} \tan \eps \le A_1 \dfrac{2 \nu_1}{\pi} \min\left(\abs*{t_0 - t}, \dfrac{\pi}{\nu_1} - \dfrac{\pi}{\nu_2}\right) \,,
     \end{equation*}
     which leads to $\tan(\eps) \abs{x_2^t} \le \abs{x_1^t}$ in virtue of~\eqref{eq:upperboundx2t} and~\eqref{eq:lowerboundx1t-}. Therefore $(x_1^t, x_2^t) \not\in C_{2 \eps}^2$.
     \item On $[t_0, t_1]$, the situation is slightly simpler because we already know that $x_1^t \le 0$ on this interval, which means that the trajectory does not cross $C_{2 \eps}^1$ by construction. In addition, condition~\eqref{eq:C2-} implies that
     \begin{equation*}
     A_2 \nu_2 \min\left(\abs*{t_0 - t}, \abs*{t_1 - t}\right) \dfrac{4 \eps}{\pi}
        \le A_1 \dfrac{2 \nu_1}{\pi} \min\left(\abs*{t - t_0}, \abs*{t_1 - t}\right) \,.
     \end{equation*}
     Then~\eqref{eq:upperboundx2t}, \eqref{eq:upperboundx2tspecial} and~\eqref{eq:lowerboundx1t0} yield $\tan(\eps) x_2^t \le \abs{x_1^t}$, hence $(x_1^t, x_2^t) \not\in C_{2 \eps}^2$.
     \item We finally consider $t \in [t_1, T]$. Notice that by construction, it holds $x_2^t \le 0$ on $[t_1, T]$, so that the trajectory does not enter $C_{2 \eps}^2$. To disprove the fact that it meets $C_{2 \eps}^1$, we check that the condition~\eqref{eq:C1-} implies
     \begin{equation*}
     A_1 \nu_1 \abs*{t - t_1} \dfrac{4 \eps}{\pi}
        \le A_2 \dfrac{2 \nu_2}{\pi} \min\left(\abs*{t - t_1}, \dfrac{\pi}{\nu_2} \delta\right) \,,
     \end{equation*}
     owing to the fact that $\frac{\pi}{\nu_2} \ge T - t_1 \ge t - t_1$ (this can be deduced from~\eqref{eq:checkonTadd}). Then~\eqref{eq:upperboundx1t} and~\eqref{eq:lowerboundx2t+} lead to $\tan(\eps) \abs{x_1^t} \le \abs{x_2^t}$, which shows indeed that $(x_1^t, x_2^t) \not\in C_{2 \eps}^1$.
\end{itemize}
To sum up, in order to ensure that $t \mapsto (x_1^t, x_2^t)$ does not cross $\omega(2 \eps)$, it suffices to choose $A_1/ A_2$ properly, as well as $\delta$, so that~\eqref{eq:C1-} and~\eqref{eq:C2-} are fulfilled. If we set
\begin{equation} \label{eq:conditionratioA1A2}
\delta = \dfrac{4 \eps^2}{\min(1, \frac{\nu_2}{\nu_1} - 1)}
    \qquad \rm{and} \qquad
\dfrac{A_1}{A_2} = 2 \eps \dfrac{\frac{\nu_2}{\nu_1}}{\min(1, \frac{\nu_2}{\nu_1} - 1)} \,,
\end{equation}
we can check that these two conditions are indeed satisfied.

The conclusion is as follows: we consider a sequence of initial data of the form $\rho_n = (A_{1, n} \sin(\pi \frac{\nu_1}{\nu_2} (1 - \delta)), A_{2, n} s \sin(\pi \delta))$ with $A_{1, n}/A_{2, n}$ as in~\eqref{eq:conditionratioA1A2} and $A_{1, n}, A_{2, n} \to \infty$ as $n \to \infty$. The $x$ component of the trajectory $t \mapsto \phi^t(\rho_n)$ is then of the same form as the projected trajectory $(x_1^t, x_2^t)$ that we studied. Given that these trajectories do not cross $\omega(2 \eps)$, we conclude that the observability condition of Theorem~\ref{thm:main} is not true in time $T$, namely
\begin{equation*}
\frak{K}_{p_0}^\infty\left(\omega(2 \eps), T\right)
	= 0 \,.
\end{equation*}
Yet for any $R > 0$, as we have already seen earlier, $\omega(\eps)_R$ is contained in $\omega(2 \eps)$ modulo a compact set. Thus for any $R > 0$, there exists a compact set $K(R) \subset \R^d$ such that
\begin{equation*}
\frak{K}_{p_0}^\infty\left(\omega(\eps)_R \setminus K(R), T\right)
	\le \frak{K}_{p_0}^\infty\left(\omega(2 \eps), T\right)
	= 0 \,.
\end{equation*}
We conclude thanks to the necessary condition in Theorem~\ref{thm:main} that observability cannot hold in $\omega(\eps)$ in time $T$. It remains to see that by definition (recall~\eqref{eq:deft2withT0} and~\eqref{eq:conditionratioA1A2}), we have
\begin{equation*}
T
    = t_2 - \dfrac{\pi}{\nu_2} \delta
    = T_0 - 2 \dfrac{\pi}{\nu_2} \delta
    = T_0 - C \eps^2 \,.
\end{equation*}
This ends the proof of the lower bound of the optimal observation time. \hfill \qedsymbol

\subsection{Proof of Proposition~\ref{prop:isotropicconicalsets}} \label{subsec:proofisotropicconicalsets}

The aim of this proposition is to study observability from measurable conical sets for the (exact) isotropic harmonic oscillator. We first simplify the situation owing to periodicity properties of the isotropic quantum harmonic oscillator.

\medskip
\paragraph*{Step 1 \--- Upper bound of the optimal observation time.}
%%%%%%%%%%%%%%%%%%%%%%%%%%%%%%%%%%%%%%%%%%%%%%%%%%%%%%%%%%%%%%%%%%%%
First recall that there exists a complex number $z$ of modulus $1$ such that
\begin{equation} \label{eq:pointreflectionproperty}
\e^{\ii \frac{\pi}{\nu} P} u
	= z u(- \bigcdot) \,,
		\qquad \forall u \in L^2(\R^d) \,.
\end{equation}
See for instance\footnote{The property~\eqref{eq:pointreflectionproperty} can be derived from the fact that the spectrum of $P$ is made of half integer multiples of $\nu$, together with parity properties of eigenfunctions.}~\cite[Subsection 11.3.1]{Zworski:book} or~\cite[(4.26)]{FollandPhaseSpace}. In particular, the propagator $\e^{- \ii t P}$ is $\frac{2 \pi}{\nu}$-periodic modulo multiplication by $z^2$. This enables us to show that observability holds in some time $T$ if and only if it holds in time $\frac{2 \pi}{\nu}$: assume the Schrödinger equation is observable from $\omega \subset \R^d$ in some time $T$; let $k$ be an integer such that $k \frac{2 \pi}{\nu} \ge T$. The aforementioned $\frac{2 \pi}{\nu}$-periodicity of the harmonic oscillator leads to
\begin{align} \label{eq:reductionperiodicity}
\norm*{u}_{L^2(\R^d)}^2
    &\le C \int_0^T \norm*{\e^{- \ii t P} u}_{L^2(\omega)}^2 \dd t
    \le C \int_0^{k \frac{2 \pi}{\nu}} \norm*{\e^{- \ii t P} u}_{L^2(\omega)}^2 \dd t \nonumber\\
    &= C k \int_0^{\frac{2 \pi}{\nu}} \norm*{\e^{- \ii t P} u}_{L^2(\omega)}^2 \dd t \,,
\end{align}
for any $u \in L^2$ so that observability holds in $\omega$ in time $\frac{2 \pi}{\nu}$. In particular, the optimal observation time is always $\le \frac{2 \pi}{\nu}$. We can further reduce the observation time by $(2 C k)^{-1}$ (see Lemma~\ref{lem:obsopenintime}), so that the optimal observation time is in fact $T_\star < \frac{2 \pi}{\nu}$. Incidentally, the property~\eqref{eq:pointreflectionproperty} yields
\begin{equation} \label{eq:symmetrizationomega}
\int_0^{\pi/\nu} \norm*{\e^{- \ii t P} u}_{L^2(\omega \cup - \omega)}^2 \dd t
    \le \int_0^{2 \pi/\nu} \norm*{\e^{- \ii t P} u}_{L^2(\omega)}^2 \dd t
    \le 2 \int_0^{\pi/\nu} \norm*{\e^{- \ii t P} u}_{L^2(\omega \cup - \omega)}^2 \dd t \,,
\end{equation}
which will be useful later on.

\medskip
\paragraph*{Step 2 \--- Necessary condition.}
%%%%%%%%%%%%%%%%%%%%%%%%%%%%%%%%%%%%%%%
Assume observability holds from $\omega = \omega(\Sigma)$ in some time $T$. Let $k$ be a positive integer such that $k \frac{2 \pi}{\nu} \ge T$. Using~\eqref{eq:reductionperiodicity} and~\eqref{eq:symmetrizationomega}, we obtain
\begin{equation*}
\forall u \in L^2(\R^d) \,, \;
    \norm*{u}_{L^2(\R^d)}^2
        \le C \int_0^T \norm*{\e^{- \ii t P} u}_{L^2(\omega)}^2 \dd t
        \le 2 C k \int_0^{\pi/\nu} \norm*{\e^{- \ii t P} u}_{L^2(\omega \cup - \omega)}^2 \dd t \,.
\end{equation*}
We choose for $u$ a particular coherent state. Following Combescure and Robert~\cite{CombescureRobertBook}, for any $\rho_0 = (x_0, \xi_0)$, we set
\begin{equation*}
\varphi_{\rho_0}(x)
    = \left(\dfrac{\nu}{\pi}\right)^{d/4} \e^{- \frac{\ii}{2} \xi_0 \cdot x_0 + \ii \xi_0 \cdot x} \exp\left(- \dfrac{\nu}{2} \abs*{x - x_0}^2\right) \,.
\end{equation*}
Then it holds
\begin{equation} \label{eq:propagationofcoherentstate}
\e^{- \ii t P} \varphi_{\rho_0}
    = \e^{- \frac{\ii}{2} t \nu d} \varphi_{\rho_t} \,,
\end{equation}
where $\rho_t = \phi^t(\rho_0)$ is the evolution of $\rho_0$ in phase space along the Hamiltonian flow associated with $p(x, \xi) = \frac{1}{2}(\nu^2 \abs{x}^2 + \abs{\xi}^2)$, that is to say
\begin{equation*}
\rho_t = \left( \cos(\nu t) x_0 + \sin(\nu t) \dfrac{\xi_0}{\nu}, - \nu \sin(\nu t) x_0 + \cos(\nu t) \xi_0 \right) \,.
\end{equation*}
Equation~\eqref{eq:propagationofcoherentstate} can be checked by observing that the derivative of both sides agree, or by applying~\cite[Proposition 3]{CombescureRobertBook}.
Selecting an initial datum of the form $\rho_0 = (0, \xi_0)$ with a non-zero $\xi_0$, the observability inequality implies
\begin{align} \label{eq:observabilitycoherentstateisotropic}
1 = \norm*{\varphi_{\rho_0}}_{L^2(\R^d)}^2
    &\le C \int_0^T \norm*{\varphi_{\rho_t}}_{L^2(\omega)}^2 \dd t
    \le 2 k C \int_0^{\pi/\nu} \norm*{\varphi_{\rho_t}}_{L^2(\omega \cup - \omega)}^2 \dd t \nonumber\\
    &= 2 k C \left( \dfrac{\pi}{\nu} \right)^{d/2} \int_0^{\pi/\nu} \int_{\omega \cup - \omega} \abs*{\exp\left( - \dfrac{\nu}{2} \abs*{x - \sin(\nu t) \dfrac{\xi_0}{\nu}}^2 \right)}^2 \dd x \dd t \nonumber\\
    &= 4 k \dfrac{C}{\nu} \left( \dfrac{\pi}{\nu} \right)^{d/2} \int_0^{\pi/2} \int_{\omega \cup - \omega} \exp\left( - \nu \abs*{x - \sin(t) \dfrac{\xi_0}{\nu}}^2 \right) \dd x \dd t \,.
\end{align}
We used a change of variables in the integral over $t$ and the fact that $\sin(x) = \sin(\pi - x)$ to obtain the last equality. Next we truncate the integrals in $t$ and in $x$ using respectively a small parameter $\delta > 0$ and a large parameter $R > 0$:
\begin{multline*}
\int_0^{\pi/2} \left( \dfrac{\pi}{\nu} \right)^{d/2} \int_{\omega \cup - \omega} \exp\left( - \nu \abs*{x - \sin(t) \dfrac{\xi_0}{\nu}}^2 \right) \dd x \dd t \\
    \le \pi \delta + \int_{\frac{\pi}{2} \delta}^{\frac{\pi}{2} (1 - \delta)} \left( \dfrac{\pi}{\nu} \right)^{d/2} \left( \int_{\omega \cup - \omega} \exp\left( - \nu \abs*{x - \sin(t) \dfrac{\xi_0}{\nu}}^2 \right) \one_{B_R(\sin(t) \frac{\xi_0}{\nu})}(x) \dd x\right. \\
    	\left.+ \int_{\R^d \setminus B_R(0)} \e^{- \nu \abs{x}^2} \dd x \right) \dd t \,.
\end{multline*}
The rightmost integral is controlled by $c/R$ for some constant $c > 0$. Therefore it holds
\begin{multline*}
\int_0^{\pi/2} \left( \dfrac{\pi}{\nu} \right)^{d/2} \int_{\omega \cup - \omega} \exp\left( - \nu \abs*{x - \sin(t) \dfrac{\xi_0}{\nu}}^2 \right) \dd x \dd t \\
    \le \pi \delta + \dfrac{c}{R} + \left( \dfrac{\pi}{\nu} \right)^{d/2} \int_{\frac{\pi}{2} \delta}^{\frac{\pi}{2} (1 - \delta)} \abs*{\left(\omega \cup - \omega\right) \cap B_R\left(\sin(t) \dfrac{\xi_0}{\nu}\right)} \dd t \,.
\end{multline*}
We get rid of the sine in the right-hand side by noticing that $\cos t \ge 1 - \frac{2}{\pi} t \ge \delta$ for any $t \in [\frac{\pi}{2} \delta, \frac{\pi}{2} (1 - \delta)]$, and changing variables:
\begin{align*}
\int_{\frac{\pi}{2} \delta}^{\frac{\pi}{2} (1 - \delta)} \abs*{\left(\omega \cup - \omega\right) \cap B_R\left(\sin(t) \dfrac{\xi_0}{\nu}\right)} \dd t
    &\le \int_{\frac{\pi}{2} \delta}^{\frac{\pi}{2} (1 - \delta)} \abs*{\left(\omega \cup - \omega\right) \cap B_R\left(\sin(t) \dfrac{\xi_0}{\nu}\right)} \dfrac{\abs{\cos t}}{\delta} \dd t \\
    &= \dfrac{1}{\delta} \int_{\sin(\frac{\pi}{2} \delta)}^{\sin(\frac{\pi}{2} (1 - \delta))} \abs*{\left(\omega \cup - \omega\right) \cap B_R\left(s \dfrac{\xi_0}{\nu}\right)} \dd s \,.
\end{align*}
Using that $\sin x \ge \frac{2}{\pi} x$ on $[0, \frac{\pi}{2}]$, we finally deduce that
\begin{equation} \label{eq:towardslowerdensitywithremainder}
\begin{multlined}
\int_0^{\pi/2} \left( \dfrac{\pi}{\nu} \right)^{d/2} \int_{\omega \cup - \omega} \exp\left( - \nu \abs*{x - \sin(t) \dfrac{\xi_0}{\nu}}^2 \right) \dd x \dd t \\
    \qquad\qquad \le \pi \delta + \dfrac{c}{R} + \dfrac{1}{\delta} \left( \dfrac{\pi}{\nu} \right)^{d/2} \int_\delta^1 \abs*{\left(\omega \cup - \omega\right) \cap B_R\left(s \dfrac{\xi_0}{\nu}\right)} \dd s \,.
\end{multlined}
\end{equation}
We plug this into~\eqref{eq:observabilitycoherentstateisotropic} to obtain
\begin{equation} \label{eq:towardslowerdensity}
\dfrac{1}{2} = \dfrac{1}{2} \norm*{\varphi_{\rho_0}}_{L^2(\R^d)}^2
    \le 4 k \dfrac{C}{\delta \nu} \left( \dfrac{\pi}{\nu} \right)^{d/2} \int_\delta^1 \abs*{\left(\omega \cup - \omega\right) \cap B_R\left(s \dfrac{\xi_0}{\nu}\right)} \dd s \,,
\end{equation}
where we absorbed the remainder terms of~\eqref{eq:towardslowerdensitywithremainder} in the left-hand side by choosing $\delta$ sufficiently small and $R$ sufficiently large. We now use a scaling argument in the right-hand side, which is possible since the set $\omega \cup - \omega$ is conical: for any $s \in [\delta, 1]$, writing
\begin{equation} \label{eq:defthetar}
\theta_0 = \frac{\xi_0}{\abs{\xi_0}}
	\qquad \rm{and} \qquad
r = \frac{\nu R}{\delta \abs{\xi_0}} \,,
\end{equation}
we have
\begin{align*}
\abs*{\left(\omega \cup - \omega\right) \cap B_R\left(s \dfrac{\xi_0}{\nu}\right)}
    &= \left(s \frac{\abs{\xi_0}}{\nu}\right)^d \abs*{\left(\omega \cup - \omega\right) \cap B_{\frac{\nu R}{s \abs{\xi_0}}}\left(\theta_0\right)} \\
    &\le \left(\frac{R}{\delta}\right)^d r^{-d} \abs*{\left(\omega \cup - \omega\right) \cap B_r\left(\theta_0\right)} \,.
\end{align*}
After integrating over the $s$ variable, the estimate~\eqref{eq:towardslowerdensity} becomes
\begin{equation} \label{eq:firstfact}
1 = \norm*{\varphi_{\rho_0}}_{L^2(\R^d)}^2
    \le 8 k \dfrac{C}{\delta \nu} \left( \dfrac{\pi}{\nu} \right)^{d/2} \left(\frac{R}{\delta}\right)^d r^{-d} \abs*{\left(\omega \cup - \omega\right) \cap B_r\left(\theta_0\right)} \,.
\end{equation}
We now reformulate the right-hand side in terms of the lower density $\Theta_{\widehat \Sigma}^-$ defined in~\eqref{eq:deflowerdensity}. To do so, we observe that the triangle inequality yields for $r \in (0, 1)$:
\begin{equation*}
\forall x \in B_r(\theta_0) \,, \;\;
    \abs{\abs{x} - 1} \le \abs{x - \theta_0}
        \;\;\, \rm{and} \;\;\,
    \abs*{\dfrac{x}{\abs{x}} - \theta_0} = \abs*{\dfrac{x - \theta_0}{\abs{x}} + \dfrac{\theta_0}{\abs{x}}\left(1 - \abs{x}\right)} \le \dfrac{2 r}{1 - r} \,,
\end{equation*}
which in turn leads to
\begin{equation*}
B_r(\theta_0)
    \subset \set{x \in \R^d}{1 - r \le \abs{x} \le 1 + r \quad \rm{and} \quad \abs*{\dfrac{x}{\abs{x}} - \theta_0} \le \dfrac{2 r}{1 - r}} \,,
    	\qquad r \in (0, 1) \,.
\end{equation*}
Recall that if $\abs{\xi_0}$ is large enough, then~\eqref{eq:defthetar} implies $r \in (0, 1)$. We conclude by a spherical change of coordinates that
\begin{align} \label{eq:secondfact}
\abs*{\left(\omega \cup - \omega\right) \cap B_r\left(\theta_0\right)}
    &\le \int_{1 - r}^{1 + r} \int_{\sph^{d - 1}} \one_{\abs{\theta - \theta_0} \le \frac{2 r}{1 - r}} \one_{\omega \cup - \omega}(\tilde r \theta) c_d \tilde r ^{d - 1} \dd \sigma(\theta) \dd \tilde r \nonumber\\
    &\le \int_{1 - r}^{1 + r} \int_{\sph^{d - 1} \cap B_{\frac{2 r}{1 - r}}(\theta_0)} \one_{\widehat \Sigma}(\theta) c_d 2^{d - 1} \dd \sigma(\theta) \dd \tilde r \nonumber\\
    &= c_d 2^{d - 1} \times 2 r \sigma\left( \widehat \Sigma \cap B_{\frac{2 r}{1 - r}}(\theta_0) \right) \,.
\end{align}
In addition, one has
\begin{equation} \label{eq:thirdfact}
\sigma\left(B_r(\theta_0)\right) \le c_d' r^{d - 1} \,.
\end{equation}
(In the above estimates, $c_d$ and $c_d'$ are constants depending only on the dimension.) Combining~\eqref{eq:firstfact}, \eqref{eq:secondfact} and~\eqref{eq:thirdfact}, we obtain
\begin{align*}
1 = \norm*{\varphi_{\rho_0}}_{L^2(\R^d)}^2
    &\le c_d c_d' 2^{d + 3} k \dfrac{C}{\delta \nu} \left( \dfrac{\pi}{\nu} \right)^{d/2} \left(\frac{R}{\delta}\right)^d
    		\times \left(\dfrac{2}{1 - r}\right)^{d-1} \dfrac{1}{c_d' (\frac{2 r}{1 - r})^{d - 1}} \sigma\left( \widehat \Sigma \cap B_{\frac{2 r}{1 - r}}(\theta_0) \right) \\
    &\le c_d c_d' 2^{d + 3} k \dfrac{C}{\delta \nu} \left( \dfrac{\pi}{\nu} \right)^{d/2} \left(\frac{R}{\delta}\right)^d \left(\dfrac{2}{1 - r}\right)^d \dfrac{\sigma\left( \widehat \Sigma \cap B_{\frac{2 r}{1 - r}}(\theta_0) \right)}{\sigma\left( B_{\frac{2 r}{1 - r}}(\theta_0) \right)} \,.
\end{align*}
Recalling that $r$ behaves as $1/\abs{\xi_0}$, it remains to let $\xi_0 \to \infty$ with $\frac{\xi_0}{\abs{\xi_0}} = \theta_0$ arbitrary, to deduce that
\begin{equation*}
1
	\le c_d c_d' 2^{d + 3} k \dfrac{C}{\delta \nu} \left( \dfrac{\pi}{\nu} \right)^{d/2} \left(\frac{2R}{\delta}\right)^d \Theta_{\widehat \Sigma}^-(\theta_0) \,,
		\qquad \forall \theta_0 \in \sph^{d - 1} \,.
\end{equation*}
This concludes the proof of the necessary condition.

\medskip
\paragraph*{Step 3 \--- Sufficient condition.}
%%%%%%%%%%%%%%%%%%%%%%%%%%%%%%%%%%%%%%%%%%%%%%
Write for short $\omega = \omega(\Sigma)$ again. The fact that $\widehat \Sigma = \Sigma \cup - \Sigma$ has full measure, namely $\sigma(\sph^{d - 1} \setminus \widehat \Sigma) = 0$, implies that $\R^d \setminus (\omega \cup - \omega)$ is Lebesgue negligible (recall the definition of $\omega(\Sigma)$ in~\eqref{eq:defomegaIconical}). Therefore the left-hand side of~\eqref{eq:symmetrizationomega} with $k = 1$ yields
\begin{equation*}
\int_0^{2 \pi/\nu} \norm*{\e^{- \ii t P} u}_{L^2(\omega)}^2 \dd t
    \ge \int_0^{\pi/\nu} \norm*{\e^{- \ii t P} u}_{L^2(\omega \cup - \omega)}^2 \dd t
    = \int_0^{\pi/\nu} \norm*{\e^{- \ii t P} u}_{L^2(\R^d)}^2 \dd t 
    = \frac{\pi}{\nu} \norm*{u}_{L^2(\R^d)}^2 \,,
\end{equation*}
where we used the fact that the propagator is an isometry. This completes the proof.~\hfill \qedsymbol

\section{Proofs of observability results from spherical sets} \label{sec:proofspherical}

In this section, we give proofs of the results presented in Subsection~\ref{subsub:sphericalsets}, which concern observation sets that are spherical in the sense of~\eqref{eq:sphericalset}. Propositions~\ref{prop:sphericalsetwithaxiallysymmetricpotential} and~\ref{prop:anisotropicsphericalsets} are proved in Subsections~\ref{subsec:proofaxiallysymmetricpotential} and~\ref{subsec:proofanisotropicsphericalsets} respectively.

\subsection{Proof of Proposition~\ref{prop:sphericalsetwithaxiallysymmetricpotential}} \label{subsec:proofaxiallysymmetricpotential}

The rotation $S_\theta$ of angle $\theta$ reads
\begin{equation*}
S_\theta y = \left(\cos \theta y_1 + \sin \theta y_2, - \sin \theta y_1 + \cos \theta y_2, y_3, \ldots, y_d\right) \,, \;\;\; y = (y_1, y_2, \ldots, y_d) \in \R^d \,.
\end{equation*}
In the sequel, we set $L_0$ to be the two dimensional plane spanned by the vectors
\begin{equation*}
e_1 = M(1, 0, 0, \ldots, 0)
    \qquad \rm{and} \qquad
e_2 = M(0, 1, 0, 0, \ldots 0) \,,
\end{equation*}
The two linear maps
\begin{equation*}
\Pi_{L_0} = \dfrac{1}{2} M (\id - S_\pi) M^{-1}
    \qquad \rm{and} \qquad
\Pi_{L_0^\perp} = \dfrac{1}{2} M (\id + S_\pi) M^{-1}
\end{equation*}
are the orthogonal projectors on $L_0$ and $L_0^\perp$ respectively, since $M$ is orthogonal. With the notation of assumption~\ref{it:assum3}, we can write, with a slight abuse of notation,
\begin{equation} \label{eq:abuseofnotationV}
V(x_0)
	= \tilde V_0\left(\abs*{M^{-1} x_0}\right) \,,
		\qquad \forall x_0 \in L_0 \,.
\end{equation}

Let us investigate the properties of the gradient of $V$ on $L_0$.

\begin{lemm}
Let $x_0 \in L_0$. Then it holds
\begin{equation*}
\nabla V(x_0) \in L_0
	\qquad \rm{and} \qquad
\exists c = c(\abs{x_0}) \ge 0 : \quad \nabla V(x_0) = c x_0 \,.
\end{equation*}
\end{lemm}

\begin{proof}
Assumptions~\ref{it:assum1} and~\ref{it:assum2} (with $\theta = \pi$) yield for any $x \in \R^d$:
\begin{equation} \label{eq:twosymmetriesgradientV}
- \nabla V(-x) = \nabla V(x)
    \qquad \rm{and} \qquad
M S_{-\pi} M^{-1} \nabla V(M S_\pi M^{-1} x) = \nabla V(x) \,.
\end{equation}
Yet since $x_0 \in L_0$, we have $\Pi_{L_0} x_0 = x_0$ so that
\begin{equation} \label{eq:rotationpi}
x_0
	= - M S_\pi M^{-1} x_0 \,,
\end{equation}
and noticing that $S_\pi = S_{-\pi}$, we obtain combining the two equations~\eqref{eq:twosymmetriesgradientV}:
\begin{equation*}
\nabla V(x_0)
	= - \nabla V(- x_0)
    = - M S_\pi M^{-1} \nabla V(- M S_\pi M^{-1} x_0)
    = - M S_\pi M^{-1} \nabla V(x_0) \,.
\end{equation*}
That means exactly that $\Pi_{L_0^\perp} \nabla V(x_0) = 0$, or in other words, $\nabla V(x_0) \in L_0$.

Next we prove that $\nabla V(x_0)$ is collinear with $x_0$. We first get rid of the case $x_0 = 0$: the first equation in~\eqref{eq:twosymmetriesgradientV} implies that $\nabla V(0) = 0$. From now on, we assume that $x_0 \neq 0$. We compute
\begin{equation*}
\dfrac{\dd}{\dd \theta} M S_\theta M^{-1}
    = \dfrac{\dd}{\dd \theta} \left(M S_\theta M^{-1} \Pi_{L_0} + M S_\theta M^{-1} \Pi_{L_0^\perp} \right)
    = M S_{\theta + \pi/2} M^{-1} \Pi_{L_0} \,.
\end{equation*}
This is true because $M S_\theta M^{-1} \Pi_{L_0^\perp}$ is independent of $\theta$ ($M S_\theta M^{-1}$ is the identity in $L_0^\perp$). Therefore, differentiating the equality $V(x) = V(M S_\theta M^{-1} x)$ at $\theta = 0$, we obtain
\begin{align*}
0
    &= \dfrac{\dd}{\dd \theta} V(x_0)_{\vert \theta = 0}
    = \dfrac{\dd}{\dd \theta} V(M S_\theta M^{-1} x_0)_{\vert \theta = 0}
    = \nabla V(x_0) \cdot M S_{\pi/2} M^{-1} \Pi_{L_0} x_0 \\
    &= \nabla V(x_0) \cdot M S_{\pi/2} M^{-1} x_0 \,.
\end{align*}
This means that $\nabla V(x_0)$ is orthogonal to $M S_{\pi/2} M^{-1} x_0$. Yet the plane $L_0$ is invariant by $M S_\theta M^{-1}$ and $x_0 \perp M S_{\pi/2} M^{-1} x_0$. Since $\nabla V(x_0) \in L_0$ and $L_0$ has dimension $2$, we deduce that $\nabla V(x_0) = c x_0$ for some $c \in \R$. We claim that $c \ge 0$ as a consequence of assumption~\ref{it:assum3} that $\tilde V_0$ is non-decreasing. Indeed for $t > 0$ close to zero, using~\eqref{eq:abuseofnotationV}, the Taylor formula at order one yields
\begin{equation*}
0
    \le \tilde V_0\left((1 + t) \abs*{M^{-1} x_0}\right) - \tilde V_0\left(\abs*{M^{-1} x_0}\right)
    = V(x_0 + t x_0) - V(x_0)
    = t \nabla V(x_0) \cdot x_0 + o(t) \,.
\end{equation*}
Dividing by $t > 0$, we find that $\nabla V(x_0) \cdot x_0 = c \abs{x_0}^2 \ge 0$. Thus $c = \nabla V(x_0) \cdot \frac{x_0}{\abs{x_0}^2}$ depends only on $\abs{x_0}$ since $V$ restricted to $L_0$ is radial, and the proof is complete.
\end{proof}

This lemma allows us to exhibit periodic circular orbits of the Hamiltonian flow of $p$. For any $x_0 \in L_0$, denoting by $c$ the scalar such that $\nabla V(x_0) = c x_0$, the phase space curve
\begin{equation} \label{eq:candidatecurve}
x^t = M S_{\sqrt{c} t} M^{-1} x_0 \,, \qquad \xi^t = \sqrt{c} M S_{\sqrt{c} t + \pi/2} M^{-1} x_0
\end{equation}
is the trajectory of the Hamiltonian flow with initial data $(x_0, \sqrt{c} M S_{\pi/2} M^{-1} x_0)$. This follows from uniqueness in the Picard-Lindelöf Theorem, since the above curve solves on the one hand:
\begin{equation*}
\dfrac{\dd}{\dd t} x^t
    = \sqrt{c} M S_{\sqrt{c} t + \pi/2} M^{-1} \Pi_{L_0} x_0
    = \xi^t \,,
\end{equation*}
and on the other hand, in view of~\eqref{eq:rotationpi} and observing that $\abs{x^t} = \abs{x_0}$ for any $t$,
\begin{equation*}
\dfrac{\dd}{\dd t} \xi^t
    = c M S_{\sqrt{c} t + \pi} M^{-1} \Pi_{L_0} x_0
    = c M S_{\sqrt{c} t} M^{-1} \left(\Pi_{L_0^\perp} - \Pi_{L_0}\right) x_0
    = - c x^t
    = - \nabla V(x^t) \,.
\end{equation*}

To conclude, we argue as follows: since by assumption observability holds from $\omega(I)$ in time $T > 0$, the necessary condition of Theorem~\ref{thm:main} implies that there exists $R > 0$ such that
\begin{equation*}
\exists \epsilon > 0, \exists A > 0 : \forall \abs{\rho} \ge A \,, \qquad
	\int_0^T \one_{\omega(I)_R \times \R^d}\left(\phi^t(\rho)\right) \dd t \ge \epsilon \,.
\end{equation*}
Let $x_0 \in L_0$ be such that $\abs{x_0} \ge A$. We consider the Hamiltonian trajectory issued from the point $(x_0, \sqrt{c(x_0)} M^{-1} S_{\pi/2} M x_0)$ constructed in~\eqref{eq:candidatecurve}. Then $\abs{x^t}$ is constant over time, which implies that
\begin{equation*}
\epsilon
	\le \int_0^T \one_{\omega(I)_R \times \R^d}\left(\phi^t(\rho)\right) \dd t
	= \int_0^T \one_{\omega(I)_R}\left(x^t\right) \dd t
	= T \one_{I_R}\left(\abs{x_0}\right) \,,
\end{equation*}
whence $\abs{x_0} \in I_R$. We deduce that
\begin{equation*}
\forall s \in \R_+ \,, \qquad
	I_R \cap [s, s + A] \neq \varnothing \,,
\end{equation*}
which implies the desired result~\eqref{eq:nobiggapinI} with $r = A + 2 R$. \hfill \qedsymbol

\subsection{Proof of Proposition~\ref{prop:anisotropicsphericalsets}} \label{subsec:proofanisotropicsphericalsets}

Firstly we assume that $\frac{\nu_2}{\nu_1}$ is rational: we write it as an irreducible fraction $\frac{p}{q}$. The number $T = \frac{2 \pi}{\nu_2} p = \frac{2 \pi}{\nu_1} q$ is the period of the Hamiltonian flow associated with $\frac{1}{2}(x \cdot A x + \abs{\xi}^2)$. Without loss of generality, we can assume that $A$ is diagonal, and that the eigenvectors associated with $\nu_1^2$ and $\nu_2^2$ are the vectors $(1, 0)$ and $(0, 1)$ of the canonical basis of $\R^2$. As mentioned in~\eqref{eq:Lambda(mu)geom}, we seek to compute a sharp uniform upper bound of the ratio between the minimal and the maximal value of the norm of projected trajectories $\abs{x^t}$ on the time interval $[0, T]$. More explicitly, we are interested in the quantity
\begin{equation} \label{eq:defLambda_0}
\Lambda_0
    = \sup_{\rho_0 \in \R^4 \setminus \{0\}} \dfrac{\min_{t \in [0, T]} \abs{(\pi \circ \phi^t)(\rho_0)}}{\max_{t \in [0, T]} \abs{(\pi \circ \phi^t)(\rho_0)}} \,,
\end{equation}
where we recall that $\pi : (x, \xi) \mapsto x$. We start with two remarks, related to explicit expressions of the Hamiltonian flow. First we can replace the supremum on $\R^4$ by a maximum on a compact set parametrizing trajectories, e.g.\ the unit sphere $\sph^3$, because the Hamiltonian flow is homogeneous of degree $1$, that is $\phi^t(\lambda \rho_0) = \lambda \phi^t(\rho_0)$ for any scalar $\lambda \in \R$ (it fact $\phi^t$ is a linear map for all $t$). Second, since $\abs{x^t}^2 = \abs{x_1^t}^2 + \abs{x_2^t}^2$, it will be easier to compute $\Lambda_0^2$. In view of these remarks, and writing the Hamiltonian trajectories in action-angle coordinates:
\begin{equation} \label{eq:hamiltonianflowinactionanglecoordinatesomegaI}
x_1^t = A_1 \sin(\nu_1 t + \theta_1)
    \qquad \rm{and} \qquad
x_2^t = A_2 \sin(\nu_2 t + \theta_2) \,,
\end{equation}
we want to study
\begin{align*}
\Lambda_0^2
    &= \sup_{\substack{A_1^2 + A_2^2 = 1 \\ \theta_1, \theta_2 \in \R}} \dfrac{\min_{t \in [0, T]} \left(A_1^2 \sin^2(\nu_1 t + \theta_1) + A_2^2 \sin^2(\nu_2 t + \theta_2)\right)}{\max_{t \in [0, T]} \left(A_1^2 \sin^2(\nu_1 t + \theta_1) + A_2^2 \sin^2(\nu_2 t + \theta_2)\right)} \\
    &= \sup_{\substack{\lambda \in [0, 1] \\ \theta_1, \theta_2 \in \R}} \dfrac{\min_{t_1 \in [0, T]} \left((1 - \lambda) \sin^2(\nu_1 t_1 + \theta_1) + \lambda \sin^2(\nu_2 t_1 + \theta_2)\right)}{\max_{t_2 \in [0, T]} \left((1 - \lambda) \sin^2(\nu_1 t_2 + \theta_1) + \lambda \sin^2(\nu_2 t_2 + \theta_2)\right)} \\
    &= \sup_{\substack{\lambda \in [0, 1] \\ \theta_1, \theta_2 \in \R}} \dfrac{\min_{t_1 \in [0, T]} \left((1 - \lambda) \sin^2(\nu_1 t_1 + \theta_1) + \lambda \sin^2(\nu_2 t_1 + \theta_2)\right)}{1 - \min_{t_2 \in [0, T]} \left((1 - \lambda) \cos^2(\nu_1 t_2 + \theta_1) + \lambda \cos^2(\nu_2 t_2 + \theta_2)\right)} \,.
\end{align*}
In view of the periodicity in the variables $\theta_1$ and $\theta_2$, the supremum in the variables $\lambda, \theta_1, \theta_2$ is in fact a supremum over $(\lambda, \theta_1, \theta_2) \in [0, 1] \times [0, 2 \pi] \times [0, 2 \pi]$. A compactness and continuity argument shows that this supremum is attained for some triple $(\lambda, \theta_1, \theta_2)$. Furthermore, one can check that $\displaystyle \max_{\lambda, \theta_1, \theta_2} = \max_{\theta_1, \theta_2} \max_\lambda$. Thus we should simplify the problem first by considering fixed values for $\theta_1$ and $\theta_2$, and maximizing with respect to these variables ultimately. Therefore our objective is to compute
\begin{equation} \label{eq:defLambdathetapm}
\Lambda_{\theta_1, \theta_2}^2
    = \max_{\lambda \in [0, 1]} \dfrac{\min_{t_1 \in [0, T]} \left((1 - \lambda) \sin^2(\nu_1 t_1 + \theta_1) + \lambda \sin^2(\nu_2 t_1 + \theta_2)\right)}{1 - \min_{t_2 \in [0, T]} \left((1 - \lambda) \cos^2(\nu_1 t_2 + \theta_1) + \lambda \cos^2(\nu_2 t_2 + \theta_2)\right)} \,.
\end{equation}
We can further simplify this by rewriting in more pleasant terms the minima in the numerator and the denominator. It relies on the following fact.

\medskip
\paragraph*{Step 1 \--- Simplification of the optimization problem.}
%%%%%%%%%%%%%%%%%%%%%%%%%%%%%%%%%%%%%%%%%%%%%%%%%%%%%%%%%%%%%%%%%%%%
The minimum we want to estimate involves a sum of two squared sine functions that oscillate at different frequencies. Intuitively, it looks reasonable that the minimum of such a sum is attained between two zeroes that achieve the minimal distance between a zero of the first sine function, and a zero of the second. This is a motivation to introduce
\begin{equation} \label{eq:defofd_0}
d_0 = d_0(\theta_1, \theta_2)
    = \dfrac{4 p q}{T} \min_{\substack{\sin(\nu_j t_j + \theta_j) = 0 \\ j = 1, 2}} \abs{t_1 - t_2} \,.
\end{equation}
It is indeed a minimum, and not only an infimum, thanks to the rational ratio between $\nu_1$ and $\nu_2$, or equivalently, thanks to the periodicity of the Hamiltonian flow. We can give an explicit expression of this quantity reasoning as follows: the numbers $t_1$ and $t_2$ are such that $\sin(\nu_j t_j + \theta_j) = 0$, $j = 1, 2$, if and only if there exist two integers $k_1$ and $k_2$ such that
\begin{equation*}
\nu_j t_j + \theta_j = k_j \pi \,.
\end{equation*}
Therefore it holds
\begin{equation*}
\abs*{t_1 - t_2}
    = \abs*{\pi \left(\dfrac{k_1}{\nu_1} - \dfrac{k_2}{\nu_2}\right) - \left(\dfrac{\theta_1}{\nu_1} - \dfrac{\theta_2}{\nu_2}\right)}
    = \dfrac{T}{2 p q} \abs*{\left(k_1 p - k_2 q\right) - \left(p \dfrac{\theta_1}{\pi} - q \dfrac{\theta_2}{\pi}\right)} \,.
\end{equation*}
Yet since $p$ and $q$ are coprime integers, it follows from Bézout's identity that $k_1 p - k_2 q$ can take any value in $\Z$ when we vary $k_1$ and $k_2$. We deduce that
\begin{equation*}
d_0
    = 2 \dist\left(p \dfrac{\theta_1}{\pi} - q \dfrac{\theta_2}{\pi}, \Z\right)
    = \dist\left(p \dfrac{\theta_1}{\pi/2} - q \dfrac{\theta_2}{\pi/2}, 2 \Z\right) \,.
\end{equation*}
Incidentally, this expression implies that $d_0 \in [0, 1]$. Now we claim that
\begin{equation} \label{eq:minintvsminins}
\begin{multlined}
\min_{t \in [0, T]} \left((1 - \lambda) \sin^2(\nu_1 t + \theta_1) + \lambda \sin^2(\nu_2 t + \theta_2)\right) \\
    \qquad\qquad = \min_{s \in [0, 1]} \left((1 - \lambda) \sin^2\left(\dfrac{\pi/2}{p} s d_0\right) + \lambda \sin^2\left(\dfrac{\pi/2}{q} (1 - s) d_0\right)\right) \,.
\end{multlined}
\end{equation}
It amounts to prove that the minimum in $t$ in the left-hand side of~\eqref{eq:minintvsminins} is attained between two zeros $t_1, t_2$ of $\sin(\nu_1 t + \theta_1)$ and $\sin(\nu_2 t + \theta_2)$ such that $\abs{t_2 - t_1} = \frac{T}{4 p q} d_0$. We first show that the minimum in $s$ (in the right-hand side) is less than the minimum in $t$ (in the left-hand side). To do so, we pick $t_0 \in [0, T]$ that attains the minimum in $t$. We choose $t_j$ two zeroes of $\sin(\nu_j t + \theta_j)$ respectively, $j = 1, 2$, that are the closest possible to $t_0$. Due to periodicity, they verify $\abs{t_j - t_0} \le \dfrac{\pi}{2 \nu_j}$. That $t_0$ attains the minimum means that it is a critical point of the function
\begin{equation} \label{eq:functionnormxtoft}
\begin{multlined}
F : t \longmapsto (1 - \lambda) \sin^2\left(\nu_1 t + \theta_1\right) + \lambda \sin^2\left(\nu_2 t + \theta_2\right) \\
    \qquad\qquad\qquad = (1 - \lambda) \sin^2\left(\nu_1 (t - t_1)\right) + \lambda \sin^2\left(\nu_2 (t - t_2)\right) \,.
\end{multlined}
\end{equation}
Classical trigonometry formulae then yield
\begin{equation} \label{eq:functionoftderivative}
(1 - \lambda) \nu_1 \sin\left(2 \nu_1 (t_0 - t_1)\right) + \lambda \nu_2 \sin\left(2 \nu_2 (t_0 - t_2)\right) = F'(t_0) = 0 \,.
\end{equation}
Recalling that $\abs{2 \nu_j (t_0 - t_j)} \le \pi$, we see that $\sin(2 \nu_j (t_0 - t_j))$ is of the same sign as $t_0 - t_j$, thus leading to the condition that
\begin{equation*}
    (t_0 - t_1) (t_0 - t_2) \le 0 \,,
\end{equation*}
or in other words, $t_0$ lies between $t_1$ and $t_2$. Let $s_0 \in [0, 1]$ be such that $t_0 = (1 - s_0) t_1 + s_0 t_2$. We obtain
\begin{align*}
F(t_0)
    &= (1 - \lambda) \sin^2\left(\nu_1 (t_0 - t_1)\right) + \lambda \sin^2\left(\nu_2 (t_0 - t_2)\right) \\
    &= (1 - \lambda) \sin^2\left(\nu_1 s_0 (t_2 - t_1)\right) + \lambda \sin^2\left(\nu_2 (1 - s_0) (t_1 - t_2)\right) \,.
\end{align*}
We finally use that $\abs{t_1 - t_2} \ge \frac{T}{4 p q} d_0$ and the monotonicity of the sine function on $[0, \frac{\pi}{2}]$ to deduce one inequality in~\eqref{eq:minintvsminins}, namely:
\begin{equation} \label{eq:oneinequality}
\min_{t \in [0, T]} F(t)
	\ge \min_{s \in [0, 1]} \left((1 - \lambda) \sin^2\left(\dfrac{\pi/2}{p} s d_0\right) + \lambda \sin^2\left(\dfrac{\pi/2}{q} (1 - s) d_0\right)\right) \,.
\end{equation}
To check the converse inequality, we proceed as follows: we pick $t_1$ and $t_2$, zeroes of $\sin(\nu_j t + \theta_j)$ respectively, that satisfy $\abs{t_1 - t_2} = \frac{T}{4 p q} d_0$. Denote by $J$ the closed interval with endpoints $t_1, t_2$. %Then we claim that
%\begin{equation} \label{eq:claimminimum}
%\exists t_0 \in J : \qquad
%	F(t_0) \le F(t) \,, \;\; \forall t \in J \,,
%		\quad \rm{and} \quad
%	F'(t_0) = 0 \,.
%\end{equation}
%Indeed, if $d_0 = 0$, we have $t_0 = t_1 = t_2$ so the derivative as computed in~\eqref{eq:functionoftderivative} clearly vanishes. Otherwise, we find from~\eqref{eq:functionoftderivative} that
%\begin{equation*}
%F'(\inf J) < 0
%	\qquad \rm{and} \qquad
%F'(\sup J) > 0 \,,
%\end{equation*}
%so that the minimum is attained in the interior of this interval, hence the claim~\eqref{eq:claimminimum}. Moreover, $t_1$ and $t_2$ attain the minima:
%\begin{equation*}
%\abs{t_j - t_0}
%	= \min_{\sin(\nu_j t + \theta_j) = 0} \abs{t - t_0} \,,
%		\qquad j = 1, 2 \,,
%\end{equation*}
%otherwise one would find a contradiction with the fact that $\abs{t_1 - t_2} = \frac{T}{4 p q} d_0$ is minimal.
Let $t_0 \in J$ be a point where $F$ restricted to $J$ attains its minimum. Then introducing a parameter $s \in [0, 1]$ such that $t = (1 - s) t_1 + s t_2$, we obtain
\begin{equation*}
F(t_0)
    \le F(t)
    = (1 - \lambda) \sin^2\left(\dfrac{\pi/2}{p} s d_0\right) + \lambda \sin^2\left(\dfrac{\pi/2}{q} (1 - s) d_0\right) \,,
\end{equation*}
for all $s \in [0, 1]$.
This results in
\begin{equation*}
\min_{t \in [0, T]} F(t)
	\le \min_{s \in [0, 1]} \left( (1 - \lambda) \sin^2\left(\dfrac{\pi/2}{p} s d_0\right) + \lambda \sin^2\left(\dfrac{\pi/2}{q} (1 - s) d_0\right) \right) \,,
\end{equation*}
which shows together with~\eqref{eq:oneinequality} that~\eqref{eq:minintvsminins} is true. We observe in the definition of $\Lambda_{\theta_1, \theta_2}$ (see~\eqref{eq:defLambdathetapm}) that a similar minimum is involved with cosine functions instead of sine functions. To reduce to the case of sine functions and use~\eqref{eq:minintvsminins}, we simply recall that $\cos(x) = \sin(x + \frac{\pi}{2})$. We obtain
\begin{align*}
\min_{t \in [0, T]} \left((1 - \lambda) \right. &\left. \cos^2(\nu_1 t + \theta_1) + \lambda \cos^2(\nu_2 t + \theta_2)\right) \\
    &= \min_{t \in [0, T]} \left((1 - \lambda) \sin^2\left(\nu_1 t + \theta_1 + \dfrac{\pi}{2}\right) + \lambda \sin^2\left(\nu_2 t + \theta_2 + \dfrac{\pi}{2}\right)\right) \\
    &= \min_{s \in [0, 1]} \left((1 - \lambda) \sin^2\left(\dfrac{\pi/2}{p} s d_{\pi/2}\right) + \lambda \sin^2\left(\dfrac{\pi/2}{q} (1 - s) d_{\pi/2}\right)\right) \,,
\end{align*}
where we set (recall the definition of $d_0$ in~\eqref{eq:defofd_0})
\begin{equation*}
d_{\pi/2} = d_{\pi/2}(\theta_1, \theta_2)
    = d_0\left(\theta_1 + \dfrac{\pi}{2}, \theta_2 + \dfrac{\pi}{2}\right)
    = \dist\left(p \dfrac{\theta_1}{\pi/2} - q \dfrac{\theta_2}{\pi/2} + p - q, 2 \Z\right) \,.
\end{equation*}
Depending on whether $p$ and $q$ have the same parity, we can state that
\begin{equation*}
d_{\pi/2} =
\left\{
\begin{aligned}
&\quad d_0 \qquad &\rm{if} \; && p - q \equiv 0 \pmod{2} \\
&1 - d_0 \qquad &\rm{if} \; && p - q \equiv 1 \pmod{2}
\end{aligned}
\right. \,.
\end{equation*}
With this at hand, we can rewrite $\Lambda_{\theta_1, \theta_2}^2$ defined in~\eqref{eq:defLambdathetapm} as
\begin{equation*}
\Lambda_{\theta_1, \theta_2}^2
    = \max_{\lambda \in [0, 1]} \dfrac{\min_{s_1 \in [0, 1]} \left((1 - \lambda) \sin^2\left(\frac{\pi/2}{p} s_1 d_0\right) + \lambda \sin^2\left(\frac{\pi/2}{q} (1 - s_1) d_0\right)\right)}{1 - \min_{s_2 \in [0, 1]} \left((1 - \lambda) \sin^2\left(\frac{\pi/2}{p} s_2 d_{\pi/2}\right) + \lambda \sin^2\left(\frac{\pi/2}{q} (1 - s_2) d_{\pi/2}\right)\right)} \,.
\end{equation*}

\medskip
\paragraph*{Step 2 \--- Computation of $\Lambda_{\theta_1, \theta_2}^2$.}
%%%%%%%%%%%%%%%%%%%%%%%%%%%%%%%%%%%%%%%%%%%%%%%%%%%%%%%%%%%%%%%%%%%
We set for any $\lambda \in [0, 1]$ and $s \in [0, 1]$:
\begin{equation*}
g_\lambda(s)
    = g_{\lambda, d_0}(s)
    = (1 - \lambda) \sin^2\left(\frac{\pi/2}{p} s d_0\right) + \lambda \sin^2\left(\frac{\pi/2}{q} (1 - s) d_0\right) \,.
\end{equation*}
In the perspective of computing $\Lambda_{\theta_1, \theta_2}^2$, we first show the following result.

\begin{lemm}
It holds
\begin{equation*}
\max_{\lambda \in [0, 1]} \min_{s \in [0, 1]} g_\lambda(s)
    = g_{\lambda_0}(s_0)
    = \sin^2\left(\dfrac{\pi/2}{p + q} d_0\right) \,,
\end{equation*}
where $s_0 = \frac{p}{p + q}$ and $\lambda_0 = \frac{q}{p + q}$.
\end{lemm}

\begin{proof}
Firstly, we observe that $g_\lambda(s_0)$ is independent of $\lambda$, since it solves
\begin{equation*}
\sin^2\left(\frac{\pi/2}{p} s_0 d_0\right) = \sin^2\left(\frac{\pi/2}{q} (1 - s_0) d_0\right) \,.
\end{equation*}
This remarkable property implies that for any $\lambda \in [0, 1]$, it holds
\begin{equation*}
\forall \lambda' \in [0, 1] \,, \qquad
	\min_{s \in [0, 1]} g_{\lambda'}(s)
		\le g_{\lambda'}(s_0)
		= g_\lambda(s_0) \,,
\end{equation*}
which results in
\begin{equation*}
\max_{\lambda' \in [0, 1]} \min_{s \in [0, 1]} g_{\lambda'}(s) \le g_\lambda(s_0) \,, \qquad \forall \lambda \in [0, 1] \,.
\end{equation*}
Now we to show that the equality is reached when $\lambda = \lambda_0$ introduced in the statement. Noticing that $(1 - \lambda_0) \frac{1}{p} = \lambda_0 \frac{1}{q} = \frac{1}{p + q}$, we obtain using classical trigonometry formulae:
\begin{align*}
g_{\lambda_0}'(s)
	&= \dfrac{\pi}{2} d_0 \Biggl(2 \dfrac{1}{p} (1 - \lambda_0) \cos\left(\dfrac{\pi/2}{p} s d_0\right) \sin\left(\dfrac{\pi/2}{p} s d_0\right) \\
		&\qquad\qquad\qquad\qquad - 2 \dfrac{1}{q} \lambda_0 \cos\left(\dfrac{\pi/2}{q} (1 - s) d_0\right) \sin\left(\dfrac{\pi/2}{q} (1 - s) d_0\right)\Biggr) \\
    &= \dfrac{\pi/2}{p + q} d_0 \left( \sin\left(\frac{\pi}{p} s d_0\right) - \sin\left(\frac{\pi}{q} (1 - s) d_0\right) \right) \\
    &= \dfrac{\pi}{p + q} d_0 \cos\left(\dfrac{\pi}{2} d_0 \left( \dfrac{s}{p} + \dfrac{1 - s}{q} \right)\right) \sin\left(\dfrac{\pi}{2} d_0 \left( \dfrac{s}{p} - \dfrac{1 - s}{q} \right)\right) \\
    &= \dfrac{\pi}{p + q} d_0 \cos\left(\dfrac{\pi}{2} d_0 \left( \dfrac{s}{p} + \dfrac{1 - s}{q} \right)\right) \sin\left(\dfrac{\pi}{2} d_0 \left( \dfrac{1}{p} + \dfrac{1}{q} \right) (s - s_0)\right) \,.
\end{align*}
We observe that the cosine is always non-negative for any $s \in [0, 1]$, because $d_0 \le 1$. As for the sine, it is non-positive for $s \le s_0$ and non-negative for $s \ge s_0$. We deduce that $g_{\lambda_0}'(s) \le 0$ on $[0, s_0]$ and $g_{\lambda_0}'(s) \ge 0$ on $[s_0, 1]$. Therefore, the minimum of $g_{\lambda_0}$ is attained at $s_0$, which concludes the proof of the lemma.
\end{proof}

Regarding the denominator in the definition of $\Lambda_{\theta_1, \theta_2}^2$, observing that $\lambda_0$ and $s_0$ in the above lemma do not dependent on $d_0$ or $d_{\pi/2}$, we find
\begin{multline*}
\min_{\lambda \in [0, 1]} \left(1 - \min_{s \in [0, 1]} \left( (1 - \lambda) \sin^2\left(\frac{\pi/2}{q} s d_{\pi/2}\right) + \lambda \sin^2\left(\frac{\pi/2}{p} (1 - s) d_{\pi/2}\right) \right) \right) \\
    = 1 - \max_{\lambda \in [0, 1]} \min_{s \in [0, 1]} g_{\lambda, d_{\pi/2}}(s)
    = 1 - g_{\lambda_0, d_{\pi/2}}(s_0)
    = \cos^2\left(\dfrac{\pi/2}{p + q} d_{\pi/2}\right) \,.
\end{multline*}
This implies that $\lambda_0$ maximizes the minimum of the numerator and minimizes the maximum of the denominator at once. Moreover, when $\lambda = \lambda_0$, the minimum of the denominator and the maximum of the numerator are reached at a common value $s_0$. Therefore
\begin{equation*}
\Lambda_{\theta_1, \theta_2}^2
    = \dfrac{\sin^2\left(\frac{\pi/2}{p + q} d_0\right)}{\cos^2\left(\frac{\pi/2}{p + q} d_{\pi/2}\right)} \,.
\end{equation*}
When $p$ and $q$ have the same parity, it holds $d_{\pi/2} = d_0$, so that
\begin{equation} \label{eq:sameparity}
\Lambda_{\theta_1, \theta_2}
    = \tan\left(\dfrac{\pi/2}{p + q} d_0\right) \,.
\end{equation}
When they do not have the same parity, then $d_{\pi/2} = 1 - d_0$ and we obtain
\begin{align} \label{eq:notsameparity}
\Lambda_{\theta_1, \theta_2}
    &= \dfrac{\sin\left(\frac{\pi/2}{p + q} d_0 \right)}{\cos\left(\frac{\pi/2}{p + q} (1 - d_0)\right)}
    = \sin\left(\frac{\pi/2}{p + q}\right) - \cos\left(\frac{\pi/2}{p + q}\right) \dfrac{\sin\left(\frac{\pi/2}{p + q} (1 - d_0) \right)}{\cos\left(\frac{\pi/2}{p + q} (1 - d_0)\right)} \nonumber\\
    &= \sin\left(\frac{\pi/2}{p + q}\right) - \cos\left(\frac{\pi/2}{p + q}\right) \tan\left(\frac{\pi/2}{p + q} (1 - d_0)\right) \,.
\end{align}
Recall that in the above formulae, the dependence on the phase shifts $\theta_1$ and $\theta_2$ is hidden in $d_0$. Thus it remains to optimize over these parameters $\theta_1, \theta_2$ to compute the quantity $\Lambda_0$ defined in~\eqref{eq:defLambda_0}. In the first case~\eqref{eq:sameparity}, we notice that $d_0 \le 1$, and that the equality is achieved for $\theta_1 = \frac{\pi}{2p}$ and $\theta_2 = 0$ for instance, so that
\begin{equation*}
\Lambda_0
    = \tan\left(\dfrac{\pi/2}{p + q}\right) \,.
\end{equation*}
In the second case~\eqref{eq:notsameparity}, the maximum is reached for $d_0 = 1$ as well, so that
\begin{equation*}
\Lambda_0
    = \sin\left(\dfrac{\pi/2}{p + q}\right) \,.
\end{equation*}
The conclusion is that $\Lambda_0 = \Lambda(\frac{p}{q})$, where the function $\Lambda$ is the one defined in~\eqref{eq:deffunctionLambda}.

\medskip
\paragraph*{Step 3 \--- Construction of an equivalent shrunk observation set.}
%%%%%%%%%%%%%%%%%%%%%%%%%%%%%%%%%%%%%%%%%%%%%%%%%%%%%%%%%%%%%%%%%%
Recall that the sufficient condition of Theorem~\ref{thm:main} implies observability from an ``enlarged" observation set. This leads us to construct a shrunk set $\tilde I \subset I$, such that $\tilde I_R = \tilde I + (-R, R)$ is contained in $I$ up to a bounded set, so that the same is true for the sets $\omega(\tilde I)$ and $\omega(I)$. In the lemma below, when $I \subset \R_+$, we use the notation $I_R := \bigcup_{s \in I} (s - R, s + R)$.

\begin{lemm}[Shrunk observation set] \label{lem:shrunkobsset}
Let $I = \bigcup_n I_n$ where $I_n \subset \R_+$ are open intervals, with $\abs{I_n} \to + \infty$ if the union is infinite. Then there exists a family of disjoint open intervals $(\tilde J_n)_n$ in $\R_+$ (with $\abs{\tilde J_n} \to + \infty$ if there are infinitely many of them) such that the set $\tilde I = \bigcup_n \tilde J_n$ satisfies the following:
\begin{enumerate} [label=(\roman*)]
\item\label{it:shrunki} $\tilde I \subset I$;
\item\label{it:shrunkii} for any $R > 0$, the set $\tilde I_R \setminus I$ is bounded;
\item\label{it:shrunkiii} for any $R > 0$, it holds $\kappa_\star(\tilde I) = \kappa_\star(\tilde I_R) = \kappa_\star(I) = \kappa_\star(I_R)$.
\end{enumerate}
\end{lemm}

\begin{proof}
Recall the definition of $\kappa_\star$ in~\eqref{eq:defkappastar}. We write the open set $I$ as a union of disjoint open intervals $I = \bigcup_n J_n$. Let us fix $R > 0$. We first deal with the case where there are only finitely many $J_n$'s. If $I$ is bounded, one has $\kappa_\star(I) = \kappa_\star(I_R) = 0$ and $\tilde I = \varnothing$ satisfies the conclusions of the lemma. If $I$ is not bounded, then there is an index $n_0$ for which $J_{n_0}$ is of the form $J_{n_0} = (a, + \infty)$. Then one has for any $R > 0$ the equality $\kappa_\star(I) = \kappa_\star(I_R) = 1$ and $\tilde I = J_{n_0}$ satisfies the conclusions of the lemma as well.

We now consider the case where there are infinitely many $J_n$'s. By assumption, one has $\abs{J_n} \to + \infty$ as $n \to \infty$. Writing $J_n = (a_n, b_n)$, with $a_n < b_n < \infty$, we define for any index $n$ the interval
\begin{equation*}
\tilde J_n = \left(a_n + \dfrac{\sqrt{4 + \delta_n}}{2}, b_n - \dfrac{\sqrt{4 + \delta_n}}{2}\right) \,,
	\qquad \rm{where} \;\, \delta_n = \min(a_n, b_n - a_n) \,.
\end{equation*}
Since the $J_n$'s are disjoint and $\abs{J_n} = b_n - a_n \to + \infty$, we also have $a_n \to + \infty$, so that $\delta_n \to + \infty$ too. Incidentally, one readily checks that $\abs{\tilde J_n} \to + \infty$ as well. Thus, defining
\begin{equation*}
\tilde I = \bigcup_n \tilde J_n \,,
\end{equation*}
we have $\tilde I \subset I$, namely the property~\ref{it:shrunki}, and given any $R > 0$, there are finitely many $n$'s such that $R \ge \frac{\sqrt{\delta_n}}{2}$. It implies that the thickened set $\tilde I_R$ is contained in $I$ modulo a bounded set, hence~\ref{it:shrunkii}. The crucial point of this construction is the claim~\ref{it:shrunkiii}. As a consequence of the inclusions $\tilde I \subset \tilde I_R$ and $I \subset I_R$, it holds $\kappa_\star(\tilde I) \le \kappa_\star(\tilde I_R)$ and $\kappa_\star(I) \le \kappa_\star(I_R)$. Moreover, in virtue of~\ref{it:shrunkii}, we can write $\tilde I_R = (\tilde I_R \cap I) \cup A$, where $A = \tilde I_R \setminus I$ is bounded. Since $\frac{1}{r} \abs{A \cap [0, r]} \le \frac{1}{r} \abs{A} \to 0$ as $r \to + \infty$, one can check that $\kappa_\star(\tilde I_R) \le \kappa_\star(\tilde I_R \cap I) \le \kappa_\star(I)$. To sum up, we have proved so far that $\kappa_\star(\tilde I) \le \kappa_\star(\tilde I_R) \le \kappa_\star(I) \le \kappa_\star(I_R)$.

Thus, in order to prove~\ref{it:shrunkiii}, it remains to check that $\kappa_\star(I_R) \le \kappa_\star(\tilde I)$. Unless we are in the straightforward case $\kappa_\star(I_R) = 0$, we pick $\kappa \in (0, \kappa_\star(I_R))$, so that by definition of $\kappa_\star$, it holds
\begin{equation} \label{eq:averagekappa}
\exists c > 0, \exists r_0 > 0 : \forall r \ge r_0 \,, \qquad
	\dfrac{1}{r} \abs*{I_R \cap [\kappa r, r]} \ge c \,.
\end{equation}
In the sequel, to simplify notation, we write $J_n^R = (J_n)_R$. Up to enlarging $r_0$, we can assume that for any index $n$ such that $J_n^R \cap [\kappa r_0, + \infty) \neq \varnothing$, it holds $\delta_n \ge 5 + 8 R$ (recall that $\delta_n \to + \infty$). Fix an $r \ge r_0$. Then there is a finite (possibly empty) set of indices $\{n_k\}_k$ such that $J_{n_k}^R \subset [\kappa r, r]$. Assume first that
\begin{equation} \label{eq:boundfrombelowJnk}
\dfrac{1}{r} \abs*{\bigcup_k J_{n_k}^R \cap [\kappa r, r]}
	= \dfrac{1}{r} \sum_k \abs{J_{n_k}^R}
	\ge \dfrac{c}{2} \,.
\end{equation}
Then it holds
\begin{align*}
\dfrac{1}{r} \abs*{\tilde I \cap [\kappa r, r]}
	&\ge \dfrac{1}{r} \sum_k \abs{\tilde J_{n_k}}
	= \dfrac{1}{r} \sum_k \left(\abs{J_{n_k}^R} - \left( \sqrt{4 + \delta_{n_k}} + 2 R \right)\right) \\
	&\ge \dfrac{1}{r} \sum_k \left(1 - \dfrac{\sqrt{4 + \delta_{n_k}} + 2 R}{\delta_{n_k} + 2R}\right) \abs{J_{n_k}^R}
	\ge \dfrac{1}{r} \sum_k \left(1 - \sqrt{\dfrac{4}{\delta_n^2} + \dfrac{1}{\delta_{n_k}}} - \dfrac{2 R}{\delta_n + 2R}\right) \abs{J_{n_k}^R} \,.
\end{align*}
To obtain the second to last inequality, we used the fact that  by definition of $\delta_n$, it holds $\abs{J_n} \ge \delta_n$, which implies in particular that $\abs{J_n^R} \ge \delta_n + 2 R$. Using in the last line that $\delta_{n_k} \ge 5 + 8 R$, together with~\eqref{eq:boundfrombelowJnk}, we obtain
\begin{equation} \label{eq:lowerboundcaseinterior}
\dfrac{1}{r} \abs*{\tilde I \cap [\kappa r, r]}
	\ge \left(1 - \sqrt{\dfrac{9}{25}} - \dfrac{1}{5}\right) \dfrac{1}{r} \sum_k \abs*{J_{n_k}^R}
	\ge \dfrac{1}{5} \times \dfrac{c}{2} \,.
\end{equation}
Otherwise, if now~\eqref{eq:boundfrombelowJnk} is not satisfied, then recalling~\eqref{eq:averagekappa}, it holds
\begin{equation*}
\dfrac{1}{r} \abs*{\left(I_R \setminus \bigcup_k J_{n_k}^R\right) \cap [\kappa r, r]}
	\ge \dfrac{c}{2} \,.
\end{equation*}
Any interval $J_n^R \subset I_R \setminus \bigcup_k J_{n_k}^R$ intersecting $[\kappa r, r]$ must contain $\kappa r$ or $r$, otherwise it would satisfy $J_n^R \cap [\kappa r, r] = \varnothing$, or $J_n^R \subset (\kappa r, r)$ (the latter would imply that $n \in \{n_k\}_k$). Therefore, there are at most two such intervals. We deduce that there is an index $n_\star$ such that $J_{n_\star}^R \not\subset [\kappa r, r]$ but $J_{n_\star}^R \cap [\kappa r, r] \neq \varnothing$, with
\begin{equation*}
\dfrac{1}{r} \abs*{J_{n_\star}^R \cap [\kappa r, r]}
	\ge \dfrac{c}{4} \,.
\end{equation*}
Writing $J_{n_\star}^R = (a_{n_\star} - R, b_{n_\star} + R)$, the fact that $J_{n_\star}^R \cap [\kappa r, r] \neq \varnothing$ imposes that $a_{n_\star} - R \le r$, hence $a_{n_\star} \le r + R$. Thus we obtain
\begin{align} \label{eq:lowerboundcaseboundary}
\dfrac{1}{r} \abs*{\tilde I \cap [\kappa r, r]}
	&\ge \dfrac{1}{r} \abs*{\tilde J_{n_\star} \cap [\kappa r, r]}
	\ge \dfrac{1}{r} \left( \abs{J_{n_\star}^R \cap [\kappa r, r]} - \sqrt{4 + \delta_{n_\star}} - 2 R \right) \nonumber\\
	&\ge \dfrac{c}{4} - \dfrac{\sqrt{4 + r + R} + 2 R}{r} \,.
\end{align}
We used the fact that $\delta_{n_\star} \le a_{n_\star} \le r + R$ to obtain the last inequality. In view of the estimates~\eqref{eq:lowerboundcaseinterior} and~\eqref{eq:lowerboundcaseboundary}, in any case we have
\begin{equation*}
\dfrac{1}{r} \abs*{\tilde I \cap [\kappa r, r]}
	\ge \min\left( \dfrac{c}{10}, \dfrac{c}{4} - \dfrac{\sqrt{4 + r + R} + 2 R}{r} \right) \,.
\end{equation*}
We conclude that
\begin{equation*}
\liminf_{r \to + \infty} \dfrac{1}{r} \abs*{\tilde I \cap [\kappa r, r]}
	> 0 \,.
\end{equation*}
Recalling that $\kappa$ is any arbitrary number $< \kappa_\star(I_R)$, we finally get the desired converse inequality $\kappa_\star(\tilde I) \ge \kappa_\star(I_R)$. Thus~\ref{it:shrunkiii} is proved, which concludes the proof of the lemma.
\end{proof}

In the sequel, we will proceed as follows: to prove that $\kappa_\star(I) > \Lambda(\frac{\nu_2}{\nu_1})$ is a sufficient condition to have observability from $\omega(I)$, we will check that the dynamical condition~\eqref{eq:dynamicalcondition} of Theorem~\ref{thm:main} is true in the smaller set $\omega(\tilde I)$, where $\tilde I$ is given by Lemma~\ref{lem:shrunkobsset}. To show that it is also necessary, we will check that the condition~\eqref{eq:dynamicalcondition} is violated in the larger set $\omega(I)_R = \omega(I_R)$ for any $R > 0$.

\medskip
\paragraph*{Step 4 \--- Geometric condition of observability for rationally dependent characteristic frequencies.}
%%%%%%%%%%%%%%%%%%%%%%%%%%%%%%%%%%%%%%%%%%%%%%%%%%%%%%%%%%%%%%%%%%%%%
We investigate the validity of the dynamical condition~\eqref{eq:dynamicalcondition} of Theorem~\ref{thm:main}. In the case where $\frac{\nu_2}{\nu_1} \in \Q$, writing $\frac{\nu_2}{\nu_1} = \frac{p}{q}$ as an irreducible fraction, the period of the Hamiltonian flow is given by $T_0 = \frac{2 \pi}{\nu_2} p = \frac{2 \pi}{\nu_1} q$. We write for short $\Lambda = \Lambda(\frac{\nu_2}{\nu_1})$ and $\kappa_\star = \kappa_\star(I)$. Our goal now is to reformulate the dynamical condition~\eqref{eq:dynamicalcondition} using the Area formula.

\begin{prop}[Area formula {\cite[Theorem 3.9]{EvansGariepy:Book}}] \label{prop:areaformula}
Let $J \subset \R$ be a bounded interval and let $\gamma : J \to \R^n$ be a Lipschitz curve. Then $\gamma$ is differentiable at Lebesgue-almost every point in $J$ and for any Borel set $E \subset \R^n$, it holds
\begin{equation*}
\int_J \one_E\left(\gamma(t)\right) \abs*{\gamma'(t)} \dd t
	= \int_{\im \gamma \cap E} \# \gamma^{-1}\left(\{x\}\right) \dd \haus^1(x) \,.
\end{equation*}
Here, $\im \gamma = \set{\gamma(t)}{t \in J} \subset \R^n$, $\# \gamma^{-1}(\{x\})$ stands for the cardinality of the set $\set{t \in J}{\gamma(t) = x}$, and $\haus^1$ is the one-dimensional Hausdorff measure.
\end{prop}

We will apply this formula to a curve of the form $\gamma : t \mapsto \abs{x^t} \in \R$ defined on $J = (0, T)$, where $t \mapsto (x^t, \xi^t)$ is a trajectory of the Hamiltonian flow. Calculations will involve the inverse Jacobian $\abs{\gamma'(t)}^{-1}$. Using anisotropy\footnote{In the excluded isotropic case ($p = q = 1$), one can choose $(x^0, \xi^0)$ so that $\abs{x^t}$ is constant, as we did in the proof of Proposition~\ref{prop:sphericalsetwithaxiallysymmetricpotential} (see Subsection~\ref{subsec:proofaxiallysymmetricpotential}). In such a situation, the set $\im \gamma \subset \R_+$ is reduced to a point. This is a very singular situation, since the Jacobian $\abs{\gamma'(t)}$ is identically zero.} of the harmonic oscillator ($p \neq q$), we can check that the Jacobian vanishes only at a finite number of points.

\begin{lemm} \label{lem:criticalpointsisolated}
Let $t \mapsto (x^t, \xi^t)$ be a trajectory of the Hamiltonian flow of an anisotropic harmonic oscillator, with initial datum $\rho_0 = (x_0, \xi_0)$. Then the curve $\gamma : \R \ni t \mapsto \abs{x^t} \in \R_+$ is Lipschitz with constant $\sqrt{2 p(\rho_0)}$. If $\rho_0 \neq 0$, then $\gamma$ is of class $\cont^\infty$ in $\R \setminus \{\gamma = 0\}$. Moreover, the set
\begin{equation*}
S_\gamma
	:= \set{t \in \R}{\gamma(t) = 0 \quad \rm{or} \quad \gamma'(t) = 0}
\end{equation*}
is locally finite, namely for any bounded interval $I \subset \R$, the set $S_\gamma \cap I$ is finite. In addition, for any bounded interval $I \subset \R$, one has
\begin{equation} \label{eq:cardinalitypreimage}
\exists k = k(I) \in \N : \forall s \in \R_+ \,, \qquad
	\# \gamma^{-1}\left(\{s\}\right) \cap I \le k \,.
\end{equation}
\end{lemm}

\begin{proof}
That $\gamma$ is Lipschitz follows from the inverse triangle inequality, the Hamilton equations~\eqref{eq:defHamiltonianflowcoord} and the fact that $p(x, \xi) = V(x) + \frac{1}{2} \abs{\xi}^2$ is preserved by the flow:
\begin{equation*}
\abs*{\gamma(t_2) - \gamma(t_1)}
	\le \abs{x^{t_2} - x^{t_1}}
	\le \abs{t_2 - t_1} \sup_{t \in \R} \abs{\xi^t}
	\le \abs{t_2 - t_1} \sqrt{2 p(\rho_0)} \,.
\end{equation*}
From now on, we assume that $\rho_0 \neq 0$. First notice that the set $\{\gamma = 0\}$ is closed since $\gamma$ is continuous. Given that $t \mapsto x^t$ is smooth, the curve $\gamma$ is smooth in a neighborhood of any point $t \in \R \setminus \{\gamma = 0\}$, so that $\gamma \in \cont^\infty(\R \setminus \{\gamma = 0\})$. To show that $S_\gamma$ is locally finite, it is sufficient to prove that it is closed and also discrete, namely that it is made of isolated points.\footnote{If $S \subset \R$ is closed and discrete, the for any compact interval $I \subset \R$, the set $S \cap I$ is compact. Since $S$ is discrete, the set $S \cap I$ can be covered by open sets containing at most one element of $S$. Then, extracting a finite subcovering shows that $S \cap I$ is finite.}

We first check that it is closed by observing that the map $f : t \mapsto \gamma^2(t) = \abs{x^t}^2$ belongs to $\cont^\infty(\R)$ and that
\begin{equation} \label{eq:Sgammaequality}
S_\gamma
	= \set{t \in \R}{f'(t) = 0} \,.
\end{equation}
To check this equality, we use the fact that $f'(t) = 2 \gamma(t) \gamma'(t)$ for all $t \in \R \setminus S_\gamma$. If $t \not\in S_\gamma$, then it follows that $\gamma(t) \gamma'(t) \neq 0$. Conversely, if $t \in S_\gamma$, either $\gamma(t) \neq 0$, so that $\gamma'(t) = 0$, in which case it holds $f'(t) = 2 \gamma(t) \gamma'(t) = 0$; or $\gamma(t) = 0$, which implies that $x^t = 0$, hence $f'(t) = 2 x^t \cdot \xi^t = 0$. This justifies~\eqref{eq:Sgammaequality}.

Thus it remains to show that $S_\gamma$ is discrete. Let us compute the derivatives of $f$ up to order $4$:
\begin{align}
f'(t)
    &= 2 x^t \cdot \xi^t \,, \label{eq:derivative1}\\
f^{(2)}(t)
    &= 2 \abs*{\xi^t}^2 - 2 x^t \cdot A x^t \,, \label{eq:derivative2}\\
f^{(3)}(t)
    &= - 4 \xi^t \cdot A x^t - 2 \xi^t \cdot A x^t - 2 x^t \cdot A \xi^t
    = - 8 \xi^t \cdot A x^t \,, \label{eq:derivative3}\\
f^{(4)}(t)
    &= 8 \left(\abs*{A x^t}^2 - \xi^t \cdot A \xi^t\right) \,.\label{eq:derivative4}
\end{align}
Let us write the Taylor expansion of $f'$ at order $3$ near $t_0 \in \R$:
\begin{equation} \label{eq:Taylorexp}
f'(t)
	= f'(t_0) + (t - t_0) f^{(2)}(t_0) + \dfrac{(t - t_0)^2}{2} f^{(3)}(t_0)
		+ \dfrac{(t - t_0)^3}{6} f^{(4)}(t_0) + o\left((t - t_0)^3\right) \,.
\end{equation}
Suppose that $t_0 \in S_\gamma$. Then $f'(t_0) = 0$ in virtue of~\eqref{eq:Sgammaequality}. If $f^{(2)}(t_0) \neq 0$, then~\eqref{eq:Taylorexp} yields
\begin{equation*}
\abs*{f'(t)} \ge \dfrac{\abs{f^{(2)}(t_0)}}{2} \abs*{t - t_0}
\end{equation*}
for all $t$ in a neighborhood $U$ of $t_0$. In particular, $S_\gamma \cap U = \{t_0\}$, meaning that $t_0$ is isolated. Likewise, if $f^{(2)}(t_0) = 0$ but $f^{(3)}(t_0) \neq 0$, then~\eqref{eq:Taylorexp} leads to
\begin{equation*}
\abs*{f'(t)} \ge \dfrac{\abs{f^{(3)}(t_0)}}{4} \abs*{t - t_0}^2
\end{equation*}
in a neighborhood of $t_0$, so that $t_0$ is isolated again.

Now, if $f^{(2)}(t_0) = f^{(3)}(t_0) = 0$, we show that necessarily $f^{(4)}(t_0) \neq 0$. In view of~\eqref{eq:derivative1}, \eqref{eq:derivative2}, and~\eqref{eq:derivative3}, it holds
\begin{equation} \label{eq:1-2-3derivativevanishes}
x^{t_0} \cdot \xi^{t_0} = 0
	\qquad
\abs{\xi^{t_0}}^2 = x^{t_0} \cdot A x^{t_0} \,,
	\quad \rm{and} \quad
\xi^{t_0} \cdot A x^{t_0} = 0 \,.
\end{equation}
The first and third equalities mean that $\xi^{t_0} \perp x^{t_0}$ and $\xi^{t_0} \perp A x^{t_0}$. Moreover, the second equality ensures that $x^{t_0} \neq 0$ and $\xi^{t_0} \neq 0$, otherwise $(x^{t_0}, \xi^{t_0}) = (0, 0)$, hence $\rho_0 = 0$. Since we are in two dimensions, we deduce that $A x^{t_0}$ and $x^{t_0}$ are parallel, and therefore $x^{t_0}$ is an eigenvector of $A$. Since $\xi^{t_0} \perp x^{t_0}$ and $\xi^{t_0} \neq 0$, we deduce that $\xi^{t_0}$ is also an eigenvector, associated with a different eigenvalue since $A$ has two distinct eigenvalues by assumption. We relabel $\nu_1$ and $\nu_2$ so that 
$A x^{t_0} = \nu_x^2 x^{t_0}$ and $A \xi^{t_0} = \nu_\xi^2 \xi^{t_0}$. Plugging this into the second equality in~\eqref{eq:1-2-3derivativevanishes} yields $\abs{\xi^{t_0}}^2 = \nu_x^2 \abs{x^{t_0}}^2$, from which we deduce that the fourth derivative~\eqref{eq:derivative4} cannot vanish at $t_0$, given that the oscillator is anisotropic ($\nu_x \neq \nu_\xi$):
\begin{equation*}
\abs*{A x^{t_0}}^2 - \xi^{t_0} \cdot A \xi^{t_0}
    = \nu_x^4 \abs*{x^{t_0}}^2 - \nu_\xi^2 \abs*{\xi^{t_0}}^2
    = \nu_x^2 (\nu_x^2 - \nu_\xi^2) \abs{x^{t_0}}^2
    \neq 0 \,.
\end{equation*}
Therefore~\eqref{eq:Taylorexp} implies that
\begin{equation*}
\abs*{f'(t)} \ge \dfrac{\abs{f^{(4)}(t_0)}}{12} \abs*{t - t_0}^3
\end{equation*}
in a neighborhood of $t_0$, that is to say the critical point $t_0$ is again isolated. To sum up, the above argument shows that there exists a neighborhood $U$ of $t_0$ such that $U \cap S_\gamma = \{t_0\}$, so $S_\gamma$ is indeed a discrete set.

Now fix $I \subset \R$ a bounded interval. We have just shown that $n = \# (S_\gamma \cap I)$ is finite. To prove~\eqref{eq:cardinalitypreimage}, we observe that the complement of $S_\gamma$ in $I$ is a union of at most $n + 1$ open intervals in $I$, on which $\gamma'$ does not vanish and has constant sign (use the intermediate value theorem). Therefore $\gamma$ is one-to-one in each of these intervals. We infer that
\begin{equation*}
\forall s \in \R \,, \qquad
	\# \set{t \in I}{\gamma(t) = s}
		\le n + 1 + \# (S_\gamma \cap I)
		= 2 n + 1 \,.
\end{equation*}
This completes the proof of the lemma.
\end{proof}

Let us assume that $\kappa_\star \le \Lambda$ and fix $R > 0$. Recalling that $\kappa_\star = \kappa_\star(I ) = \kappa_\star(I_R)$ from~\ref{it:shrunkiii} in Lemma~\ref{lem:shrunkobsset}, we know that there exists a sequence $(r_n)_{n \in \N}$ tending to $+ \infty$ along which
\begin{equation} \label{eq:avgtendtozerokappastar}
\dfrac{1}{r_n} \abs*{I_R \cap [\kappa_\star r_n, r_n]}
	\strongto{n \to \infty} 0 \,.
\end{equation}

According to Step 2, considering actions $(A_1, A_2) = (\sqrt{1 - \lambda_0}, \sqrt{\lambda_0}) = (\sqrt{\frac{p}{p + q}}, \sqrt{\frac{q}{p + q}})$ and initial angles $(\theta_1, \theta_2) = (\frac{\pi}{2p}, 0)$, one obtains a trajectory of the Hamiltonian flow $t \mapsto (x^t, \xi^t)$ such that
\begin{equation} \label{eq:criticaltrajectory}
\min_{t \in [0, T]} \abs{x^t} = \Lambda \max_{t \in [0, T]} \abs{x^t} \,,
\end{equation}
that is to say a trajectory that attains the supremum~\eqref{eq:defLambda_0}. Here $T$ is any real number larger than the period of the flow $T_0$. In view of the homogeneity of degree $1$ of the Hamiltonian flow, we know that $t \mapsto (c x^t, c \xi^t)$ is still a trajectory of the Hamiltonian flow, for any scalar $c \in \R$. Note that~\eqref{eq:criticaltrajectory} above ensures that $\abs{x^t}$ is bounded from below by a positive constant for all times. Therefore, Lemma~\ref{lem:criticalpointsisolated} implies that the curve $\gamma : (0, T) \ni t \mapsto \abs{x^t}$ is smooth. The corresponding set $S_\gamma$ of Lemma~\ref{lem:criticalpointsisolated} is nothing but $S_\gamma = \{\gamma' \neq 0\}$. A consequence of this lemma is that $S_\gamma$ has vanishing measure. Thus we write
\begin{equation} \label{eq:monotoneconvergenceBN}
(0, T) \setminus S_\gamma
	= \bigcup_{N \in \N} B_N
	\qquad \rm{where} \quad B_N = \set{t \in (0, T)}{\abs*{\gamma'(t)} \ge 2^{-N}} \,.
\end{equation}
Fix an arbitrary $N \in \N$ and a scalar $c > 0$. Then we obtain
\begin{align} \label{eq:areaformula1}
\int_{B_N} \one_{I_R}\left(c \abs*{x^t}\right) \dd t
    &\le \dfrac{2^N}{c} \int_{B_N} \one_{I_R}\left(c \gamma(t)\right) \abs*{c \gamma'(t)} \dd t
    = \dfrac{2^N}{c} \int_{c \gamma(B_N)} \one_{I_R}(s) \# \set{t \in (0, T)}{\gamma(t) = \dfrac{s}{c}} \dd s \nonumber\\
    &\le \dfrac{2^N k}{c} \int_{c \gamma(B_N)} \one_{I_R}(s) \dd s
    \le \dfrac{2^N k}{c} \int_{c \Lambda \max \gamma}^{c \max \gamma} \one_{I_R}(s) \dd s \,,
\end{align}
where the equality results from the Area formula (Proposition~\ref{prop:areaformula}) applied to $E = I_R$. The integer $k$ is the one from Lemma~\ref{lem:criticalpointsisolated}~\eqref{eq:cardinalitypreimage}. The last inequality follows from the fact that $\abs{x^t}$ spans the interval $[\Lambda \max \gamma, \max \gamma]$ by construction (recall~\eqref{eq:criticaltrajectory}). Thus taking $c = c_n = r_n/ \max \gamma$, with $(r_n)_{n \in \N}$ the sequence from~\eqref{eq:avgtendtozerokappastar}, we obtain
\begin{equation*}
\int_{B_N} \one_{I_R}\left(c_n \gamma(t)\right) \dd t
    \le 2^N k (\max \gamma) \times \dfrac{1}{r_n} \abs*{I_R \cap [\Lambda r_n, r_n]}
    \strongto{n \to \infty} 0 \,,
\end{equation*}
by~\eqref{eq:avgtendtozerokappastar}, since $\Lambda \ge \kappa_\star$. Now going back to~\eqref{eq:monotoneconvergenceBN}, since the set $S_\gamma$ is negligible, monotone convergence ensures that $\abs{B_N} \to T$ as $N \to \infty$. We finally obtain that
\begin{equation*}
\int_0^T \one_{\omega(I_R)}\left(c_n x^t\right) \dd t
    \le \abs*{(0, T) \setminus B_N} + \int_{B_N} \one_{I_R}\left(c_n \abs*{x^t}\right) \dd t
    = T - \abs*{B_N} + o(1)
\end{equation*}
as $n \to \infty$. We let $N \to \infty$ to conclude that the dynamical condition~\eqref{eq:dynamicalcondition} is not fulfilled, namely
\begin{equation*}
\liminf_{\rho \to \infty} \int_0^T \one_{\omega(I)_R \times \R^d}\left(\phi^t(\rho)\right) \dd t
    = 0 \,.
\end{equation*}
The parameter $R > 0$ is arbitrary. Therefore the necessary condition of Theorem~\ref{thm:main} tells us that observability from $\omega(I)$ in time $T$ does not hold, and $T \ge T_0$ itself is arbitrary.

\medskip

We turn to the case where $\kappa_\star > \Lambda$. This time, we take $T = T_0$ to be the period of the Hamiltonian flow and check that the observability condition~\eqref{eq:dynamicalcondition} holds in $\omega(\tilde I)$. We pick $\kappa \in (\Lambda, \kappa_\star)$. In virtue of Lemma~\ref{lem:shrunkobsset}~\ref{it:shrunkiii}, we have $\kappa_\star = \kappa_\star(I) = \kappa_\star(\tilde I)$ so that
\begin{equation} \label{eq:densityboundedfrombelow}
\exists c > 0, \exists r_0 > 0 : \forall r \ge r_0 \,, \qquad
	\dfrac{1}{r} \abs*{\tilde I \cap [\kappa r, r]} \ge c \,.
\end{equation}
Let $(x^t, \xi^t)$ be a trajectory of the Hamiltonian flow with initial datum $\rho_0$. One can estimate $\tilde r = \max \abs{x^t}$ from below as follows: since the time $t_0$ at which the maximum is reached is also a (local) maximum of $\abs{x^t}^2$, the second derivative satisfies:
\begin{equation*}
\dfrac{\dd^2}{\dd t^2} {\abs{x^t}^2}_{\vert t = t_0}
	= 2 \abs{\xi^{t_0}}^2 - 2 x^{t_0} \cdot A x^{t_0}
	\le 0 \,.
\end{equation*}
Thus it holds
\begin{equation} \label{eq:estimatetilderbyp}
\tilde r^2
    := \abs*{x^{t_0}}^2
	\ge x^{t_0} \cdot \dfrac{A}{\norm{A}} x^{t_0}
	\ge \dfrac{1}{\norm{A}} \left( \dfrac{1}{2} x^{t_0} \cdot A x^{t_0} + \dfrac{1}{2} \abs{\xi^{t_0}}^2 \right)
	= \dfrac{1}{\norm{A}} p(\rho_0) \,.
\end{equation}
Provided $\abs{\rho_0}$ is large enough so that $p(\rho_0) \ge \norm{A} r_0^2$, we see in particular that $\tilde r \ge r_0$. Introduce $\gamma : (0, T) \ni t \mapsto \abs{x^t}$. We know from Lemma~\ref{lem:criticalpointsisolated} that $\gamma$ is Lipschitz with constant $\sqrt{2 p(\rho_0)} \le \sqrt{2 \norm{A}} \tilde r$ (this inequality is a consequence of~\eqref{eq:estimatetilderbyp} above). In particular, we have $\abs{\gamma'(t)} \le \sqrt{2 \norm{A}} \tilde r$ outside the set $S_\gamma$ from Lemma~\ref{lem:criticalpointsisolated}. Thus we can apply again the Area formula (Proposition~\ref{prop:areaformula}):
\begin{align} \label{eq:lowerboundareaformula}
\int_0^T \one_{\omega(\tilde I)}(x^t) \dd t
	&= \int_0^T \one_{\tilde I}\left(\abs{x^t}\right) \dd t
	\ge (2 \norm{A})^{-1/2} \dfrac{1}{\tilde r} \int_0^T \one_{\tilde I}\left(\gamma(t)\right) \abs*{\gamma'(t)}\dd t \nonumber\\
	&= (2 \norm{A})^{-1/2} \dfrac{1}{\tilde r} \int_{\gamma((0, T))} \one_{\tilde I}(s) \# \set{t \in (0, T)}{\gamma(t) = s} \dd s \nonumber\\
	&\ge (2 \norm{A})^{-1/2} \dfrac{1}{\tilde r} \int_{\gamma((0, T))} \one_{\tilde I}(s) \dd s \,.
\end{align}
This time, one has $\gamma((0, T)) \supset [\Lambda \tilde r, \tilde r] \supset [\kappa \tilde r, \tilde r]$ (by definition of $\Lambda$, see~\eqref{eq:defLambda_0}). This means that
\begin{equation} \label{eq:checkdynassumptionT/2}
\int_0^T \one_{\omega(\tilde I)}(x^t) \dd t
	\ge (2 \norm{A})^{-1/2} \dfrac{1}{\tilde r} \int_{\kappa \tilde r}^{\tilde r} \one_{\tilde I}(s) \dd s
	\ge (2 \norm{A})^{-1/2} c \,,
\end{equation}
where the last inequality is due to~\eqref{eq:densityboundedfrombelow} (recall that $\tilde r \ge r_0$). Therefore the dynamical condition~\eqref{eq:dynamicalcondition} of Theorem~\ref{thm:main} is satisfied. In fact, the explicit expression of the Hamiltonian flow in action-angle coordinates~\eqref{eq:hamiltonianflowinactionanglecoordinatesomegaI} shows that $\abs{x^t}^2$ is $\frac{T_0}{2}$-periodic.\footnote{One can check that the projected trajectories of rational harmonic oscillators are invariant by point reflection with respect to the origin or axial symmetry with respect to some coordinate axis, depending on whether $p$ and $q$ have the same parity or not.} Therefore, setting $\tilde c := (2 \norm{A})^{-1/2} c$, the dynamical condition~\eqref{eq:dynamicalcondition} is equivalently satisfied in time $\frac{T_0}{2} - \frac{\tilde c}{4}$:
\begin{equation*}
\int_0^{\frac{T_0}{2} - \frac{\tilde c}{4}} \one_{\omega(\tilde I)}(x^t) \dd t
	\ge \dfrac{1}{2} \int_0^{T_0} \one_{\omega(\tilde I)}(x^t) \dd t - \dfrac{\tilde c}{4}
	\ge (2 \norm{A})^{-1/2} \dfrac{c}{4} \,.
\end{equation*}
By Theorem~\ref{thm:main}, this implies that observability holds from $\omega(\tilde I)_R \setminus K$ for some $R > 0$ and any compact set $K \subset \R^d$ in any time $> \frac{T_0}{2} - \frac{\tilde c}{4}$, which in turn implies observability from $\omega(I)$ in virtue of Lemma~\ref{lem:shrunkobsset}~\ref{it:shrunkii}. Incidentally, the optimal observation time is strictly smaller than $T_0/2$.

%\medskip
\paragraph*{Step 5 \--- Diophantine approximation in the irrational case.}
%%%%%%%%%%%%%%%%%%%%%%%%%%%%%%%%%%%%%%%%%%%%%%%%%%%%%%%%%%%%%%%%%%
We assume that $\frac{\nu_2}{\nu_1} \in \R \setminus \Q$ and denote by $\frac{p_j}{q_j}$ the reduced fraction expression of its convergents (see Remark~\ref{rmk:nonDiophantineirrationals}). We investigate the validity of the dynamical condition~\eqref{eq:dynamicalcondition} by approximating the trajectories of the ``irrational" Hamiltonian flow by the trajectories of the ``rational" Hamiltonian flow obtained by replacing $\frac{\nu_2}{\nu_1}$ with its convergent $\frac{p_j}{q_j}$. For instance, a projected trajectory of the irrational harmonic oscillator of the form
\begin{equation} \label{eq:initialtrajectory}
x_1^t = A_1 \sin\left(\nu_1 t + \theta_1\right) \,, \qquad x_2^t = A_2 \sin\left(\nu_2 t + \theta_2\right) \,,
\end{equation}
should be compared to
\begin{equation} \label{eq:approximatingtrajectory}
x_{j, 1}^t = A_1 \sin\left(\nu_1 t + \theta_1\right) \,, \qquad x_{j, 2}^t = A_2 \sin\left(\dfrac{p_j}{q_j} \nu_1 t + \theta_2\right) \,,
\end{equation}
which is a trajectory of the Hamiltonian flow of the (rational) harmonic oscillator with characteristic frequencies $\nu_1$ and $\frac{p_j}{q_j} \nu_1$, whose classical Hamiltonian is:
\begin{equation*}
p_j(x, \xi)
	= \dfrac{1}{2} \left( \nu_1^2 x_1^2 + \dfrac{p_j^2}{q_j^2} \nu_1^2 x_2^2\right) + \dfrac{1}{2} \left( \xi_1^2 + \xi_2^2 \right) \,.
\end{equation*}
The distance between these two trajectories is
\begin{equation} \label{eq:approximationirrationalflow}
\abs*{x^t - x_j^t}
    = \abs*{x_2^t - x_{j, 2}^t}
    \le A_2 \abs*{\nu_2 - \dfrac{p_j}{q_j} \nu_1} \abs*{t}
    \le A_2 \frac{\nu_1 \abs{t}}{q_j^2} \,,
\end{equation}
owing to the fact that the sine function is $1$-Lipschitz and to the Diophantine approximation result~\eqref{eq:diophantineapproximation}. We already know from Step 2 that
\begin{equation*}
\min_{t \in [0, T_j]} \abs{x_j^t} \le \Lambda_j \max_{t \in [0, T_j]} \abs{x_j^t} \,,
	\qquad \rm{where} \quad T_j = \frac{2 \pi}{\nu_1} q_j \,, \;\; \Lambda_j = \Lambda\left(\dfrac{p_j}{q_j}\right) \,.
\end{equation*}
The time $T_j$ is the period of the flow of the rational harmonic oscillator with characteristic frequencies $\nu_1$ and $\frac{p_j}{q_j} \nu_1$. Let us set
\begin{equation} \label{eq:defM_jm_j}
m_j = \min_{t \in \R} \abs{x_j^t}
	\qquad \rm{and} \qquad
M_j = \max_{t \in \R} \abs{x_j^t} \,.
\end{equation}
Although the trajectory $t \mapsto x_j^t$ is $T_j$-periodic, it will be convenient to compare $x_j^t$ and $x^t$ on smaller times. Then in view of~\eqref{eq:approximationirrationalflow}, on the time interval $[0, \eta T_j]$, where $\eta \in (0, 1]$, the norm $\abs{x^t}$ spans an interval $J_j^\eta$ such that
\begin{equation} \label{eq:imageofcurvextxjteta}
J_j^\eta
    \subset \left[m_j - A_2 \eta \dfrac{2 \pi}{q_j}, M_j + A_2 \eta \dfrac{2 \pi}{q_j}\right] \,,
\end{equation}
and if $\eta = 1$, since $\abs{x_j^t}$ attains $m_j$ and $M_j$ on the time interval $[0, T_j]$, it holds
\begin{equation} \label{eq:imageofcurvextxjteta=1}
\left[m_j + A_2 \dfrac{2 \pi}{q_j}, M_j - A_2 \dfrac{2 \pi}{q_j}\right]
    \subset J_j^1 \,.
\end{equation}
So now, according to the value of $\kappa_\star$, we check whether the dynamical condition~\eqref{eq:dynamicalcondition} of Theorem~\ref{thm:main} is satisfied, using the Area formula.

\medskip
\paragraph*{Step 6 \--- Geometric condition of observability for rationally independent characteristic frequencies.}
%%%%%%%%%%%%%%%%%%%%%%%%%%%%%%%%%%%%%%%%%%%%%%%%%%%%%%%%%%%%%%%
Take $\eta = 1$, that is we consider a whole period of the rational Hamiltonian flow. We first establish a lower bound on the time spent by $t \mapsto x^t$ in $\omega(\tilde I)$. We consider $\kappa_\star > 0$ here. From Lemma~\ref{lem:criticalpointsisolated}, we know that $\gamma : (0, T_j) \ni t \mapsto \abs{x^t}$ is Lipschitz with constant $\sqrt{2 p(\rho_0)}$. Yet, similarly to~\eqref{eq:estimatetilderbyp}, it holds
\begin{equation*}
p(\rho_0)
	\le \norm{A} \tilde M_j^2 \,,
		\qquad \rm{where} \;\; \tilde M_j = \max_{t \in [0, T_j]} \abs{x^t} \,,
\end{equation*}
so that $\gamma$ is Lipschitz with constant $\sqrt{2 \norm{A}} \tilde M_j$. Applying the Area formula (Proposition~\ref{prop:areaformula}), we obtain as in~\eqref{eq:lowerboundareaformula} the lower bound:
\begin{equation*}
\int_0^{T_j} \one_{\omega(\tilde I)}(x^t) \dd t
	\ge (2 \norm{A})^{-1/2} \dfrac{1}{\tilde M_j} \int_{J_j^1} \one_{\tilde I}(s) \dd s \,,
\end{equation*}
and in view of~\eqref{eq:imageofcurvextxjteta=1} and~\eqref{eq:imageofcurvextxjteta} with $\eta = 1$, we deduce that
\begin{equation*}
\int_0^{T_j} \one_{\omega(\tilde I)}(x^t) \dd t
	\ge \dfrac{(2 \norm{A})^{-1/2}}{M_j + A_2 \frac{2 \pi}{q_j}} \int_{m_j + A_2 \frac{2 \pi}{q_j}}^{M_j - A_2 \frac{2 \pi}{q_j}} \one_{\tilde I}(s) \dd s \,.
\end{equation*}
Observing that $A_2 \le M_j$ and that $m_j \le \Lambda_j M_j$, we obtain
\begin{equation} \label{eq:A_2lessthanM_j}
\int_0^{T_j} \one_{\omega(\tilde I)}(x^t) \dd t
	\ge \dfrac{(2 \norm{A})^{-1/2}}{M_j (1 + \frac{2 \pi}{q_j})} \int_{M_j (\Lambda_j + \frac{2 \pi}{q_j})}^{M_j(1 - \frac{2 \pi}{q_j})} \one_{\tilde I}(s) \dd s \,.
\end{equation}
Setting
\begin{equation} \label{eq:defofrinfunctionofM_j}
r = M_j \left(1 - \dfrac{2 \pi}{q_j}\right)
    \qquad \rm{and} \qquad
\tilde \Lambda_j = \dfrac{\Lambda_j + \frac{2 \pi}{q_j}}{1 - \frac{2 \pi}{q_j}} \,,
\end{equation}
we can write the lower bound in~\eqref{eq:A_2lessthanM_j} under the form
\begin{equation} \label{eq:estimatevalidforallcurves}
\int_0^{T_j} \one_{\omega(\tilde I)}(x^t) \dd t
	\ge (2 \norm{A})^{-1/2} \dfrac{1 - \frac{2 \pi}{q_j}}{1 + \frac{2 \pi}{q_j}} \times \dfrac{1}{r} \int_{\tilde \Lambda_j r}^{r} \one_{\tilde I}(s) \dd s \,.
\end{equation}
We assume that $q_j > 2 \pi$ so that $r > 0$, which is the case for $j$ large enough since $q_j \to \infty$. The above estimate~\eqref{eq:estimatevalidforallcurves} is valid for any trajectory of the (irrational) Hamiltonian flow with initial datum $\rho_0 \neq 0$. In addition, we remark that $M_j$ defined in~\eqref{eq:defM_jm_j} tends to infinity as $\rho_0 \to \infty$, so that $r$ defined in~\eqref{eq:defofrinfunctionofM_j} tends to $+\infty$ as $\rho_0 \to \infty$ too. Thus~\eqref{eq:estimatevalidforallcurves} leads to
\begin{equation} \label{eq:dynconditionokirrational}
\liminf_{\rho \to \infty} \int_0^{T_j} \one_{\omega(\tilde I) \times \R^d}\left(\phi^t(\rho)\right) \dd t
	\ge (2 \norm{A})^{-1/2} \dfrac{1 - \frac{2 \pi}{q_j}}{1 + \frac{2 \pi}{q_j}} \times \liminf_{r \to + \infty} \dfrac{1}{r} \abs*{\tilde I \cap [\tilde \Lambda_j r, r]} \,.
\end{equation}
In order to deduce a positive lower bound, it suffices that $\tilde \Lambda_j < \kappa_\star = \kappa_\star(\tilde I) = \kappa_\star(I)$. This is achieved provided $q_j \ge 6 \pi / \kappa_\star \ge 6 \pi$. Indeed, under this condition, we have on the one hand
\begin{equation} \label{eq:dynconditionokirrational1}
\dfrac{1 - \frac{2 \pi}{q_j}}{1 + \frac{2 \pi}{q_j}}
    \ge \dfrac{1 - \frac{2 \pi}{6 \pi}}{1 + \frac{2 \pi}{6 \pi}}
    = \dfrac{1}{2} \,,
\end{equation}
and on the other hand, recalling the definition of $\tilde \Lambda_j$ in~\eqref{eq:defofrinfunctionofM_j}, the formula for $\Lambda_j$~\eqref{eq:deffunctionLambda} shown in Step 2, and using that $\sin x \le x$ and $\tan x \le \frac{4}{\pi} x$ for $x \in [0, \frac{\pi}{4}]$, we obtain
\begin{equation} \label{eq:dynconditionokirrational2}
\tilde \Lambda_j
    \le \dfrac{\frac{4}{\pi} \times \frac{\pi/2}{p_j + q_j} + \frac{2 \pi}{q_j}}{1 - \frac{2 \pi}{6 \pi}}
    \le 3 \dfrac{1 + \pi}{q_j}
    \le \dfrac{1}{2} \left(\dfrac{1}{\pi} + 1\right) \kappa_\star
    < \kappa_\star \,.
\end{equation}

\medskip

Now we turn to the upper bound on the time spent by projected trajectories of the (irrational) Hamiltonian flow in $\omega(I)_R$, for a fixed $R > 0$. We consider $\kappa_\star \in [0, 1]$ arbitrary now, with the convention $1/\kappa_\star = + \infty$ if $\kappa_\star = 0$. We go back to $\eta \in (0, 1]$. We select a curve $t \mapsto (x_j^t, \xi_j^t)$ of the rational flow that maximizes the ratio $\min_t \abs{x_j^t} / \max_t \abs{x_j^t}$, namely that satisfies
\begin{equation*}
m_j
	= \min_{t \in [0, T_j]} \abs{x_j^t}
    = \Lambda_j \max_{t \in [ 0, T_j]} \abs{x_j^t}
    = \Lambda_j M_j \,.
\end{equation*}
This curve is of the form~\eqref{eq:approximatingtrajectory} for well-chosen action and angle variables. We consider $t \mapsto (x^t, \xi^t)$ the corresponding trajectory of the irrational flow given by~\eqref{eq:initialtrajectory}, that is the integral curve obtained by substituting $\nu_2$ for $\frac{p_j}{q_j} \nu_1$ in~\eqref{eq:approximatingtrajectory}. Notice that this trajectory depends on $j$. We still write $\gamma(t) = \abs{x^t}$. By Lemma~\ref{lem:criticalpointsisolated}, we know that it is a Lipschitz map and that there exists an integer $k_0$ such that $\# \gamma^{-1}(s) \cap [0, T_j] \le k_0$ for all $s \in \R_+$. Reproducing the computation~\eqref{eq:areaformula1}, we find:
\begin{equation*}
\int_{B_N} \one_{I_R}\left(c \abs*{x^t}\right) \dd t
    \le \dfrac{2^N k_0}{c} \int_{c \Lambda_j \max \gamma}^{c \max \gamma} \one_{I_R}(s) \dd s
    \le \dfrac{2^N k_0}{c} \int_{c J_j^\eta} \one_{I_R}(s) \dd s \,,
\end{equation*}
where we recall that the parameter $c > 0$ is an arbitrary scaling factor, and $B_N$ is defined similarly to~\eqref{eq:monotoneconvergenceBN} by
\begin{equation*}
B_N = \set{t \in [0, \eta T_j]}{\abs*{\gamma'(t)} \ge 2^{-N}} \,.
\end{equation*}
In view of~\eqref{eq:imageofcurvextxjteta}, this leads to
\begin{equation*}
\int_{B_N} \one_{I_R}\left(c \abs*{x^t}\right) \dd t
    \le \dfrac{2^N k_0}{c} \int_{c (m_j - A_2 \eta \frac{2 \pi}{q_j})}^{c (M_j + A_2 \eta \frac{2 \pi}{q_j})} \one_{I_R}(s) \dd s \,.
\end{equation*}
As we did before in~\eqref{eq:A_2lessthanM_j}, we use the fact that $A_2 \le M_j$, together with $m_j = \Lambda_j M_j$ (the equality is important here) to obtain
\begin{equation*}
\int_{B_N} \one_{I_R}\left(c \abs*{x^t}\right) \dd t
    \le \dfrac{2^N k_0}{c} \int_{c M_j (\Lambda_j - \eta \frac{2 \pi}{q_j})}^{c M_j (1 + \eta \frac{2 \pi}{q_j})} \one_{I_R}(s) \dd s \,.
\end{equation*}
Defining now
\begin{equation*}
r = M_j \left(1 + \eta \dfrac{2 \pi}{q_j}\right)
    \qquad \rm{and} \qquad
\tilde \Lambda_j = \dfrac{\Lambda_j - \eta \frac{2 \pi}{q_j}}{1 + \eta \frac{2 \pi}{q_j}} \,,
\end{equation*}
we end up with
\begin{equation*}
\int_{B_N} \one_{I_R}\left(c \abs*{x^t}\right) \dd t
    \le 2^N k_0 M_j \left(1 + \eta \dfrac{2 \pi}{q_j}\right) \dfrac{1}{c r} \int_{\tilde \Lambda_j c r}^{c r} \one_{I_R}(s) \dd s \,.
\end{equation*}
We finally prove that this upper bound tends to zero along a well-chosen sequence of parameters $c$ provided $\tilde \Lambda_j \ge \kappa_\star$. This is fulfilled whenever $q_j \le \delta / \kappa_\star$, for a small enough constant $\delta$. To see this, we can use that $\tan x \ge x$ and $\sin x \ge \frac{2}{\pi} x$ on $[0, \frac{\pi}{2}]$ to control $\Lambda_j$ from below by $\frac{1}{p_j + q_j}$. Then~\eqref{eq:diophantineapproximation} leads to $\frac{p_j}{q_j} \le \frac{\nu_2}{\nu_1} + 1$, which yields
\begin{equation*}
\Lambda_j
	\ge \dfrac{1}{p_j + q_j}
    \ge \dfrac{1}{q_j (2 + \frac{\nu_2}{\nu_1})} =: \dfrac{C}{q_j} \,.
\end{equation*}
Assuming that $\eta < \frac{C}{2 \pi} \le 1$, we obtain
\begin{equation*}
\tilde \Lambda_j
    \ge \dfrac{\frac{C - 2 \pi \eta}{q_j}}{1 + \eta \frac{2 \pi}{q_j}}
    \ge \dfrac{1}{q_j} \times \dfrac{C - 2 \pi \eta}{1 + 2 \pi \eta}
    \ge \dfrac{C - 2 \pi \eta}{\delta (1 + 2 \pi \eta)} \kappa_\star \,.
\end{equation*}
This yields $\tilde \Lambda_j \ge \kappa_\star$ if $\delta$ is small enough, so that by definition of $\kappa_\star$, letting $c \to + \infty$, we obtain
\begin{align} \label{eq:dyncondnotokirrational}
\liminf_{c \to + \infty} \int_0^{\eta T_j} \one_{\omega(I_R)}\left(c x^t\right) \dd t
    &\le \abs*{[0, \eta T_j] \setminus B_N} + 2^N k_0 M_j \left(1 + \eta \dfrac{2 \pi}{q_j}\right) \liminf_{c \to + \infty} \dfrac{1}{c r} \int_{\tilde \Lambda_j c r}^{c r} \one_{I_R}(s) \dd s \nonumber\\
    &= \eta T_j - \abs*{B_N} \,,
\end{align}
which tends to zero as $N \to \infty$.

The general conclusion is the following: if $\kappa_\star > 0$ and $j \in \N$ is such that $q_j \ge \frac{6 \pi}{\kappa_\star}$, we know by~\eqref{eq:dynconditionokirrational2} that $\tilde \Lambda_j < \kappa_\star$, so that by definition of $\kappa_\star$, the estimate~\eqref{eq:dynconditionokirrational}, together with~\eqref{eq:dynconditionokirrational1}, proves that the dynamical condition~\eqref{eq:dynamicalcondition} of Theorem~\ref{thm:main} holds for $\omega(\tilde I)$ in time $T_j = \frac{2 \pi}{\nu_1} q_j$. If on the contrary $\kappa_\star \in [0, 1]$ and it holds $q_j \le \frac{\delta}{\kappa_\star}$ for some $\delta > 0$ depending only on $\frac{\nu_2}{\nu_1}$, then from~\eqref{eq:dyncondnotokirrational}, the dynamical condition~\eqref{eq:dynamicalcondition} is violated in $\omega(I)_R$ for any $R > 0$ on the time interval $[0, \eta T_j]$, where $\eta > 0$ depends only on $\frac{\nu_2}{\nu_1}$ again. Theorem~\ref{thm:main} then implies that the Schrödinger equation is observable from $\omega(I)$ if and only if $\kappa_\star > 0$. If indeed $\kappa_\star > 0$, then the optimal observation time $T_\star = T_\star(\omega(I))$ is controlled as follows: there exist constants $C, c > 0$ such that
\begin{equation} \label{eq:boundsonTstar}
c q_{j_1} \le T_\star \le C q_{j_2} \,,
\end{equation}
where $j_1$ is the largest index such that $q_j \le \frac{\delta}{\kappa_\star}$ and $j_2$ is the smallest index such that $q_j \ge \frac{6 \pi}{\kappa_\star}$.

To go from~\eqref{eq:boundsonTstar} to the desired estimate~\eqref{eq:optimaltimeinDiophantinecase} in the case where $\frac{\nu_2}{\nu_1}$ is Diophantine, we use the fact that the irrationality exponent $\tau$, defined in~\eqref{eq:defirrationalityexponent}, is related to the growth of the $q_j$'s. This comes from the formula
\begin{equation*}
\tau(\mu)
	= 1 + \limsup_{j \to \infty} \dfrac{\log q_{j+1}}{\log q_j}
\end{equation*}
(see~\cite[Proposition 1.8]{Durand:mnt} or~\cite[Theorem 1]{Sondow:04}). When $\tau$ is finite, we deduce in particular that for any $\eps > 0$, we have for any $j$ large enough
\begin{equation*}
\dfrac{\log q_{j + 1}}{\log q_j} \le \tau - 1 + \eps \,,
\end{equation*}
which leads to the existence of a constant $C_\eps > 0$ such that
\begin{equation*}
q_{j + 1} \le C_\eps {q_j}^{\tau - 1 + \eps} \,, \qquad \forall j \in \N \,.
\end{equation*}
We obtain by definition of the indices $j_1$ and $j_2$:
\begin{equation*}
\dfrac{\delta}{\kappa_\star} \le q_{j_1 + 1} \le C_\eps {q_{j_1}}^{\tau - 1 + \eps}
    \qquad \rm{and} \qquad
q_{j_2} \le C_\eps {q_{j_2 - 1}}^{\tau - 1 + \eps} \le C_\eps \left(\dfrac{6 \pi}{\kappa_\star}\right)^{\tau - 1 + \eps} \,.
\end{equation*}
Plugging this into~\eqref{eq:boundsonTstar}, we finally deduce~\eqref{eq:optimaltimeinDiophantinecase}. This concludes the proof of Proposition~\ref{prop:anisotropicsphericalsets}.

%%%%%%%%%%%%%%%%%%%%%%%%%%
\appendix %%%%%%%%%%%%%%%%%
%%%%%%%%%%%%%%%%%%%%%%%%%%

\section{Reduction to a weaker observability inequality} \label{app:reduction}
%%%%%%%%%%%%%%%%%%%%%%%%%%%%%%%%%%%%%%%%%%%%%%%%%%%%%%%%

The following proposition shows that $\Obs(\omega, T)$ is equivalent to a similar inequality with a remainder involving a compact operator. The argument goes back to Bardos, Lebeau and Rauch~\cite{BLR:92}. This reformulation of the problem paves the way for the use of microlocal analysis: we are interested in the propagation of high-energy modes through the Schrödinger evolution, discarding anything that is microlocalized near a fixed energy sub-level $\{p \le \rm{cst}\}$. An alternative route could be to slice the phase space according to energy layers of the Hamiltonian $p(x, \xi) = V(x) + \frac{1}{2} \abs{\xi}^2$; see~\cite{Leb:92,BZ:12,AM:14}.

\begin{prop} \label{prop:reduction}
Suppose $P$ is a self-adjoint operator with compact resolvent, and let $B$ be a bounded operator on $L^2(\R^d)$ satisfying the unique continuation property:
\begin{equation} \label{eq:uniquecontinuationP}
\textrm{for any eigenfunction $u$ of $P$,} \qquad B u = 0 \; \Longrightarrow u = 0 \,.
\end{equation}
Let $T_0 > 0$ and assume there exists a compact self-adjoint operator $K$ such that
\begin{equation} \label{eq:obsK}
\exists C_0 > 0 : \forall u \in L^2(\R^d) \,, \qquad
    \norm*{u}_{L^2}^2 \le C_0 \int_0^{T_0} \norm*{B \e^{- \ii t P} u}_{L^2}^2 \dd t + \inp*{u}{K u}_{L^2} \,.
\end{equation}
Then for every $T > T_0$, there exists $C > 0$ such that
\begin{equation*}
\forall u \in L^2(\R^d) \,, \qquad
    \norm*{u}_{L^2}^2 \le C \int_0^T \norm*{B \e^{- \ii t P} u}_{L^2}^2 \dd t \,.
\end{equation*}
\end{prop}

\begin{rema} \label{rmk:uniquecontinuation}
The operators of the form $P = V(x) - \tfrac{1}{2} \lap$ that we consider, with $V$ subject to Assumption~\ref{assum:V}, satisfy the unique continuation property of the statement when $B$ is the multiplication by the indicator function of a non-empty open set. See~\cite[Theorem 5.2]{LRLRbook:22}.
\end{rema}

\begin{proof}
Let us introduce for any $S \in \R$:
\begin{equation*}
A_S = \int_0^S \e^{\ii t P} B^\ast B \e^{- \ii t P} \dd t \,,
\end{equation*}
and denote by $\cal{I}_S$ its kernel (the space of so-called invisible solutions). One can check that
\begin{equation*}
\cal{I}_S = \bigcap_{t \in [0, S]} \ker B \e^{- \ii t P} = \set{u \in L^2(\R^d)}{\forall t \in [0, S] \,, \; B \e^{- \ii t P} u = 0} \,,
\end{equation*}
using the fact that $\e^{\ii t P} B^\ast B \e^{- \ii t P} \ge 0$ for all $t \in \R$ as operators, and that the map $t \mapsto B \e^{- \ii t P}$ is strongly continuous.
The space $\cal{I}_S$ is a closed linear subspace of $L^2(\R^d)$, both for the strong and the weak topology (use for instance that $A_S$ is a bounded operator). Moreover, one has the property that $S_1 \le S_2$ yields $\cal{I}_{S_1} \supset \cal{I}_{S_2}$. It implies that for any $S$, the set
\begin{equation*}
\cal{I}_S^- = \bigcup_{S' > S} \cal{I}_{S'}
\end{equation*}
is also a linear subspace, contained in $\cal{I}_S$.

\medskip
\emph{Step 1 \--- $\cal{I}_{T_0}$ is finite-dimensional.}
This assertion is a consequence of the fact that $K$ is coercive on $\cal{I}_{T_0}$, namely
\begin{equation*}
\forall u \in \cal{I}_{T_0} \,, \qquad
    \norm*{u}_{L^2} \le \norm*{K u}_{L^2} \,,
\end{equation*}
It follows directly from the assumption~\eqref{eq:obsK} and the Cauchy-Schwarz inequality. Setting $W = \ran K_{\vert \cal{I}_{T_0}}$, we deduce that $K : \cal{I}_{T_0} \to W$ is one-to-one and its inverse $K^{-1}$ is bounded as an operator in $\cal{L}(W, \cal{I}_{T_0})$. Now denote by $\bar B_{\cal{I}_{T_0}}$ the closed unit ball of $\cal{I}_{T_0}$. Since $\cal{I}_{T_0}$ is strongly and weakly closed, the same holds for its closed unit ball as a subset of $L^2(\R^d)$. We deduce that $\bar B_{\cal{I}_{T_0}}$ is weakly compact. The compactness of $K$ implies that $K(\bar B_{\cal{I}_{T_0}})$ is (strongly) compact in $L^2(\R^d)$. Since it is contained in $W$, it is compact in $W$. Therefore the fact that $K^{-1} : W \to \cal{I}_{T_0}$ is bounded implies that $\bar B_{\cal{I}_{T_0}} = K^{-1}(K(\bar B_{\cal{I}_{T_0}}))$ is compact. We deduce by the Riesz's Theorem that $\cal{I}_{T_0}$ is finite-dimensional.

\medskip
\emph{Step 2 \--- $\cal{I}_{T_0}^-$ is stable by $P$.}
Let us check that $\cal{I}_{T_0}^- \subset \dom P$. Let $u \in \cal{I}_{T_0}^-$ and set
\begin{equation*}
u_\epsilon = \dfrac{\e^{- \ii \epsilon P} u - u}{\epsilon} \,, \qquad \forall \epsilon \neq 0 \,.
\end{equation*}
By definition of $\cal{I}_{T_0}^-$, the function $u$ belongs to $\cal{I}_{T_0 + \epsilon_0}$ for some $\epsilon_0 > 0$, so that $u_\epsilon \in \cal{I}_{T_0}^-$ for any $\epsilon \in (0, \epsilon_0)$. Recall from the previous step that $\cal{I}_{T_0} \supset \cal{I}_{T_0}^-$ is finite dimensional. We observe that $v \mapsto \norm*{(P - \ii)^{-1} v}_{L^2}$ is a norm on $\cal{I}_{T_0}$, so it is equivalent to the $L^2$ norm. Yet we see that
\begin{equation*}
(P - \ii)^{-1} u_\epsilon
    = \dfrac{\e^{- \ii \epsilon P} (P - \ii)^{-1} u - (P - \ii)^{-1} u}{\epsilon} \,,
\end{equation*}
with $(P - \ii)^{-1} u \in \dom P$, so that $(P - \ii)^{-1} u_\epsilon$ converges as $\epsilon \to 0$ to some $v \in \dom P$. Therefore we conclude that
\begin{equation*}
\norm*{u_\epsilon - (P - \ii) v}_{L^2}
    \le C \norm*{(P - \ii)^{-1} u_\epsilon - v}_{L^2}
    \strongto{\epsilon \to 0} 0 \,.
\end{equation*}
The fact that $u_\epsilon$ converges shows that $u \in \dom P$, hence $\cal{I}_{T_0}^- \subset \dom P$. It remains to see that $\lim_{\epsilon \to 0} u_\epsilon = - \ii P u$ belongs to $\cal{I}_{T_0}^-$, which is a consequence of the fact that $\cal{I}_{T_0}^-$ is finite-dimensional, hence closed.

\medskip
\emph{Step 3 \--- $\cal{I}_{T_0}^- = \{0\}$.}
This results from the unique continuation property~\eqref{eq:uniquecontinuationP}. Indeed, we can argue as follows: from the previous steps, $\cal{I}_{T_0}^-$ is a finite-dimensional linear subspace of $L^2(\R^d)$ which is stable by the self-adjoint operator $P$. Therefore there exists a basis $(u_1, u_2, \ldots, u_n)$ of $\cal{I}_{T_0}^-$ made of eigenvectors of $P$. By definition of $\cal{I}_S$, these eigenvectors satisfy in particular $B u_j = 0$. So by the unique continuation result~\eqref{eq:uniquecontinuationP}, we find that $\cal{I}_{T_0}^-$ must be trivial.

\medskip
\emph{Step 4 \--- Conclusion.}
Let $T > T_0$. We want to show that $A_T \ge c$ for some $c > 0$. To do this, it suffices to prove that $A_T$ is invertible, because $A_T$ is self-adjoint and $A_T \ge 0$. The assumption~\eqref{eq:obsK} implies that the self-adjoint operator $A_T + K$ is invertible, meaning that zero does not belong to its spectrum. Since $K$ is compact and self-adjoint, we classically know that $A_T$ has the same essential spectrum as $A_T + K$, so in particular zero is not in the essential spectrum of $A_T$. It is not an eigenvalue neither since $\ker A_T \subset \cal{I}_{T_0}^- = \{0\}$. Therefore $A_T$ is invertible, and the conclusion follows.
\end{proof}

The following lemma is not related to the previous proposition. Still, it is worth stating it properly since we use it on several occasions throughout the article.

\begin{lemm} \label{lem:obsopenintime}
Let $\omega \subset \R^d$ be measurable. Assume $\Obs(\omega, T)$ holds in some time $T > 0$ with a cost $C > 0$, namely
\begin{equation*}
\forall u \in L^2(\R^d) \,, \qquad
	\norm*{u}_{L^2(\R^d)}^2
		\le C \int_0^T \norm*{\e^{- \ii t P} u}_{L^2(\omega)}^2 \dd t \,.
\end{equation*}
Then $\Obs(\omega, T - \eps)$ holds for any $\eps < 1/C$.
\end{lemm}

\begin{proof}
We use the fact that the propagator $\e^{- \ii t P}$ is an isometry on $L^2(\R^d)$ to get
\begin{equation*}
C \int_{T - \eps}^T \norm*{\e^{- \ii t P} u}_{L^2(\omega)}^2 \dd t
	\le C \int_{T - \eps}^T \norm*{\e^{- \ii t P} u}_{L^2(\R^d)}^2 \dd t
	= C \eps \norm*{u}_{L^2(\R^d)}^2 \,.
\end{equation*}
Thus we can absorb this term in the left-hand side of the observability inequality provided $C \eps < 1$:
\begin{equation*}
(1 - C \eps) \norm*{u}_{L^2(\R^d)}^2
	\le C \int_0^{T - \eps} \norm*{\e^{- \ii t P} u}_{L^2(\omega)}^2 \dd t \,,
\end{equation*}
namely $\Obs(\omega, T - \eps)$ holds with cost $C (1 - C \eps)^{-1}$.
\end{proof}

\section{Pseudodifferential operators} \label{app:pseudo}
%%%%%%%%%%%%%%%%%%%%%%%%%%%%%%%%%%%%%

We recall below basics of the theory of pseudodifferential operators (see the textbooks~\cite{Hormander:V3,Lerner:10,MartinezBook, Zworski:book} for further details). We will also need a precise bound on the remainder of the pseudodifferential calculus and of the sharp G{\aa}rding inequality. This is why we reproduce the proofs of these results below.

\subsection{Weyl quantization}

Let $a \in \sch(\R^{2d})$. We define the operator $\Op{a}$ acting on the Schwartz class $\sch(\R^d)$ by
\begin{equation*}
\left[\Op{a} u\right](x)
	= (2 \pi)^{-d} \int_{\R^{2d}} \e^{\ii (x - y) \cdot \xi} a\left(\dfrac{x + y}{2}, \xi\right) u(y) \dd y \dd \xi \,, \; u \in \sch(\R^d), x \in \R^d \,.
\end{equation*}
It is known that $\Op{a} : \sch(\R^d) \to \sch(\R^d)$ is continuous. The quantization $\quantization$ extends to tempered distributions: for any $a \in \sch'(\R^{2d})$, the operator $\Op{a} : \sch(\R^d) \to \sch'(\R^d)$ is continuous.

\subsection{Symbol classes}

\begin{defi}[Symbol classes]
Let $f$ be an order function.\footnote{A positive function $f$ on the phase space is said to be an order function if
\begin{equation*}
\exists C > 0, \exists N > 0 : \forall \rho, \rho_0 \in \R^{2d} \,, \qquad
	f(\rho) \le C \jap{\rho - \rho_0}^N f(\rho_0) \,.
\end{equation*}} Then the symbol class $S(f)$ is the set of functions $a \in \cont^\infty(\R^{2d})$ satisfying
\begin{equation*}
\forall \alpha \in \N^{2d}, \exists C_\alpha > 0 : \forall \rho \in \R^{2d} \,, \qquad
    \abs{\partial^\alpha a(\rho)} \le C_\alpha f(\rho) \,.
\end{equation*}
Collecting the best constants $C_\alpha$ for each $\alpha$, the quantities
\begin{equation*}
\abs{a}_{S(f)}^\ell = \max_{\abs{\alpha} \le \ell} C_\alpha \,, \qquad \ell \in \N \,,
\end{equation*}
are seminorms that turn the vector space $S(f)$ into a Fréchet space.
\end{defi}

Any $a \in S(f)$ is a tempered distribution and yields a continuous linear operator $\Op{a} : \sch(\R^d) \to \sch(\R^d)$.

\subsection{\texorpdfstring{$L^2$}{L^2}-boundedness of pseudodifferential operators}

\begin{theo}[\texorpdfstring{Calder\'{o}n-Vaillancourt}{Calderon-Vaillancourt}] \label{thm:CV}
There exist constants $C_d, k_d > 0$ depending only on the dimension $d$ such that the following holds: for any $a \in S(1)$, the operator $\Op{a}$ can be extended to a bounded operator on $L^2(\R^d)$ with the bound
\begin{equation*}
\norm*{\Op{a}}_{L^2 \to L^2}
    \le C_d \abs{a}_{S(1)}^{k_d} \,.
\end{equation*}
\end{theo}

\subsection{Refined estimate in the pseudodifferential calculus}

Let $a_1, a_2$ be two symbols. We have seen previously that the composition $\Op{a_1} \Op{a_2}$ makes sense as an operator on the Schwartz space. This operator is also a pseudodifferential operator, whose symbol is denoted by $a_1 \moyal a_2$, called the Moyal product of $a_1$ and $a_2$, and satisfies
\begin{equation} \label{eq:defmoyalproduct}
\Op{a_1} \Op{a_2} = \Op{a_1 \moyal a_2} \,.
\end{equation}
More generally, one can define the $h$-Moyal product, depending on a parameter $h \in (0, 1]$, as
\begin{equation*}
\left(a_1 \moyal_h a_2\right)(\rho)
    = \e^{- \frac{\ii h}{2} \sympf(\partial_{\rho_1}, \partial_{\rho_2})} a_1(\rho_1) a_2(\rho_2)_{\vert \rho_1 = \rho_2 = \rho} \,,
\end{equation*}
where $\sympf$ is the canonical symplectic form on $\R^{2d}$. Taking $h = 1$, one gets a formula for the Moyal product in~\eqref{eq:defmoyalproduct} above. The $h$-Moyal product is known to be a bilinear continuous map between symbol classes; see~\cite[Theorem 4.17]{Zworski:book} or~\cite[Theorem 2.3.7]{Lerner:10} for instance.

\begin{prop}[Continuity of Moyal product] \label{prop:continuitymoyalproduct}
Let $f_1, f_2$ be two order functions. Then the map
\begin{align*}
S(f_1) \times S(f_2) &\to S(f_1 f_2) \\
(a_1, a_2) &\mapsto a_1 \moyal_h a_2
\end{align*}
is bilinear continuous, with constants independent of $h \in (0, 1]$. More precisely, for any $\ell \in \N$, there exist $k \in \N$ and $C_\ell > 0$ such that
\begin{equation*}
\abs*{a_1 \moyal_h a_2}_{S(f_1 f_2)}^\ell
    \le C_\ell \abs*{a_1}_{S(f_1)}^k \abs*{a_2}_{S(f_2)}^k \,, \qquad \forall h \in (0, 1], \forall (a_1, a_2) \in S(f_1) \times S(f_2) \,.
\end{equation*}
\end{prop}

A stationary phase argument leads to an asymptotic expansion of the Moyal product
\begin{equation*}
a_1 \moyal a_2
    \sim \sum_j \dfrac{(- \ii/2)^j}{j!} \sympf(\partial_{\rho_1}, \partial_{\rho_2})^j a_1(\rho_1) a_2(\rho_2)_{\vert \rho = \rho_1 = \rho_2} \,.
\end{equation*}
In the sequel, we denote by $\cal{R}_{j_0}(a_1, a_2)$ the remainder of order $j_0$ in this asymptotic expansion, namely
\begin{equation*}
\cal{R}_{j_0}(a_1, a_2)(\rho)
    = a_1 \moyal a_2 - \sum_{j = 0}^{j_0 - 1} \dfrac{(- \ii/2)^j}{j!} \sympf(\partial_{\rho_1}, \partial_{\rho_2})^j a_1(\rho_1) a_2(\rho_2)_{\vert \rho_1 = \rho_2 = \rho} \,.
\end{equation*}

Estimates on this remainder term are usually stated as follows.

\begin{prop}[Pseudodifferential calculus]
Let $f_1, f_2$ be two order functions. Then for any integer $j_0 \ge 1$, the map
\begin{align*}
S(f_1) \times S(f_2) &\to S(f_1 f_2) \\
(a_1, a_2) &\mapsto \cal{R}_{j_0}(a_1, a_2)
\end{align*}
is bilinear continuous.
\end{prop}

In our study, it will be convenient to have a slightly more precise statement. Actually, the explicit formula for the remainder allows to prove that its seminorms are controlled not only by the seminorms of $a_1$ and $a_2$ but more precisely by the seminorms of the derivatives $\dd^{j_0} a_1$ and $\dd^{j_0} a_2$.

\begin{prop}[Refined estimate] \label{prop:refinedestimateremainderpseudocalc}
Let $f_1, f_2$ be two order functions. Then for any $j_0 \ge 1$, it holds
\begin{equation*}
\forall \ell \in \N, \exists k \in \N, \exists C_\ell > 0 : \qquad
	\abs*{\cal{R}_{j_0}(a_1, a_2)}_{S(f_1 f_2)}^\ell
    	\le C_\ell \abs*{\dd^{j_0} a_1}_{S(f_1)}^k \abs*{\dd^{j_0} a_2}_{S(f_2)}^k \,,
\end{equation*}
for all $(a_1, a_2) \in S(f_1) \times S(f_2)$.
\end{prop}

\begin{proof}
We outline the arguments of the proof, which are classical, trying to keep track of constants carefully. The starting point of this result is the explicit expression of the remainder (see~\cite[Theorem 4.11]{Zworski:book} for instance):
\begin{equation*}
\cal{R}_{j_0}(a_1, a_2)(\rho)
    = \left(\dfrac{- \ii}{2}\right)^{j_0} \int_0^1 \dfrac{(1 - t)^{j_0 - 1}}{(j_0 - 1)!} \e^{-\ii \frac{t}{2} \sympf(\partial_{\rho_1}, \partial_{\rho_2})} \sympf(\partial_{\rho_1}, \partial_{\rho_2})^{j_0} a_1(\rho_1) a_2(\rho_2)_{\vert \rho_1 = \rho_2 = \rho} \dd t \,.
\end{equation*}
The binomial expansion of $\sympf(\partial_{\rho_1}, \partial_{\rho_2})^{j_0}$ exhibits a particular structure: we observe that the integrand of the integral over $t$ can be written as a sum of terms of the form
\begin{equation*}
\e^{-\ii \frac{t}{2} \sympf(\partial_{\rho_1}, \partial_{\rho_2})} (\partial^{\alpha_1} a_1)(\rho_1) (\partial^{\alpha_2} a_2)(\rho_2)_{\vert \rho_1 = \rho_2 = \rho}
\end{equation*}
with $\abs{\alpha_1} = \abs{\alpha_2} = j_0$, which corresponds exactly to $\partial^{\alpha_1} a_1 \moyal_t \partial^{\alpha_2} a_2$. By Proposition~\ref{prop:continuitymoyalproduct}, we know that the Moyal product is a bilinear continuous map $S(f_1) \times S(f_2) \to S(f_1 f_2)$ with respect to the Fréchet space topology, with seminorm estimates independent of $t \in (0, 1]$. This yields
\begin{equation} \label{eq:continuityRcalell=0}
\abs*{\cal{R}_{j_0}(a_1, a_2)}_{S(f_1 f_2)}^0
	\le C_0 \abs*{\dd^{j_0} a_1}_{S(f_1)}^k \abs*{\dd^{j_0} a_2}_{S(f_2)}^k \,.
\end{equation}
In order to handle seminorms of order $\ell \ge 0$, we use the Leibniz formula:
\begin{equation*}
\partial \cal{R}_{j_0}(a_1, a_2)
	= \cal{R}_{j_0}\left(\partial a_1, a_2\right) + \cal{R}_{j_0}\left(a_1, \partial a_2\right) \,,
\end{equation*}
and we apply~\eqref{eq:continuityRcalell=0}. The result follows.
\end{proof}

\subsection{Positivity}

Heuristically, the quantization of a non-negative symbol is an almost-non-negative operator. The formal statement, known as the G{\aa}rding inequality, says that the negative part of the operator is controlled in terms of the Planck parameter in semiclassical analysis, or exhibits some decay at infinity in the phase space in microlocal analysis. In the main part of the article, we need to apply the G{\aa}rding inequality to a symbol in $S(1)$ whose derivatives, of any order, behave like $1/R$, where $R$ is a large parameter. Unfortunately, such a symbol does not fit in the semiclassical framework, in which derivatives of order $j$ behave like $1/R^j$. Thus we provide in this paragraph a refined statement of the sharp G{\aa}rding inequality that keeps track of the dependence of the remainder term on the seminorms of the derivatives of the symbol.

\begin{prop}[Sharp G{\aa}rding inequality] \label{prop:Gaarding}
There exists a constant $c_d > 0$ and an integer $k_d \ge 0$ depending only on the dimension $d$ such that the following holds. For any real-valued symbol $a \in S(1)$ satisfying $a \ge 0$, one has
\begin{equation*}
    \Op{a} \ge -c_d \abs*{\hess a}_{S(1)}^{k_d} \id \,.
\end{equation*}
\end{prop}

\begin{proof}
We redo the usual proof (see for instance~\cite{Zworski:book}) using the refined estimate on the remainder in the pseudodifferential calculus (Proposition~\ref{prop:refinedestimateremainderpseudocalc}). Let us prove that for $z$ sufficiently negative, the operator $\Op{a - z}$ is invertible, which in turn shows that it is non-negative by classical arguments.

\emph{Step 1 \--- Estimate of the derivatives of $(a - z)^{-1}$.}
Using the assumption that $a \ge 0$, we classically have
\begin{equation} \label{eq:nablaa}
\abs{\nabla a(\rho)} \le \sqrt{2 \abs{\hess a}_\infty a(\rho)} \,, \qquad \forall \rho \in \R^{2d}
\end{equation}
(see~\cite[Lemma 4.31]{Zworski:book} for instance). Besides the Faà di Bruno Formula tells us that for any nonzero $\alpha \in \N^{2d}$, the partial derivative $\partial^\alpha (a - z)^{-1}$ can be computed as a sum of terms of the form
\begin{equation*}
\dfrac{1}{(a - z)^{1 + \ell}} \prod_{j = 1}^\ell \partial^{\alpha_j} a \,,
\end{equation*}
with $1 \le \ell \le \abs{\alpha}$, $\sum_{j = 1}^\ell \alpha_j = \alpha$, $\abs{\alpha_j} \neq 0, \forall j$. Denote by $\ell'$ the number indices $j$ such that $\abs{\alpha_j} = 1$. We apply~\eqref{eq:nablaa} to the $\ell'$ factors of the form $\partial^{\alpha_j} a$ corresponding to these indices, and we bound the $\ell - \ell'$ other ones by seminorms of the Hessian of $a$ (recall that $\abs{\alpha_j} \ge 2$ for those remaining indices). We obtain
\begin{equation*}
\abs*{\dfrac{1}{(a - z)^{1 + \ell}} \prod_{j = 1}^\ell \partial^{\alpha_j} a}
    \le \dfrac{1}{\abs{a - z}^{1 + \ell}} \left(2 \abs*{\hess a}_\infty a(\rho)\right)^{\ell'/2} \left(\abs*{\hess a}_{S(1)}^{\abs{\alpha}}\right)^{\ell - \ell'} \,.
\end{equation*}
We deduce that
\begin{align*}
\abs*{\dfrac{1}{(a - z)^{1 + \ell}} \prod_{j = 1}^\ell \partial^{\alpha_j} a}
    &\le \dfrac{1}{\abs{a - z}^{1 + \ell}} 2^{\frac{\ell'}{2}} \abs*{a(\rho)}^{\frac{\ell'}{2}} \left(\abs*{\hess a}_{S(1)}^{\abs{\alpha}}\right)^{\ell - \frac{\ell'}{2}} \\
    &\le \dfrac{1}{\abs{a - z}^{1 + \ell}} 2^{\ell'} \left(\abs*{a - z}^{\frac{\ell'}{2}} + \abs{z}^{\frac{\ell'}{2}}\right) \left(\abs*{\hess a}_{S(1)}^{\abs{\alpha}}\right)^{\ell - \frac{\ell'}{2}} \,.
\end{align*}
Putting together all the terms in the Faà di Bruno Formula, and using that $a - z \ge \abs{z}$ (since $z \le 0$), we finally get that there exists a constant $C > 0$ (depending on $\abs{\alpha}$) such that
\begin{equation*}
\abs*{\partial^\alpha \dfrac{1}{a - z}}
    \le \dfrac{C}{\abs{z}} \max_{\substack{1 \le \ell \le \abs{\alpha} \\ 0 \le \ell' \le \ell}} \left( \dfrac{\abs{\hess a}_{S(1)}^{\abs{\alpha}}}{\abs{z}} \right)^{\ell - \frac{\ell'}{2}} \,.
\end{equation*}
Assuming that $\abs{z} \ge \abs{\hess a}_{S(1)}^{\abs{\alpha}}$, we arrive at
\begin{equation} \label{eq:seminorminversea-z}
\abs*{\partial^\alpha \dfrac{1}{a - z}}
    \le \dfrac{C}{\abs{z}} \sqrt{\dfrac{\abs{\hess a}_{S(1)}^{\abs{\alpha}}}{\abs{z}}} \,.
\end{equation}

\emph{Step 2 \---  Invertibility of $\Op{a - z}$.}
From the previous step, we know that $a - z$ and $(a - z)^{-1}$ are in $S(1)$ with explicit seminorm estimates, provided $\abs{z}$ is large enough. We perform the pseudodifferential calculus:
\begin{equation*}
\Op{a - z} \Op{\dfrac{1}{a - z}} = \id + 0 + \Op{\cal{R}_2} \,,
\end{equation*}
keeping in mind that the second term in the asymptotic expansion vanishes because both symbols are functions of the same symbol. According to the Calder\'{o}n-Vaillancourt Theorem (Theorem~\ref{thm:CV}), our refined estimate on the remainder (Proposition~\ref{prop:refinedestimateremainderpseudocalc}), and finally to~\eqref{eq:seminorminversea-z}, we obtain
\begin{equation*}
\norm*{\Op{\cal{R}_2}}_{L^2 \to L^2}
    \le C_d \abs*{\cal{R}_2}_{S(1)}^{k_d}
    \le C_d \abs*{\hess a}_{S(1)}^{k_1'} \abs*{\hess (a - z)^{-1}}_{S(1)}^{k_2'}
    \le C \left(\dfrac{\abs{\hess a}_{S(1)}^k}{\abs{z}}\right)^{3/2} \,,
\end{equation*}
for some constant $C$ and some integer $k$ independent of $a$ and $z$, and provided $z$ is negative enough. Actually when $z \le - (2 C)^{2/3} \abs{\hess a}_{S(1)}^k$, we obtain that $\norm{\Op{\cal{R}_2}} \le 1/2$, so that $\id + \Op{\cal{R}_2}$ is invertible by Neumann series. This leads classically to the invertibility of $\Op{a - z}$, which concludes the proof.
\end{proof}

\backmatter
\small
\bibliographystyle{plain}
\bibliography{biblio}

\end{document}